\documentclass[a4paper]{article}

\usepackage[utf8x]{inputenc}
\usepackage[T1]{fontenc}
\usepackage{xcolor}

\usepackage[a4paper,top=3cm,bottom=2cm,left=3cm,right=3cm,marginparwidth=1.75cm]{geometry}

\usepackage{amsfonts,amssymb,amsmath,amsthm}
\usepackage{graphicx}
\usepackage[colorinlistoftodos]{todonotes}
\usepackage[colorlinks=true, allcolors=blue]{hyperref}
\usepackage[square,sort,comma,numbers]{natbib}
\usepackage{cancel}
\usepackage{stmaryrd}
\usepackage{hyperref}

\usepackage{soul} 

\RequirePackage{fix-cm}

\newcommand{\R}{\mathbb{R}}
\newcommand{\Z}{\mathbb{Z}}
\newcommand{\N}{\mathbb{N}}
\newcommand{\T}{\mathbb{T}}
\newcommand{\D}{\mathcal{D}}
\newcommand{\E}{\mathcal{E}}
\newcommand{\II}{\mathcal{I}}
\newcommand{\LL}{\mathcal{L}}
\newcommand{\LLin}{\mathcal{L}_{\mathrm{in}}}

\newcommand{\OO}{\mathcal{O}}
\newcommand{\X}{\mathcal{X}}
\newcommand{\Y}{\mathcal{Y}}
\newcommand{\PP}{\mathcal{P}}
\newcommand{\RR}{\mathcal{R}}
\newcommand{\ZZ}{\mathcal{Z}}
\newcommand{\C}{\mathbb{C}}
\newcommand{\F}{\mathcal{F}}
\newcommand{\Fd}{\mathcal{F}_{\mathrm{in},d}}
\newcommand{\wtFd}{\widetilde{\mathcal{F}}_{\mathrm{in},d}}
\newcommand{\Fod}{\mathcal{F}_{\mathrm{out},d}}
\newcommand{\wtFod}{\widetilde{\mathcal{F}}_{\mathrm{out},d}}
\newcommand{\G}{\mathcal{G}}
\newcommand{\Gi}{\mathcal{G}_{\mathrm{in}}}
\newcommand{\Gf}{\mathcal{G}_{\mathrm{flow}}}
\newcommand{\Gid}{\mathcal{G}_{\mathrm{in},d}}
\newcommand{\God}{\mathcal{G}_{\mathrm{out},d}}

\newcommand{\NN}{\mathcal{N}}
\newcommand{\QQ}{\mathcal{Q}}
\newcommand{\wt}{\widetilde }
\newcommand{\lln}{\llfloor}
\newcommand{\rrn}{\rrfloor}
\newcommand{\B}{\mathcal{B}}

\newcommand{\LLi}{L^+_{\mathrm{in}}}
\newcommand{\LLic}{L_{\mathrm{in}}}
\newcommand{\LLipm}{L^\pm_{\mathrm{in}}}
\newcommand{\LLim}{L^-_{\mathrm{in}}}
\newcommand{\LLo}{L^+_{\mathrm{out}}}
\newcommand{\LLom}{L^-_{\mathrm{out}}}
\newcommand{\Fi}{\mathcal{F}_{\mathrm{in}}}
\newcommand{\Fext}{\mathcal{F}_{\mathrm{ext}}}
\newcommand{\whFi}{\wh{\mathcal{F}}_{\mathrm{in}}}
\newcommand{\Fii}{\widehat{\mathcal{F}}_{\mathrm{in},a}}
\newcommand{\wh}{\widehat}

\newcommand{\Lip}{\mathrm{Lip}\,}
\newcommand{\Deltainn}{\Delta_{\mathrm{in}}}
\newcommand{\Deltaout}{\Delta_{\mathrm{out}}}
\newcommand{\wtDeltainn}{{\wt \Delta}_{\mathrm{in}}}
\newcommand{\wtDeltaout}{{\wt \Delta}_{\mathrm{out}}}
\newcommand{\Id}{\mathrm{Id}}
\newcommand{\DD}[1]{D\!#1}
\newcommand{\M}{\mathcal{M}}
\newcommand{\A}{\mathcal{A}}
\newcommand{\Psiloc}{\Psi_{\mathrm{loc}}}
\newcommand{\wtPsiloc}{\wt \Psi_{\mathrm{loc}}}
\newcommand{\Psiglob}{\Psi_{\mathrm{glob}}}
\newcommand{\wtPsiglob}{\wt \Psi_{\mathrm{glob}}}
\newcommand{\fH}{{\rm \bf H1}}
\newcommand{\sH}{{\rm \bf H2}}
\newcommand{\rr}{\rho}
\newcommand{\nLi}{\Theta_V}

\newtheorem{lemma}{Lemma}[section]
\newtheorem{corollary}[lemma]{Corollary}

\newtheorem{remark}[lemma]{Remark}

\newtheorem{proposition}[lemma]{Proposition}
\newtheorem{theorem}[lemma]{Theorem}

\title{Chaotic scattering of He atoms off a Cu surface with corrugated Morse potential}
\author{}

\numberwithin{equation}{section}
\begin{document}
\author{Florentino Borondo\footnote{f.borondo@uam.es, {Departamento de Qu\'{i}mica,} 
Universidad Aut\'onoma de Madrid.}, Ernest Fontich\footnote{fontich@ub.edu, Departament de Matemàtiques i Informàtica, Universitat de Barcelona (UB), and Centre de Recerca Matemàtica (CRM).}, and Pau Mart\'{\i}n\footnote{p.martin@upc.edu, Departament de Matemàtiques, Universitat Polit\`ecnica de Catalunya (UPC), and Centre de Recerca Matemàtica (CRM)}.}





\maketitle

\begin{abstract}
We consider a Hamiltonian system that models the scattering of helium atoms off a copper surface. 
The interaction between the He and the Cu atoms is described by a corrugated Morse potential. 
Using corrugation coefficients values in the potential obtained by fitting to experimental values, 
we prove that, provided some coefficient of an auxiliary function is different from 0, there are regions of the phase space, corresponding to sufficiently large energy of the incident atom, where the scattering is chaotic. 
Furthermore, we prove that the system has oscillatory motions.
\end{abstract}


\tableofcontents

\section{Introduction}

\subsection{The He-Cu scattering problem. Main statement}

We consider the motion of a helium atom bouncing off a copper surface. 
The problem arises from experimental techniques, where the scattering of noble gas atoms after collisions with a surface 
is used to characterize surface structures in a non-destructive way \cite{Hulpke92}. 
The corrugation of the surface is modeled by the \emph{effective potential} seen by the atoms as they come near the surface. Corrugation depends on the incident energy: the higher the energy of the He atom, the closer will get to the surface. 
Numerical evidence of the existence of chaotic scattering was given in~\cite{Borondo97,Borondo1999}, 
where it was related to some invariant manifolds in the system.
The purpose of the present paper is to provide an actual rigorous proof of such chaotic behavior, thus establishing a link between experiments and theory on firm grounds.

We simplify the problem assuming that the He atom motion takes place on a plane; actually the
out-of-plane scattering is relatively small for the values of the energies and incident angles considered in the experiments.
We denote by $(x,z) \in \R^2$ the position of the He atom, where $x$ and $z$ are the horizontal and vertical displacements, respectively.
Let $(p_x,p_z)$ be the conjugate momenta. We model the interaction of the He atom with the copper surface by a corrugated Morse
potential, where the corrugation represents the presence of the copper atoms in the surface. More concretely, we consider the Hamiltonian
\begin{equation}
  \label{eq:H1}
  H_{\mathrm{CM}}(x,z,p_x,p_z) = \dfrac1{2m}(p_x^2+p_z^2)+V_\textrm{M}(z) + V_\mathrm{C}\left(\frac{2 \pi x}{a},z\right),
\end{equation}
where
\begin{equation}
\label{def:potencialdeMorsecorrugat}
\begin{aligned}
  V_\mathrm{M}(z) &= D e^{-\alpha z}(e^{-\alpha z}-2), \\
  V_\mathrm{C}(\theta,z ) &= De^{-2\alpha z} \, V(\theta),\\
  V(\theta) &= \sum_{n\geq 1} \left( r_n \cos (n \theta )+ s_n \sin ( n \theta)\right).
 \end{aligned}
\end{equation}
The coefficients $r_n$ and $s_n$ are determined experimentally. Their values are
\[
r_1= 0.06, \qquad r_2 = 0.008, \qquad
r_n=0, \quad n\geq 3,\qquad  s_n=0,\quad  n\geq 1,
\]
$D=6.35$ meV, $a=3.6$ \AA, $\alpha=1.05$ \AA$^{-1}$.
See~\cite{Borondo97}. 
In particular, $r_1 \neq 0$.   In the present paper, our only requirement on $V$ is that it is an analytic even function,
 that is, $s_n = 0$, for all~$n$. We emphasize that we do not require $V$ to be a trigonometric polynomial. Evenness is not an important requirement and the results will hold with the same techniques for non-even $V$, but the added symmetry will simplify some of the details, particularly the numerical computations performed in the problem that are carried out in~\cite{BarrabesFMO23}.  Since $\theta = 2\pi x/a$ appears in the equation through $V$, we will take $\theta\in \T = \R/2\pi\Z$.

The purpose of this paper is twofold. On the one hand, we want to prove the presence of chaos in some parts of the phase space of~\eqref{eq:H1}.
Here, the notion of chaos is the one introduced by Smale in~\cite{Smale65} and it is based on the presence of a \emph{Smale horseshoe} and a conjugation with the shift on a space of sequences of symbols. It is worth to remark that the set where chaos takes place is a hyperbolic set.
We recall the definition of symbolic dynamics with an infinite number of symbols, as introduced by Moser in~\cite{Moser01}. Let $S = \{ s= (\dots,s_{-1},s_0,s_1,\dots) \mid s_i \in \N\}$  be the space of two sided sequences of infinite symbols, with the topology induced by the neighborhood basis of $s^* = (\dots,s_{-1}^*,s_0^*,s_1^*,\dots)$,
\[
I_j(s^*) = \{ s \in S\mid \; s_k = s_k^*, \; |k| < j\},\qquad s^* \in S.
\]
It is well known that the shift $\sigma: S \to S$ defined by $\sigma(s)_i = s_{i+1}$ is a homeomorphism. The shift $\sigma$ is transitive and the set of its periodic orbits is dense in $S$. It has sensitive dependence on initial conditions. It is one of the paradigms of chaos.

The second goal of the paper is to prove the existence of \emph{oscillatory orbits}, that is, solutions $(x(t),z(t))$ of~\eqref{eq:H1} with the property that $\limsup_t z(t) = \infty$ and $\liminf_t z(t) < \infty$, that is, solutions that go higher and higher but always go back again to a finite distance of the Cu surface. 

The claims in the present paper follow the scheme proposed by Moser in~\cite{Moser01}, which was also used in the restricted planar three body problem in~\cite{LlibreS80}. In this last paper, an important technicality, not present in Moser's work, appeared: the exponentially small behavior of the difference between the manifolds which give rise to the interesting dynamics. This problem was later overcome in~\cite{GuardiaMS16}. This exponentially small behavior also appears in the model under consideration here and dealing with it will represent an important part of our work.

Our main theorem is the following.

\begin{theorem}
\label{thm:teoremaprincipalenvariablesoriginals}
There exists a function $f_1$ of the coefficients $r_1,r_2,\dots$  such that if $f_1 \neq 0$, then, for any $h$ large enough, there exists a section $\Sigma \subset \{H_\mathrm{CM} = h\}$ of the vector field associated to $H_\mathrm{CM}$ and a subset $\II \subset \Sigma$ such that the Poincar\'e map $\Psi: \II \to \II$ is a homeomorphism
and is conjugate to the shift $\sigma$ of infinite symbols. The set $\II$ is a hyperbolic set for $\Psi$.

Furthermore, for any $h$ large enough, the Hamiltonian $H_{\mathrm{CM}}$ in~\eqref{eq:H1} possesses oscillatory orbits.
\end{theorem}

\begin{remark}
The function $f_1$ appears in a problem independent of the energy $h$. This equation is usually known in the literature as \emph{inner equation}.
The numerical evidence in~\cite{BarrabesFMO23} strongly suggests that for the experimental values mentioned above, $f_1$ is non-zero. Furthermore, we prove that
if one replaces $r_i$ by $\varepsilon \tilde r_i$ in the definition of $V$,  in~\eqref{def:potencialdeMorsecorrugat}, that is, if we consider the term $V_\mathrm{C}$ as a small perturbation, then
\[
f_1 = \varepsilon \frac{\pi}{4} \tilde r_1 + \OO(\varepsilon^2).
\]
\end{remark}

The rest of the paper is devoted to prove Theorem~\ref{thm:teoremaprincipalenvariablesoriginals}. In particular, it will be an immediate consequence of Theorem~\ref{thm:conjugacioambelshift}.

We would like to remark that the hyperbolic set, as is the case of the Sitnikov problem considered in~\cite{Moser01} or the restricted planar three body problem in \cite{LlibreS80,GuardiaMS16}, is related to the invariant manifolds of a certain periodic orbit at $z = \infty$. In our case, since the equations of motion are
\begin{equation}
\label{def:equacionsdelmovimentoriginals}
\begin{aligned}
  \dot x &= \dfrac1{m}\,p_x, \\
  \dot z &= \dfrac1{m}\,p_z ,\\
  \dot p_x & = -\frac{2 \pi}{a}De^{-2\alpha z} V'\left(\frac{2 \pi x}{a}\right), \\
  \dot p_z & = -2D\alpha\,e^{-\alpha z} \left(1-e^{-\alpha z}\left(1+V\left(\frac{2 \pi x}{a}\right)\right)\right),
\end{aligned}
\end{equation}
we have that the set
\begin{equation}
\label{def:orbita_periodica_dinfinit_en_coordenades_originals}
\{(x,z,p_x,p_z) \in \R/a\Z \times \R \times \R^2 \mid z = \infty,\; p_z = 0,\; p_x = (2m h)^{1/2}\}
\end{equation}
is a periodic orbit of~\eqref{def:equacionsdelmovimentoriginals} at infinity in the energy level $h$. The proof of Theorem~\ref{thm:teoremaprincipalenvariablesoriginals} consists in checking that this periodic orbit possesses invariant stable and unstable manifolds,
that these invariant manifolds intersect transversally if $h$ is large enough and then prove a \emph{suitable $\lambda$-lemma} that ensures that this transversal intersection gives rise to the standard isolating blocks with cone conditions. This last part is due to the fact that the periodic orbit~\eqref{def:orbita_periodica_dinfinit_en_coordenades_originals} is neither hyperbolic nor elliptic, but degenerate. However, although degenerate, it possesses stable and unstable invariant manifolds. The study of the invariant manifolds of these type of degenerate objects goes back to~\cite{McGehee73}. See also~\cite{BaldomaFdlLM07}, where the parametrization method is used, and the subsequent works, \cite{BFM2020a,BFM2020b,BFM17,BFM20}. In this case, when the periodic orbit is degenerate, the standard $\lambda$-lemma does not hold. In particular, it is not true that the forward images of a manifold intersecting transversally the stable manifold of the orbit accumulate to the whole unstable manifold. Hence, a different argument is needed to control the passage close to the periodic orbit. This was already known by Moser in~\cite{Moser01}. However, his proof does not directly apply to the present case, because the degree of degeneracy of our case is different from the one of the Sitnikov problem. We present another proof, based in the ideas in~\cite{GuardiaMPS22}.

\subsection{Exponentially small splitting of invariant manifolds}

One of the main difficulties of the present work is to establish that the angle of intersection between the stable and unstable invariant manifolds of certain periodic orbits is non-zero. Indeed, since our goal is to deal with the physical problem, in which the corrugation, modelled by the function~$V$ in~\eqref{def:potencialdeMorsecorrugat}, is fixed, the only parameter we will have to deal with the problem will be the energy of the system. We will see that the angle of intersection is in fact exponentially small in the energy hence precluding the use of the standard Melnikov theory. 

Starting with the seminal paper of Lazutkin~\cite{Lazutkin84russian} (see the English translation in \cite{Lazutkin84}), there is a long list of works in the literature concerning the exponentially small splitting of separatrices. 
In particular, the method introduced in~\cite{Sauzin01} has been essential in posterior developments of the field. See \cite{BaldomaFGS11} and the references therein. Our approach here is similar to the
the one established in \cite{BaldomaFGS11} and  \cite{Baldoma06}. It is important to remark that the results in these last two papers do not apply to our setting. As a matter of fact, if one tries to write Hamiltonian~\eqref{eq:H1}, which has 2 degrees of freedom, as a 1$\frac12$ degrees of freedom one using the Poincar\'e-Cartan reduction, the reduced Hamiltonian does not satisfy some of the hypotheses in  \cite{BaldomaFGS11,Baldoma06}. Particularly, when written as an integrable Hamiltonian plus a perturbation, the perturbation is not polynomial and has branching singularities when extended to the complex domain. This is true even in the McGehee variables that are introduced in Section~\ref{sec:McGeheeetal}.  

Since we are not assuming the perturbation to be small, we need to study the \emph{inner equation} associated to the problem in order to capture the leading term of the exponentially small behavior of the splitting. This inner equation was used by Gelfreich in~\cite{Gelfreich99} to prove the splitting of sepatrices in the Chirikov standard map. See also~\cite{MartinSS10a,MartinSS10b}, where the McMillan map was studied by using resurgence theory. The study of the inner equation for the H\'enon map was performed in~\cite{GelfreichS01}. 

The exponentially small splitting of invariant manifolds also appears in other physical problems, since it is related to the existence of some fast frequencies on the system. In particular, it has been dealt with successfully in some problem of celestial mechanics~\cite{GuardiaMS16,GuardiaMPS22,GuardiaPS23}. In these last cases, however, the Melnikov function predicts correctly the first order of the splitting.

\subsection{Structure of the paper}

The structure of the paper is as follows.

In Section~\ref{sec:McGeheeetal} we write the system in suitable McGehee coordinates, describe the geometric behaviour of it and claim the main theorem of the splitting of separatrices.

In Section~\ref{sec:chaotic_dynamics} we apply the splitting theorem to deduce the existence of chaotic dynamics for the system. To do so, we need to cope with the problem that the periodic orbits under consideration are not hyperbolic but degenerate, which implies that the standard lambda lemma does not apply. Here we use a version of the parabolic lambda lemma in~\cite{GuardiaMPS22} to obtain the conjugation to the Bernoulli shift.

Sections~\ref{sec:HJouter} to~\ref{sec:diferenciaentresolucionseqHJ} deal with the actual splitting of the invariant manifolds of the periodic orbit at infinity. In Sections~\ref{sec:HJouter} and~\ref{sec:extesio_coeficients_de_Fourier} we obtain suitable approximations of the invariant manifolds by solving an appropriate Hamilton-Jacobi equation. These approximations are valid in a certain complex domain. However, we are not able to obtain enough information of their difference. To do so, we introduce the inner equation in Section~\ref{sec:equacio_inner}, from which we obtain two solutions. In Section~\ref{sec:extensio_al_domini_inner}, we check that these solutions are also good approximations of the invariant manifolds in some region of their complex domain of definition. In Section~\ref{sec:diferencia_solucions_inner} we obtain an exponentially small formula for the difference of solutions of the inner equation which, in turn, we use in Section~\ref{sec:diferenciaentresolucionseqHJ} to obtain the exponentially small formula for the difference of the invariant manifolds.

We have left to Appendix~\ref{sec:provadelteoremathm:lambdalemma} the proof of the parabolic lambda lemma (Theorem~\ref{thm:lambdalemma}). Appendix~\ref{sec:apendix_proves_lemes_lambda_lema} contains the proofs of some technical lemmas used in the proof of Theorem~\ref{thm:lambdalemma}.

\section{Hamiltonian formulation of the He-Cu scattering problem and splitting of separatrices}
\label{sec:McGeheeetal}

\subsection{McGehee-like coordinates and fast dynamics}

We are interested in motions for which the He particle arrives to $z = \infty$ with zero momentum.
For this reason, we introduce the McGehee-like coordinates
\begin{equation*}
\left\{
\begin{aligned}
A q^2 &= e^{-\alpha z}, \qquad &  \frac{a}{2\pi}\theta &= x, \\
B p &= p_z, \qquad & C I &= p_x,
\end{aligned}
\right.
\end{equation*}
where
\[
A = 2, \qquad B = \frac{a\alpha}{4 \pi}C, \qquad C^2 = 2 mD \left( \frac{8 \pi}{a \alpha}\right)^2.
\]
Without loss of generality, we take $C>0$.
This change transforms the standard $2$-form
$dx \wedge dp_x + dz \wedge dp_z$ into the $b$-symplectic form $(aC/2\pi)\, \omega$, where
\begin{equation}
\label{def:rescaledbsymplecticform}
\omega =  d\theta \wedge dI -\dfrac{1}{ q} dq \wedge dp
\end{equation}
and the Hamiltonian function $H_{\mathrm{CM}}$ in~\eqref{eq:H1} becomes $8D\cdot H$, where
\begin{equation}
\label{def:rescaledHamiltonian}
H(q,p,\theta,I) = \dfrac1{2}(\nu I^2+p^2)-\frac{1}{2}q^2+\frac{1}{2}q^4+\frac{1}{2}q^4 V(\theta)
\end{equation}
and
\begin{equation*}
\nu = \left(\frac{4\pi}{a \alpha}\right)^2.
\end{equation*}
For the experimental values of the He-Cu problem, the value of $\nu$ is
\[
\nu = 11.051879175935...
\]
Rescaling time, the dynamics of $H_{\mathrm{CM}}$ is equivalent to the one generated by $H$ with respect to the form $\omega$.

To study the behavior of the system for large values of $I_0$,  we introduce  the new variable $J$ by $I = I_0+J$ in~\eqref{def:rescaledHamiltonian} with $I_0 \gg 1$. This change preserves the $2$-form $\omega$. We will denote the Hamiltonian in these
new variables with the same letter, namely,
\begin{equation}
\label{def:rescaledHamiltonianI0}
H(q,p,\theta,J) =  H_0(q,p,J)+  H_1 (q,p,\theta,J),
\end{equation}
where
\begin{equation}
\label{def:H0H1}
\begin{aligned}
H_0(q,p,J) & = \frac{1}{2}(\nu (I_0+J)^2+p^2)-\frac{1}{2}q^2+\frac{1}{2}q^4, \\
H_1(q,p,\theta,J) & = \frac{1}{2} q^4 V(\theta).
\end{aligned}
\end{equation}
Since $H_0$ does not depend on $\theta$, it is integrable.
We remark that $H$ looks like a periodically perturbed Duffing equation. However, since the vector field generated by $H$ is obtained through the non-standard $2$-form $\omega$, the system is not equivalent to the Duffing equation. Indeed, the equations of motion generated by $H$ are
\begin{equation}
\label{eq:sistema1}
  \begin{aligned}
 \dot q & = -q \frac{\partial H}{\partial p} =
 -qp, & \quad    \dot \theta &= \frac{\partial H}{\partial I} = \nu I_0 + \nu J,   \\
  \dot p & = -q \left(-\frac{\partial H}{\partial q}\right)= - q^2+2 q^4 + 2 q^4 V (\theta), & \quad \dot J &= -\frac{\partial H}{\partial \theta} = - \frac{ q^4}{2} V'(\theta),
\end{aligned}
\end{equation}
which is not a Duffing oscillator due to the presence of the factor $q$ in the $(q,p)$ components of the vector field.

We consider, at the energy level $\nu I_0^2/2$ of $H$ in~\eqref{def:rescaledHamiltonianI0}, the set
\[
\Lambda_{\nu I_0} = \{q = p = J = 0, \,\theta \in \T\}.
\]
From~\eqref{eq:sistema1}, it is invariant and the dynamics on $\Lambda_{\nu I_0}$ is given by $\theta = \theta_0+ \nu I_0t$, that is, it is periodic with frequency $\nu I_0/(2\pi)$, which is fast when $\nu I_0$ is large.
That is, at the energy level $\nu I_0^2/2$, for large values of $\nu I_0$, $H$ can be seen as a fast periodic perturbation of the integrable Hamiltonian $H_0$. We emphasize that the perturbation $H_1$ is not small since we are assuming that $r_1$ is fixed and we will consider motions in which $q$ will reach size of $\OO(1)$. The proof of Theorem~\ref{thm:teoremaprincipalenvariablesoriginals} consists in proving that each of these periodic orbits possesses invariant manifolds and that they intersect transversally. From this intersection we will deduce the existence of a horseshoe with infinitely many symbols. However, since the system is a fast periodic perturbation of an integrable system, it is well known that the angle of intersection of the manifolds will be smaller than any power of the inverse of the frequency which makes the question of computing it a \emph{beyond all orders phenomenon}.

\begin{remark}
\label{rem:campproperaintegragle}
Although $H_1$ is not small, it is well known that the fact that depends on a fast angle implies that its contribution
averages out up to an exponentially small remainder. Indeed, we can apply a step of averaging  and still obtain an explicit expression for the remainder, which is of size~$O((\nu I_0)^{-1})$.

It is immediate to check that the change
\[
\theta = \Theta, \qquad J = K + A'(\Theta) Q^4, \qquad q = Q, \qquad p = P+4 A(\Theta) Q^4,
\]
preserves the $2$-form $\omega$ in~\eqref{def:rescaledbsymplecticform}.
In these new variables, $H$ becomes
\[
\begin{aligned}
\widetilde H (Q,P,\Theta,K)  = &
\frac{1}{2}(\nu (I_0+K - A'(\Theta) Q^4)^2+(P+4 A(\Theta) Q^4)^2)-\frac{1}{2}Q^2+\frac{1}{2}Q^4+\frac{1}{2}Q^4 V(\Theta) \\
 = & \frac{1}{2}(\nu (I_0+K)^2+P^2)-\frac{1}{2}Q^2+\frac{1}{2}Q^4+\left( \frac{1}{2} V(\Theta)-\nu I_0 A'(\Theta)\right)Q^4 \\
 & -\nu A'(\Theta) K Q^4 + 4 A(\Theta) PQ^4 + 8 A(\Theta)^2 Q^8
+\frac{1}{2} \nu A'(\Theta)^2 Q^8 .
\end{aligned}
\]
Taking $A$ such that  $A'(\Theta) = -V(\theta)/(2\nu I_0)$
and introducing $W(\Theta)= \int ^{\Theta} V(\vartheta)\, d\vartheta $, we obtain
\[
\widetilde H (Q,P,\Theta,K)
 =  H_0 (Q,P,K) + \wt H_1 (Q,P,\Theta,K),
\]
where
\[
\widetilde H_1 (Q,P,\Theta,K) = \frac{1}{2\nu I_0} \left(-\nu V(\Theta) K Q^4 + 4   W (\Theta) PQ^4 + \frac{1}{4 I_0}  V (\Theta)^2 Q^8 + \frac{4}{\nu I_0} W(\Theta)^2 Q^8  \right).
\]
The computations of the Melnikov potential in Section~\ref{sec:Melnikov} below lead to a formula of the same order. We will not follow this approach. On the contrary, we will work directly with the Hamiltonian $H = H_0+H_1$ in~\eqref{def:rescaledHamiltonianI0}.
\end{remark}

\subsection{Dynamics of  $H_0$}

In view of Remark~\ref{rem:campproperaintegragle}, although $H_1$ is not small, we can see $H_0$ as a reference system for $H$. Here we describe its dynamics.
First we observe that, since $H_0$ does not depend on $\theta$, $J$ is a conserved quantity. Hence, $H_0$ is integrable.
Using the $2$-form $\omega$ in~\eqref{def:rescaledbsymplecticform}, the equations of motion of $H_0$ are
\begin{equation}
\label{eq:sistemaH0}
  \begin{aligned}
  \dot q &  =  -qp,  & \quad  \dot \theta & = \nu I_0 + \nu J,  \\
   \dot p & = - q^2+2 q^4, &  \quad\dot J & = 0.
\end{aligned}
\end{equation}
The set
\[
\Lambda_{\nu I_0} = \{q = p = J = 0, \,\theta \in \T\}
\]
is a periodic orbit in $\{H_0 = \nu I_0^2/2\}$. From~\eqref{eq:sistemaH0}, as we have already mentioned, it is clear that the restriction of $\Lambda_{\nu I_0}$ to
$\{H_0 = \nu I_0^2/2, \; J =0\}$ is not hyperbolic. However, it possesses invariant stable and unstable manifolds, that generate two homoclinic loops, given by
\[
W^0_{\nu I_0} = \left\{(q,p,\theta,J) \mid J = 0,\; \theta \in \T,\, H_0(q,p,0) = \frac{1}{2} \nu I_0^2\right\}.
\]
See Figure~\ref{fig:homoclinica}. We will focus on the loop with $q\ge 0$. The figure is slightly misleading: for any $J$, the set $\{q = 0, \; p =p_0\}$ is a periodic orbit. In particular,
$\{q=0\}$ is invariant, even for the full Hamiltonian $H$. This means that, in fact, in the $(q,p)$-plane, $(q,p) = (0,0)$ is not a topological saddle because, in particular,  is not an isolated equilibrium point because $q=0$ is a line of equilibria. As a consequence, the standard $\lambda$-lemma is no longer true and another statement will be necessary. The needed result is  Theorem~\ref{thm:lambdalemma} below. We will be interested in the part of the phase space where $q>0$.
\begin{figure}[h]
\begin{center}
\includegraphics[width=.5\textwidth]{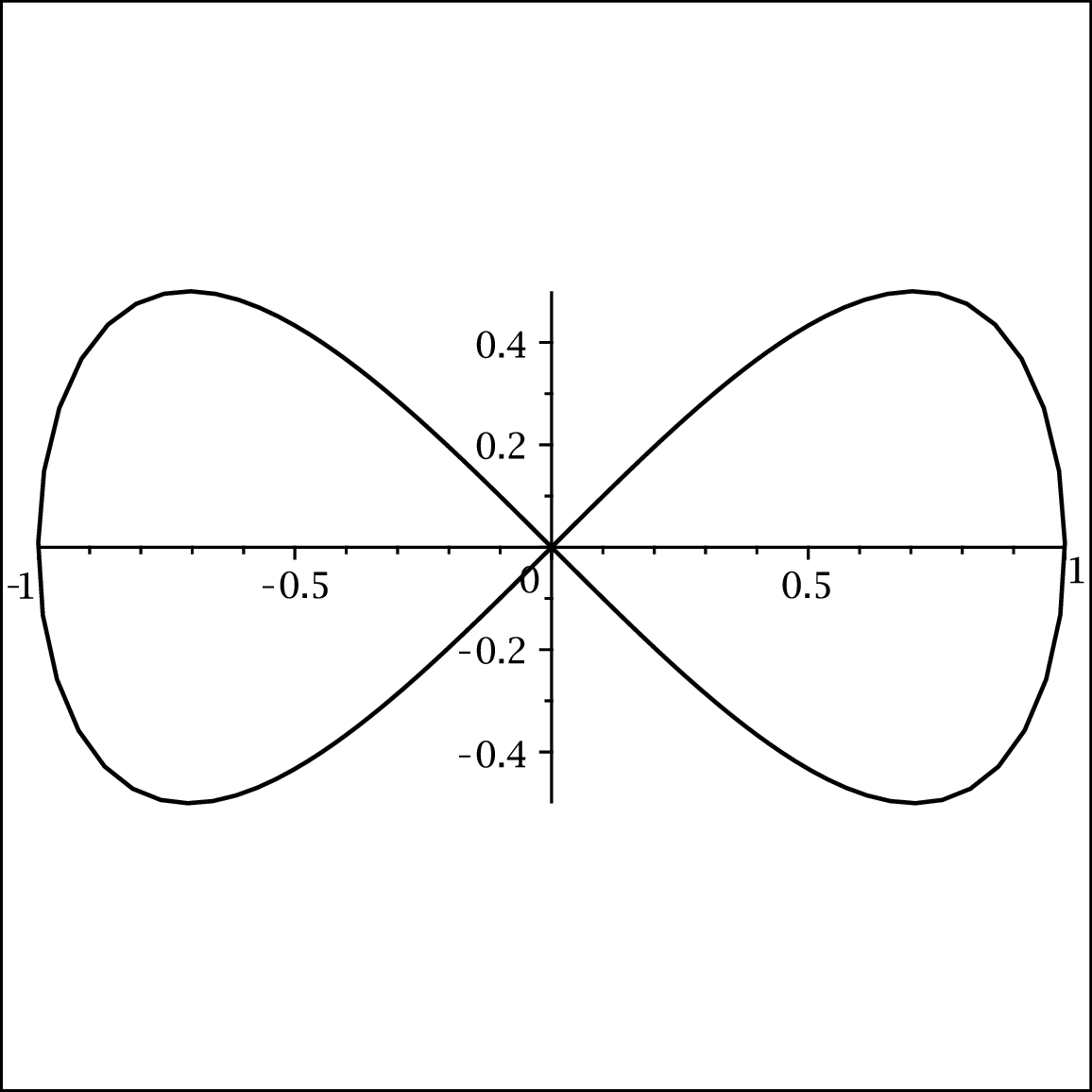}
\end{center}
\caption{The projection of the invariant manifolds $W^0_{\nu I_0}$ onto the $(q,p)$-plane. For each $p = p_0$,
$\{ (0,p_0,\theta,0)\mid \theta \in \T \}$
is a periodic orbit.}\label{fig:homoclinica}
\end{figure}

The following lemma, whose proof is a straightforward computation, provides the time parame\-tri\-zation of the right hand side homoclinic orbit.

\begin{lemma}
\label{lem:homoclinic} The time parametrization of the homoclinic orbit to $(0,0)$ of $(H_0)_{\mid J=0}$ that passes through $(1,0)$ at $t = 0$ is
given by
\[
q_h(t) = \frac{1}{\sqrt{1+t^2}},\qquad p_h(t) = - \frac{\dot q_h(t)}{q_h(t)} =  \frac{t}{1+t^2}.
\]
\end{lemma}

As a consequence,
\begin{equation}
\label{def:Gamma0}
\Gamma^0(u,\theta) =
\begin{pmatrix}
 q_h(u) \\
 p_h(u)\\
 \theta \\
 0
 \end{pmatrix}
\end{equation}
parametrizes $W^0_{\nu I_0}$. Denoting $\phi^0_t$ the flow generated by $H_0$, this parametrization satisfies
$\phi^0_t \circ \Gamma^0(u,\theta) = \Gamma^0(u+t,\theta+\nu I_0 t)$.

\subsection{The Melnikov potential}
\label{sec:Melnikov}

As usual, we define the Melnikov potential associated to the Hamiltonian~$H$ as
\begin{equation}
\label{def:Melnikovpotential}
L(u,\theta) = -\int_{-\infty}^{\infty} H_1(q_h(u+t),p_h(u+t),\theta+\nu I_0 t,0)\, dt.
\end{equation}
Expanding $V$ in Fourier series,
\begin{equation*}
V(\theta) = \sum_{k\in \Z}V^{[k]}  e^{i k \theta},
\end{equation*}
by~\eqref{def:H0H1}, we can write
\[
H_1(q,p,\theta,0)  = \frac{q^4}{2} V(\theta)
 = \frac{q^4}{2} \sum_{k\in \Z}V^{[k]}  e^{i k \theta}.
\]
Since $V$ is real analytic, even and has zero average, we have that
\[
V^{[k]} \in \R, \qquad V^{[0]} = 0, \qquad V^{[k]} = V^{[-k]}, \quad k \in \Z\setminus\{0\}
\]
and there exists $\sigma_0 >0$ and $K>0$ such that, for all $k\in \Z$,
\begin{equation}
	\label{fitadelsVk}
	|V^{[k]}| \le K e^{-|k|\sigma_0}.
\end{equation}

Hence,
\[
\begin{aligned}
L(u,\theta) 
=- \int_{-\infty}^{\infty} H_1(q_h(t),p_h(t),\theta-\nu I_0 u +\nu I_0 t,0)\, dt
= -\sum_{k\in \Z} e^{i k (\theta-\nu I_0 u)} \int_{-\infty}^{\infty}  \frac{q_h^4(t)}{2}  V^{[k]} e^{i k \nu I_0 t}
\, dt.
\end{aligned}
\]
That is,
\[
L(u,\theta) = \sum_{k\in \Z} L^{[k]}(\nu I_0) e^{i k (\theta-\nu I_0 u)},
\]
where
\[
L^{[k]}(\nu I_0) = -\frac{1}{2}V^{[k]}\int_{-\infty}^{\infty} q_h^4(t)   e^{i k \nu I_0 t}\, dt .
\]

\begin{proposition}
The Melnikov potential $L$ satisfies
\[
L^{[k]}(\nu I_0) = -\frac{\pi \nu I_0 V^{[k]}}{4} e^{-|k|\nu I_0}\left( |k|  + \frac{1}{\nu I_0}\right), \qquad k \in \Z.
\]
In particular, since $V^{[0]} = 0$ then $L^{[0]}(\nu I_0) = 0$.
\end{proposition}

\begin{proof}
Since, in view of Lemma~\ref{lem:homoclinic}, $q_h^4$ is a meromorphic function that has poles of order $2$ at $\pm i$, the claim follows from a straightforward residue computation.
\end{proof}

\begin{corollary}
If $r_1 \neq 0$ in~\eqref{def:potencialdeMorsecorrugat}, then
\[
L(u,\theta) =   -\frac{\pi}{4} \nu I_0   r_1 e^{-\nu I_0}\left[\left( 1  + \frac{1}{\nu I_0}\right) \cos (\theta-\nu I_0 u)+ O(e^{-\nu I_0})\right].
\]
\end{corollary}

If the coefficients $r_i$ in the definition of $V$ in~\eqref{def:potencialdeMorsecorrugat} satisfy $r_i = \varepsilon \tilde r_i$  and $\tilde r_1 \neq 0$, the standard Melnikov theory implies that the distance between the unstable and stable invariant manifolds of $\Lambda_{\nu I_0}$ is given by the derivatives of $\varepsilon L$ plus an error of size $\OO(\varepsilon^2)$. However, since $L$ has size $\OO(e^{-\nu I_0})$, this means that at the energy level $\nu I_0^2/2$, $\varepsilon L$ gives the leading order of the distance only if $\varepsilon < \OO(e^{-\nu I_0})$, that is, for a fixed value of the energy, the invariant manifolds of infinity split only if the corrugation is exponentially small in the energy.  We will see that $\varepsilon$ and $\nu I_0$ can be taken as independent parameters while the exponentially small formula for the distance will remain valid. The physical implications are important: we will see that, for a given corrugation, if certain coefficient is different from $0$ (that depends on the coefficients of the corrugation and is generically non-zero), for any large enough energy, the invariant manifolds of infinity split. Moreover, this coefficient depends on first order on  $L^{[1]}$ if $\varepsilon$ is small.

It is important to remark that if $r_1 = 0$, the Melnikov function is of order $e^{-2\nu I_0}$. However, in this case, the true behavior, given by formula~\eqref{eq:formula_exponencialment_petita_de_la_diferencia} in Theorem~\ref{thm:teorema_principal_trencament_de_separatrius}, below, can be of order $e^{-\nu I_0}$,
if certain coefficient $f_1$ is different from $0$. See~\cite{BMS23}.

\subsection{Dynamics of the full system. Splitting of the invariant manifolds}

The periodic orbit $\Lambda_{\nu I_0} = \{q = p = J = 0, \,\theta \in \T\}$ of $H_0$ remains when we consider the full system $H$ in~\eqref{def:rescaledHamiltonianI0}, with the same frequency $\nu I_0$. Although it is not hyperbolic, we will prove that it has invariant unstable and stable manifolds, which, if certain non-degeneracy condition holds, no longer coincide. Next theorem summarizes the claim. It is the first step in the proof of Theorem~\ref{thm:teoremaprincipalenvariablesoriginals}.

\begin{theorem}
\label{thm:teorema_principal_trencament_de_separatrius}
If $\nu I_0$ is large enough, the periodic orbit $\Lambda_{\nu I_0} = \{q = p = J = 0,\, \theta \in \T\}$ possesses invariant unstable and stable manifolds, $W^{\pm}_{\nu I_0}$. Moreover, for any $0< u_0 < u_1$, there exist analytic functions $\Phi^{\pm} : [u_0,u_1] \times \T \to \R$ such that
\begin{equation}
\label{def:Gammapm}
\Gamma^{\pm} (u, \theta) =
\begin{pmatrix}
q_h(u) \\
p_h(u)^{-1} \partial_u \Phi^{\pm}(u,\theta) \\
\theta \\
\partial_\theta \Phi^{\pm}(u,\theta)
\end{pmatrix}
\end{equation}
are parametrizations of pieces of $W^{\pm}_{\nu I_0}$ satisfying
\begin{enumerate}
\item[(1)] $\Gamma^{\pm} = \Gamma^0 + \OO( (\nu I_0)^{-1})$, where $\Gamma^0$ is given in \eqref{def:Gamma0},  
\item[(2)] there exists $f_1, \Lambda_0  \in \R$ such that, for any $j,k \ge 0$, $0 \le j+k \le 3$,
\begin{multline}
\label{eq:formula_exponencialment_petita_de_la_diferencia}
\partial^j_u \partial^k_\theta (\Phi^+(u,\theta) - \Phi^-(u,\theta) -\Lambda_0) \\
 = (-1)^j
(\nu I_0)^{j+1} e^{-\nu I_0} \left(2f_1 \partial_\alpha^{j+k} \cos \alpha_{\mid{\alpha = (\theta-\nu I_0 u)}} +  \OO\left(\frac{1}{ \log (\nu I_0)}\right)\right).
\end{multline}
The coefficients $f_1$ and $\Lambda_0$ do not depend on $\nu I_0$. They only depend on the coefficients $r_i$ of the function $V$ in $H_1$.
If one assumes that $r_i = \varepsilon \tilde r_i$, i.e.\ the corrugation term $H_1$ is a small perturbation of $H_0$, then
\begin{equation}
\label{eq:expansio_f1_epsilon}
f_1 = \frac{\pi}{2} \tilde r_1 \varepsilon+ \OO(\varepsilon^2).
\end{equation}
\end{enumerate}
\end{theorem}

\begin{remark}
For the experimental values of $V$ corresponding to the interaction between helium and copper atoms, numerical experiments in~\cite{BarrabesFMO23} suggest that $f_1 \neq 0$. However, if $r_i = \varepsilon \tilde r_i$, we will see that $f_1$ is an analytic function of $\varepsilon$. Then, \eqref{eq:expansio_f1_epsilon} implies that, if $\tilde r_1 \neq 0$, $f_1$ can only vanish for a discrete number of values of $\varepsilon$.
\end{remark}

\begin{corollary}
\label{cor:existencia_de_punts_homoclinics}
If $\nu I_0$ is large enough, the invariant manifolds $W^\pm_{\nu I_0}$ of $\Lambda_{\nu I_0}$ intersect transversely along two primary homoclinic orbits.
\end{corollary}

\begin{proof} Since $W^{\pm}_{\nu I_0}$ are contained in the level $\nu I_0^2/2$ of $H$, the expression of $\Gamma^\pm$ in~\eqref{def:Gammapm} implies that their intersections are given by $\partial_u \Phi^{+}(u,\theta) - \partial_u \Phi^{-}(u,\theta) = 0$. Then, \eqref{eq:formula_exponencialment_petita_de_la_diferencia} allows us to apply  the standard implicit function theorem to obtain that the primary homoclinic orbits are given by $\theta -\nu I_0 u + \OO((\nu I_0)^{-1}) = k \pi$.
\end{proof}

Theorem~\ref{thm:teorema_principal_trencament_de_separatrius} provides an exponentially small formula for the difference of the invariant manifolds $W^{\pm}_{\nu I_0}$, with the energy of the system as parameter. 
We will find suitable approximations of the invariant manifolds of $\Lambda_{\nu I_0}$ which extend to complex values of their variables. The leading term of these approximations, for both the stable and unstable manifolds, will be the unperturbed separatrix $\Gamma^0$. This is done in Sections~\ref{sec:HJouter} and~\ref{sec:extesio_coeficients_de_Fourier}. In many cases, this is enough to obtain the exponentially small formula, whose leading term is given by the Melnikov potential~\eqref{def:Melnikovpotential}. Here, however, in order to capture the leading term of the exponentially small difference between the manifolds, we need to resort to what is often known as the \emph{inner equation}. This is done in Section~\ref{sec:equacio_inner} and~\ref{sec:diferencia_solucions_inner}. In Section~\ref{sec:extensio_al_domini_inner} the manifolds are compared with the solutions of the inner equation and, finally, all the information is gathered in Section~\ref{sec:diferenciaentresolucionseqHJ}, completing the proof of Theorem~\ref{thm:teorema_principal_trencament_de_separatrius}.

\section{Chaotic dynamics}
\label{sec:chaotic_dynamics}

Here we define a return map on a suitable section
such that it possesses an invariant hyperbolic set and when restricted to this set, it is topologically conjugate to the shift of infinite symbols. The construction is analogous to the one of Moser in \cite{Moser01}.

We have that the set
\begin{equation*}
\Lambda = \{(q,p,\theta,I)\mid \; q = p = 0\}
\end{equation*}
is invariant by the flow of $H$ in~\eqref{def:rescaledHamiltonian}. Moreover,
\begin{equation*}
\Lambda_{\nu I_0} = \{(q,p,\theta,I)\mid \; q = p = 0,\; I = I_0\}
\end{equation*}
is a periodic orbit in the level set $H = \nu I_0^2/2$ and $\Lambda = \cup_{I_0\in \R} \Lambda_{\nu I_0}$.
Each of these periodic orbits, although parabolic, possesses invariant manifolds, whose union is a foliation of the invariant manifolds of $\Lambda$, which can be seen as  ``normally parabolic manifolds''. We will use these invariant manifolds, which are regular enough at $\Lambda_{\nu I_0}$, to construct a good system of coordinates close to $q=p = 0$. In these coordinates, we will consider suitable Poincar\'e sections and Poincar\'e maps between them, with subsets in which the dynamics is conjugate to the standard shift with infinitely many symbols.

\subsection{Poincar\'e-Cartan reduction}
\label{sec:PoincareCartan}
Consider the Hamiltonian~\eqref{def:rescaledHamiltonianI0}.
At the energy level $H= \nu I_0^2/2$, by means of the Poincar\'e-Cartan reduction, we can write $H$
as a $1\frac{1}{2}$ degrees of freedom Hamiltonian, with~$\theta$ as the new time. Indeed, let $K(q,p,\theta)$ be such that
\begin{equation}
\label{HamiltonianK}
H(q,p,\theta,-K(q,p,\theta)) = \frac{\nu I_0^2}{2} .
\end{equation}
Then, the equations of motion defined by the Hamiltonian $H$ (with respect to
the $b$-symplectic form~\eqref{def:rescaledbsymplecticform}) in the energy level $H= \nu I_0^2/2$ are equivalent to the equations of motion of $K$ with respect to
the $b$-symplectic form
\begin{equation*}
\tilde \omega = - \frac{1}{q } dq \wedge dp,
\end{equation*}
that is,
\[
\begin{aligned}
\frac{d q}{d \theta} & = -   q \frac{\partial K}{\partial p} =
 -   q\frac{\partial H}{\partial p} \left(\frac{\partial H}{\partial J}\right)^{-1},\\
 \frac{d p}{d \theta} & = -   q \left(-\frac{\partial K}{\partial q}\right)
  =
   q \frac{\partial H}{\partial q}\left(\frac{\partial H}{\partial J}\right)^{-1}.
\end{aligned}
\]
Introducing
\[
\widetilde H(q,p,\theta) = \dfrac1{2}p^2-\frac{1}{2}q^2+\frac{1}{2}q^4+\frac{1}{2}q^4 V(\theta),
\]
from~\eqref{HamiltonianK} we have that
\[
K(q,p,\theta) = \frac{1}{\nu I_0} \widetilde H(q,p,\theta) + \frac{1}{(\nu I_0)^{3}} \OO( \widetilde H^2)
= \frac{1}{\nu I_0} \widetilde H(q,p,\theta) + \OO_4(q,p).
\]
Rescaling the variables $(q,p)$, the equations of motion are
\begin{equation}
\label{eq:sistemareduitPoincareCartanireescalat}
\begin{aligned}
\frac{dq}{d\theta}  & = -q(p + \OO_3(q,p)),\\
\frac{dp}{d\theta}  & = -q(q+\OO_3(q,p)),
\end{aligned}
\end{equation}
where the $\OO_3(q,p)$ terms are $2\pi$-periodic functions of $\theta$.
Now, the set $\{q=0,\; p = 0, \; \theta \in \T\}$ is a periodic orbit of~\eqref{eq:sistemareduitPoincareCartanireescalat}, which corresponds to $\Lambda_{\nu I_0}$ in this setting.

\subsection{Local coordinates around $q=p=0$ and a parabolic $\lambda$-lemma}

After the Poincar\'e-Cartan reduction in Section~\ref{sec:PoincareCartan}, the periodic orbit $\Lambda_{\nu I_0}$ becomes $\{q = p = 0\}$. It is proven in~\cite{BaldomaFdlLM07} that this periodic orbit, although not hyperbolic, possesses stable and unstable invariant manifolds which are analytic everywhere except at $q=p=0$, where they are $C^{\infty}$. Next proposition provides the local straightening of the invariant manifolds.

It is important to remark that, in this case, due to the parabolic character of the equilibrium point, the (real) invariant manifolds defined as the set of initial points of orbits that tend to the equilibrium point when $t\to -\infty$ (for the unstable one) or when  $t\to \infty$  (for the stable one) are located on one side of the equilibrium point.

\begin{proposition}
\label{prop:varietatsredressades}
There exists a $C^{\infty}$ change of coordinates
\[
(u,v) = F(q,p,\theta) = \left(\frac{q-p}{2},\frac{q+p}{2} \right) + \OO_2(q,p), \qquad t = \theta,
\]
$2\pi$-periodic in $\theta$, defined in a neighborhood $U$ of $\{(q,p,\theta)\mid q = p = 0\}$,  that transforms~\eqref{eq:sistemareduitPoincareCartanireescalat} into
\begin{equation}
\label{def:varietatsredressades}
\left\{
\begin{aligned}
\dot u & = u (u+v + \OO_2(u,v)), \\
\dot v & = -v (u+v + \OO_2(u,v)), \\
\dot t & = 1,
\end{aligned}
\right.
\end{equation}
where the $\OO_2(u,v)$ terms are $2\pi$-periodic functions of $t$.
\end{proposition}

\begin{proof}
It is analogous to the proof of Proposition~4 in \cite{GuardiaSMS17}.
\end{proof}

Let $V_\rho = \{(u,v,t) \in \R^2 \times \T \mid \; |u|, |v| < \rho\}$. If $\rho>0$ is small enough, equations~\eqref{def:varietatsredressades} are well defined in $V_\rho$. In the local coordinates $(u,v,t)$, the unstable and stable invariant manifolds of $\Lambda_{\nu I_0}$ are
\begin{equation*}
	\begin{aligned}
		W^u(\Lambda_{\nu I_0}) & = \{(u,v,t) \in U \mid \; v = 0, \; u>0\}, \\
		W^s(\Lambda_{\nu I_0}) & = \{(u,v,t)\in U \mid \; u = 0, \; v>0\}.
	\end{aligned}
\end{equation*}
For $0<a< \rho $ and  $0< \delta <a$, small enough, we introduce the sections
\begin{equation}
	\label{def:seccionsSigma0iSigma1}
	\begin{aligned}
\Sigma^0_{a,\delta} & = \{(u,v,t)\mid \; u = a, \; 0 < v < \delta, \; t \in \T\}	,
		 \\
 \Sigma^1_{a,\delta} & = \{(u,v,t)\mid \; 0 < u < \delta, \;  v = a, \; t \in \T\}.
\end{aligned}
\end{equation}
We assume that $\Sigma^0_{a,\delta}$ and $\Sigma^1_{a,\delta}$ are contained in $F(U)$, where $F$ and $U$ are given in Proposition~\ref{prop:varietatsredressades}
Let $\Psiloc$ be the Poincar\'e map from $\Sigma^1_{a,\delta}$ to $\Sigma^0_{a,\delta}$ induced by the flow of~\eqref{def:varietatsredressades}. The following theorem, which is the analogous to the classical $\lambda$-lemma to the parabolic case, describes the behavior of $\Psiloc$ at a $C^1$ level.

\begin{theorem}
	\label{thm:lambdalemma}
	There exist $0< \rho <1 $ and $C>0$  such that for any $0<a<\rho$ there exists $0<\delta < a/2$ such that the Poincar\'e map $\Psiloc: \Sigma^1_{a,\delta}\to \Sigma^0_{a,\delta^{1-Ca}}$ associated to system~\eqref{def:varietatsredressades} is well defined. Moreover,
	\begin{enumerate}
		\item[(1)]
		there exist $\wt C_1, \wt C_2 >0$ such that, for any $(u,a,t_0) \in \Sigma^1_{a,\delta}$,
		$(a,v_1,t_1)=\Psiloc(u,a,t_0)  $ satisfies
		\begin{equation*}
			\begin{aligned}
				u^{1+Ca} \le v_1 & \le u^{1-C a}, \\
				\wt C_1 u^{-(1-C a)/2} \le t_1 -t_0 & \le \wt C_2 u^{-(1+C a)/2},
			\end{aligned}
		\end{equation*}
		\item[(2)]
		if $\gamma:[0,\delta) \to \Sigma^1_{a,\delta}$,  is a $C^1$ curve of the form $\gamma(u) = (u,a,t_0(u))$ then if
		$(a, v_1(u),t_1(u))=\Psiloc (\gamma(u))  $, and $\delta$ is small enough, there exist $C_1, C_2>0$ such that
		\begin{equation*}
			\left|\frac{v_1'(u)}{t_1'(u)}\right|\le C_1 u^{1-Ca}, \qquad
			|t_1'(u)|  \ge C_2 \frac{1}{u^{9/10-Ca}}, \qquad u\in (0,\delta).
		\end{equation*}
	\end{enumerate}
\end{theorem}

The proof of this theorem is placed in Appendix~\ref{sec:provadelteoremathm:lambdalemma}.

\subsection{The local and the global maps}

Let $(u,v,t)$ be the coordinates given by Proposition~\ref{prop:varietatsredressades}.
Theorem~\ref{thm:teorema_principal_trencament_de_separatrius} ensures that $W^u(\Lambda_{\nu I_0})$ and $W^s(\Lambda_{\nu I_0})$ intersect transversally along two primary homoclinic orbits. Let $\Gamma$ be one of them. Clearly, $\Gamma$ intersects $\{u = a\}$ and $\{v=a\}$ at
two unique points, $p_h^0 = (a,0,t_h^0)$ and $p_h^1 = (0,a,t_h^1)$, respectively.
Besides the local map $\Psiloc: \Sigma^1_{a,\delta}\to \Sigma^0_{a,\delta}$ given by Theorem~\ref{thm:lambdalemma}, this transversal intersection allows us to define a global map $\Psiglob$ from $\Sigma^0_{a,\delta}$ to $\Sigma^1_{a,\delta}$, in neighborhoods of $p_h^0$ and $p_h^1$,  and then the return map $\Psi = \Psiloc \circ \Psiglob : \Sigma^0_{a,\delta} \to \Sigma^0_{a,\delta}$, as follows.

It will be convenient to choose two different sets of coordinates in neighborhoods of $p_h^0$ in $\{u=a\}$ and $p_h^1$ in $\{v=a\}$ to write the map $\Psiglob$ in an appropriate way. We proceed as follows.

\begin{enumerate}
\item Since
$W^s(\Lambda_{\nu I_0})$ intersects transversally both
$\{u = a\}$ and
$W^u(\Lambda_{\nu I_0})$, we have that, in a neighborhood of $p_h^0$, $W^s(\Lambda_{\nu I_0}) \cap \{u = a\} = \{(u,v,t) \mid \;u = a, \;t = \gamma^s(v)\}$ for some $C^{\infty}$~function $\gamma^s$ defined around $v=0$ with $\gamma^s(0) = t_h^0$.  We define new coordinates $(v,\tau)$ in a neighborhood of $p_h^0$ in $\{u=a\}$ by
\begin{equation}
\label{def:coordenadesA}
\begin{pmatrix}
v \\ t
\end{pmatrix} = A (v,\tau) = \begin{pmatrix}
v \\ \tau + \gamma^s (v)
\end{pmatrix}.
\end{equation}
In these coordinates, in a neighborhood of $p_h^0$, $W^s(\Lambda_{\nu I_0}) \cap \{u = a\} = \{\tau = 0\}$,
$W^u(\Lambda_{\nu I_0}) \cap \{u = a\} = \{v = 0\}$ and the point $p_h^0$ becomes $(0,0)$.
\item
Reasoning analogously, we have that, in a neighborhood of $p_h^1$,
$W^u(\Lambda_{\nu I_0}) \cap \{v = a\} = \{(u,v,t) \mid \;v = a, \;t = \gamma^u(u)\}$ for some $C^{\infty}$~function $\gamma^u$ defined around $u=0$ with $\gamma^u(0) = t_h^1$. We define new coordinates $(u,\tau)$ in a neighborhood of $p_h^1$ in $\{v=a\}$ by
\begin{equation}
\label{def:coordenadesB}
\begin{pmatrix}
u \\ t
\end{pmatrix} = B (u,\tau) = \begin{pmatrix}
u \\ \tau + \gamma^u (u)
\end{pmatrix}.
\end{equation}
In these coordinates, in a neighborhood of $p_h^1$, $W^u(\Lambda_{\nu I_0}) \cap \{v = a\} = \{\tau = 0\}$,
$W^s(\Lambda_{\nu I_0}) \cap \{v = a\} = \{u = 0\}$ and the point $p_h^1$ becomes $(0,0)$.
\end{enumerate}

Let
\begin{equation}
\label{def:wtPsiglob_wtPsiloc_wtPsi}
\wtPsiglob = B^{-1} \circ \Psiglob \circ A, \quad \wtPsiloc = A^{-1} \circ \Psiloc \circ B, \quad \wt \Psi = \wtPsiloc \circ \wtPsiglob = A^{-1} \circ \Psi \circ A
\end{equation}
be the expression of the maps $\Psiglob$, $\Psiloc$ and $\Psi$ is these coordinates, respectively.

\begin{proposition}
\label{prop:expansiowtPsiglob}
The map $\wtPsiglob$ satisfies
\[
\wtPsiglob (v,\tau)
= \begin{pmatrix}
\nu_0 \tau(1+\OO_1(v,\tau)) \\
\nu_1 v (1+\OO_1(v,\tau))
\end{pmatrix},
\]
for some $\nu_0 ,\nu_1$ such that   $\nu_0 \nu_1 \neq 0$.
\end{proposition}

\begin{proof}
Since $\Psiglob$ sends $W^s(\Lambda_{\nu I_0}) \cap \{u=a\}$ to $W^s(\Lambda_{\nu I_0}) \cap \{v=a\}$, we have that $\pi_u \wtPsiglob(v,0) = 0$.
Since $\Psiglob$ is $C^\infty$ we can write  $\pi_u \wtPsiglob(v,\tau) = \nu_0 \tau(1+\OO_1(v,\tau))$ for some $\nu_0$. Also, since $\Psiglob$ sends $W^u(\Lambda_{\nu I_0}) \cap \{u=a\}$ to $W^u(\Lambda_{\nu I_0}) \cap \{v=a\}$, we have that $\pi_\tau \wtPsiglob(0,\tau) = 0$, which implies that $\pi_\tau \wtPsiglob(v,\tau) = \nu_1 v(1+\OO_1(v,\tau))$ for some $\nu_1$. Since $\wtPsiglob$ is a diffeomorphism, $\nu_0 \nu_1 \neq 0$.
\end{proof}

Let $\rho_0 >0$ be such that Theorem~\ref{thm:lambdalemma} holds. For $0< \delta < a \le  \rho_0$, we consider the rectangle $\QQ_\delta \subset \Sigma^0_{a,\delta}$ given, in the new coordinates $(v,\tau)$, by
\begin{equation*}
\QQ_\delta = \{(v,\tau) \mid \; 0 < v,\tau < \delta\}.
\end{equation*}
Observe that the boundaries of the rectangle $\{v=0\}$ and $\{\tau = 0\}$ are pieces of the unstable and stable manifolds of $\Lambda_{\nu I_0}$, respectively, while the corner $(v,\tau) = (0,0)$ is the homoclinic point.

\subsection{Symbolic dynamics. Proof of Theorem~\ref{thm:teoremaprincipalenvariablesoriginals}}

In this section we establish the existence of a subset $\Sigma \subset \QQ_\delta$ such that $\wt \Psi: \Sigma \to \Sigma$ is conjugate to the shift of infinite symbols. We proceed as in \cite{Moser01}. To begin with, we introduce the already classical concepts of \emph{horizontal} and \emph{vertical strips} and \emph{unstable} and \emph{stable} cones.

We will say that $H$ is a ``horizontal'' strip in $\QQ_\delta$ if
\begin{equation}
\label{def:banda_horitzontal}
H = \{(v,\tau) \mid \; h^-(\tau) \le v \le h^+(\tau)\},
\end{equation}
where $h^\pm:[0,\delta] \to [0,\delta]$ are $\mu_h$-Lipschitz.
Analogously, we will say that $V$ is a ``vertical'' strip in $\QQ_\delta$ if
\begin{equation}
\label{def:banda_vertical}
V = \{(v,\tau) \mid \; v^-(v) \le \tau \le v^+(v)\},
\end{equation}
where $v^\pm:[0,\delta] \to [0,\delta]$ are $\mu_v$-Lipschitz.

For the horizontal strip $H$ in~\eqref{def:banda_horitzontal}, we split $\partial H = \partial_h H \cup \partial_v H$, where
\[
\begin{aligned}
\partial_h H & = \{(v,\tau)\in \overline{\QQ_\delta}\mid \, v = h^-(\tau),\; 0 \le \tau \le \delta\} \cup \{(v,\tau)\in \overline{\QQ_\delta}\mid \, v = h^+(\tau),\; 0 \le \tau \le \delta\}, \\
\partial_v H & = \{(v,\tau)\in \overline{\QQ_\delta}\mid \, \tau = 0,\; h^-(0) \le v \le h^+(0) \} \cup \{(v,\tau)\in \overline{\QQ_\delta}\mid \, \tau = \delta,\; h^-(\delta) \le v \le h^+(\delta) \}
\end{aligned}
\]
and, analogously, for the vertical strip $V$ in~\eqref{def:banda_vertical}, we split $\partial V = \partial_h V \cup \partial_v V$, where
\[
\begin{aligned}
\partial_h V & = \{(v,\tau) \in \overline{\QQ_\delta} \mid \; v = 0, \; v^-(0) \le \tau \le v^+(0)\} \cup \{(v,\tau)\in \overline{\QQ_\delta} \mid \; v = \delta, \; v^-(\delta) \le \tau \le v^+(\delta) \}, \\
\partial_v V & = \{(v,\tau)\in \overline{\QQ_\delta} \mid \; \tau = v^-(v),\; 0 \le v \le \delta\} \cup \{(v,\tau)\in \overline{\QQ_\delta}\mid \; \tau = v^+(v),\; 0 \le v \le \delta\}.
\end{aligned}
\]

Given $z = (v,\tau) \in \QQ_\delta$, we consider the basis of $T_z\QQ_\delta$ given by $\partial/\partial v$ and $\partial/\partial \tau$. Using this basis, given $\eta >0$, we define the $\eta$-unstable cone at $z$ as
\begin{equation}
\label{def:coninestable}
C^u_{z,\eta} = \{(V,T) \in T_z\QQ_\delta \mid \; |V| \le \eta |T|\}
\end{equation}
and the $\eta$-stable cone at $z$ as
\begin{equation}
\label{def:conestable}
C^s_{z,\eta} = \{(V,T) \in T_z\QQ_\delta \mid \; |T| \le \eta |V|\}.
\end{equation}

Following~\cite{Moser01}, we introduce the following hypotheses. Let $F:\QQ_\delta \to \R^2$ be a $C^1$
diffeomorphism onto its image.
\begin{enumerate}
\item[\fH]
There exists two families $\{H_n\}_{n\in \N}$, $\{V_n\}_{n\in \N}$ of horizontal and vertical rectangles in $\QQ_\delta$, with $\mu_h \mu_v < 1$, such that
$H_n \cap H_{n'} = \emptyset$, $V_n \cap V_{n'} = \emptyset$, $n\neq n'$, $H_n \to \{v = 0\}$, $V_n \to \{\tau = 0\}$, when $n\to \infty$, with respect to the Hausdorff distance, $F(V_n) = H_n$, homeomorphically, $F(\partial_v V_n) \subset \partial_v H_n$ and $F(\partial_h V_n) \subset \partial_h H_n$, $n\in \N$.
\item[\sH]
There exist $\eta_u,\eta_s,\kappa >0$ satisfying $0 < \kappa < 1 - \eta_u \eta_s$ such that if $z \in \cup_{n} V_n$, $DF(z) C^u_{z,\eta_u} \subset C^u_{F(z),\eta_u}$, if $z \in \cup_{n} H_n$, $DF^{-1}(z) C^s_{z,\eta_s} \subset C^s_{F^{-1}(z),\eta_s}$ and,
denoting $x^+ = DF(z)x$ and $x^- = DF^{-1}(z)x$, if $x_u \in C^u_{z,\eta_u}$, $|x_u^+| \ge \kappa^{-1} |x_u|$ and if $x_s \in C^s_{z,\eta_s}$, $|x_s^-| \ge \kappa^{-1} |x_s|$.
\end{enumerate}

We introduce symbolic dynamics in the following way. Consider the \emph{space of sequences} $S = \N^{\Z}$, with the topology induced by the neighborhood basis of $s^* = (\dots,s_{-1}^*,s_0^*,s_1^*,\dots)$
\[
I_j(s^*) = \{ s \in S \mid \; s_k = s_k^*, \; |k| < j\}, \qquad s^* \in S,
\]
and the \emph{shift map} $\sigma: S \to S$ defined by $\sigma(s)_j = s_{j+1}$. The map $\sigma$ is a homeomorphism.

The combination of Theorems~3.1 and~3.2 in~\cite{Moser01} can be rephrased as follows.

\begin{theorem}[Moser]
\label{thm:Moser}
Let $F:\QQ_\delta \to \R^2$ be a $C^1$
diffeomorphism onto its image. If $F$ satisfies hypotheses \fH and \sH, then $F$ has the shift as a subsystem, that is, there exists $\II \subset \QQ_\delta$, invariant by $F$, such that $F_{\mid \II}$ is topologically conjugate to the shift $\sigma$ on $S$.
\end{theorem}

The following result checks that the map $\wt \Psi $ satisfies the hypotheses of Theorem~\ref{thm:Moser}.

\begin{theorem}
\label{thm:conjugacioambelshift}
If $\delta$ is small enough, $\wt \Psi: \QQ_\delta \to \R^2$ satisfies hypotheses \fH and \sH.
\end{theorem}

The proof of Theorem~\ref{thm:conjugacioambelshift} is a consequence of the following proposition.

\begin{proposition}
\label{prop:diferencial_del_mapa_de_retorn}
If $\delta$ is small enough, for all $z \in \QQ_\delta \cap \wt \Psi^{-1}(\QQ_\delta)$ we have that
\[
\begin{aligned}
D\wt \Psi (z) & = \begin{pmatrix}
\lambda^{-1} (\nu_1+\OO(\delta))+\lambda \nu_0 \OO(\delta^{2(1-Ca)}) & \lambda \OO(\delta^{1-Ca}) \\
\nu_1 \lambda^{-1} \OO(\delta^{1-Ca}) + \lambda \nu_0  \OO(\delta^{1-Ca})& \lambda (\nu_0 + \OO(\delta))
\end{pmatrix}, \\
D\wt \Psi^{-1} (z) & = \begin{pmatrix}
\lambda (\nu_1^{-1} + \OO(\delta)) & \nu_0^{-1} \lambda^{-1} \OO(\delta^{1-Ca}) + \lambda \nu_1^{-1}  \OO(\delta^{1-Ca}) \\
\lambda \OO(\delta^{1-Ca}) & \lambda^{-1} (\nu_0^{-1}+\OO(\delta))+\lambda \nu_1^{-1} \delta^{2(1-Ca)}
\end{pmatrix},
\end{aligned}
\]
where $\lambda > C \delta^{-9/10+Ca}$.
\end{proposition}

\begin{proof} Using definition~\eqref{def:wtPsiglob_wtPsiloc_wtPsi}, we have that
\[
D \wt \Psi  = D \wtPsiloc D \wtPsiglob = DA^{-1} D \Psiloc D B \,D \wtPsiglob,
\]
where the derivatives are evaluated at the corresponding points according to the chain rule.
Hence, by Proposition~\ref{prop:expansiowtPsiglob}, the definitions of $A$ and $B$ in~\eqref{def:coordenadesA} and~\eqref{def:coordenadesB}, respectively, and Theorem~\ref{thm:lambdalemma}, we have that, for all $z \in \QQ_\delta$,
\begin{multline*}
D \wt \Psi(z)  \begin{pmatrix}
0 \\1
\end{pmatrix} = DA^{-1} D \Psiloc D B \begin{pmatrix}
\nu_0 + \OO(\delta) \\ \OO(\delta)
\end{pmatrix}
=   (\nu_0 + \OO(\delta)) DA^{-1} D \Psiloc \begin{pmatrix}
1 \\ \OO(\delta)
\end{pmatrix} \\
=  \lambda (\nu_0 + \OO(\delta)) DA^{-1}  \begin{pmatrix}
\OO(\delta^{1-Ca}) \\ 1
\end{pmatrix}
= \lambda (\nu_0 + \OO(\delta))   \begin{pmatrix}
\OO(\delta^{1-Ca}) \\ 1
\end{pmatrix},
\end{multline*}
where $\lambda > C \delta^{-9/10+Ca}$. This proves the formula for the second column of $D \wt \Psi $.

Applying the same argument to $D\wt \Psi^{-1}$, we obtain
\[
(D \wt \Psi )^{-1}(z)  \begin{pmatrix}
1 \\0
\end{pmatrix} = D\wt \Psi^{-1}(\wt \Psi(z))\begin{pmatrix}
1 \\0
\end{pmatrix} =   \lambda (\nu_1^{-1} + \OO(\delta))   \begin{pmatrix}
1 \\ \OO(\delta^{1-Ca})
\end{pmatrix},
\]
that is, there exists $T(z)$, with $|T(z)| = \OO(\delta^{1-Ca})$ such that
\[
D \wt \Psi (z) \begin{pmatrix}
1 \\ T(z)
\end{pmatrix} = \lambda^{-1} (\nu_1 + \OO(\delta))   \begin{pmatrix}
1 \\ 0
\end{pmatrix}.
\]
From this equality one immediately obtains the expressions of $D\wt \Psi$ and $D\wt \Psi^{-1}$.
\end{proof}

\begin{proof}[Proof of Theorem~\ref{thm:conjugacioambelshift}]
In this proof, the norm of a vector $(V,T)^\top $ will be $\|(V,T)\| = |V|+|T|$.

We only need to check that $\wt \Psi$, if $\delta$ is small enough, satisfies hypotheses \fH\ and \sH\ on $\QQ_\delta$ to apply Moser's Theorem~\ref{thm:Moser}.

\fH\ follows immediately from (1) of Theorem~\ref{thm:lambdalemma}. Indeed, the image of $\QQ_\delta$ by $\wt \Psi$ is a countable union of horizontal strips. If $\delta$ is small enough, different strips have empty intersection. Their preimages provide the vertical family.

Now we check \sH, which also implies the conditions on the Lipschitz constants of the boundaries of the strips. To do so, we consider  the cones
$C^u_{z,\eta^u}$ and $C^s_{z,\eta^s}$, introduced in~\eqref{def:coninestable} and~\eqref{def:conestable}, with $\eta^u = \eta^s = \eta =  \OO(\delta^{1-Ca})$. Now, if $(V,T)^\top \in C^u_{z,\eta}$, by Proposition~\ref{prop:diferencial_del_mapa_de_retorn}, denoting $(V^+,T^+)^\top = D\wt \Psi(z)(V,T)^\top$, we have first that, if $\delta$ is small enough,
\begin{equation}
\label{fita:TmesperT}
|T^+| \ge \lambda (\nu_0 - \OO(\delta)) \left(1- \OO(\delta^{1-Ca})\eta\right) |T|,
\end{equation}
which implies that
\[
\begin{aligned}
|V^+| & \le \left(\left(\lambda^{-1} (\nu_1+\OO(\delta))+\lambda \nu_0 \OO(\delta^{2(1-Ca)})\right)\eta + \lambda \OO(\delta^{1-Ca})\right) |T| \\
& \le \left(\left(\lambda^{-2} (\nu_1\nu_0^{-1}+\OO(\delta))+ \OO(\delta^{2(1-Ca)})\right)\eta +  \OO(\delta^{1-Ca})\right)|T^+|.
\end{aligned}
\]
Hence, since $\eta = \OO(\delta^{1-Ca})$, $(V^+,T^+)^\top \in C^u_{\wt \Psi(z),\eta}$. Furthermore, using again~\eqref{fita:TmesperT}, we have that
\[
\|(V^+,T^+)\| \ge (1-\eta) |T^+| \ge (1-\eta) \lambda (\nu_0 - \OO(\delta)) \left(1- \OO(\delta^{1-Ca})\eta\right) |T| \ge
\lambda (\nu_0 - \OO(\delta^{1-Ca})) \|(V,T)\|.
\]
An analogous argument, using the bounds for $D \wt \Psi^{-1}$ in Proposition~\ref{prop:diferencial_del_mapa_de_retorn}, provides the corresponding claim for
$C^s_{z,\eta}$ with $\eta = \OO(\delta^{1-Ca})$. In this case, the expansion factor is  $\lambda (\nu_1^{-1} + \OO(\delta^{1-Ca}))$. Clearly, if $\delta$ is small enough,
the condition $0< \kappa < 1- \eta _u \eta_s$ holds because
\[
0 \le \max\{\lambda^{-1} (\nu_0 + \OO(\delta^{1-Ca}))^{-1}, \lambda^{-1} (\nu_1^{-1} + \OO(\delta^{1-Ca}))^{-1}\} \le 1 - \eta^2.
\]
In conclusion, \sH\ also holds. Then, Theorem~\ref{thm:Moser} applies and implies Theorem~\ref{thm:conjugacioambelshift}. This finishes the proof of Theorem~\ref{thm:teoremaprincipalenvariablesoriginals}.
\end{proof}

\section{Hamilton-Jacobi equation}
\label{sec:HJouter}

\subsection{Notation}

Here we introduce some notations we will use hereafter.
Given a normed space $\X$, we will denote $\B_{\rho}\subset \X$, the ball of radius $\rho$ in $\X$ centered at zero.

We recall that $\T = \R/ 2\pi\Z$. The complex torus of width $\sigma$ is
\[
\T_\sigma = \{\theta \in \C \mid\, |\Im \theta| < \sigma\}.
\]

Given $U \subset \C$ and $f:U \times \T \to \C$,  its average will be denoted $\langle f \rangle(u) = \frac{1}{2\pi} \int_0^{2\pi} f(u,\theta) \, d\theta$.
Moreover, $f^{[k]}(u)$ will denote its $k$-th Fourier coefficient,
\[
f^{[k]}(u) = \frac{1}{2\pi} \int_0^{2\pi} f(u,\theta) e^{-i k \theta}\,d\theta.
\]

The letters $K$ and $C$ will represent generic constants that may take different values at different places, even in the same chain of inequalities.

\subsection{Deriving the Hamilton-Jacobi equation}

It is easy to check that the invariant manifolds of $\Lambda_{\nu I_0}$ are Lagrangian. Here, we do not prove this claim in a direct way, but, instead, we look for them as solutions of a suitable Hamilton-Jacobi equation, which will then imply that  indeed they are Lagrangian.  In view of~\eqref{def:Gammapm}, we expect to find the invariant manifolds close to the separatrix~\eqref{def:Gamma0}. Hence, we first introduce the new coordinates $(u,P,J,\theta)$, adapted to the separatrix, defined through
\begin{equation}
\label{def:variablesuP}
q  = q_h(u), \qquad
p  = \frac{P}{p_h(u)},
\end{equation}
where the functions $q_h$ and $p_h$ where introduced in Lemma~\ref{lem:homoclinic}. An immediate computation shows that $2$-form~\eqref{def:rescaledbsymplecticform} is transformed into the standard form $du\wedge dP + d\theta\wedge dJ$.  Hamiltonian $H$ in~\eqref{def:rescaledHamiltonianI0} becomes
\begin{equation*}
\wt H(u,P,\theta,J) = \wt H_0(u,P,J)+ \wt H_1(u,\theta),
\end{equation*}
where
\begin{equation*}
\begin{aligned}
\wt H_0(u,P,J) & = H_0(q_h(u),p_h(u)^{-1} P, J) = \frac{1}{2}\left(\nu (I_0+J)^2+\frac{(1+u^2)^2}{u^2}P^2\right)-\frac{1}{2}\frac{u^2}{(1+u^2)^2}, \\
\wt H_1(u,\theta) & = H_1(q_h(u),\theta) = \frac{1}{2}\frac{1}{(1+u^2)^2} V(\theta).
\end{aligned}
\end{equation*}
Using that in the new variables we are dealing with the standard form, we will look for the manifolds as graphs over $(u,\theta)$, that is,
\begin{equation}
\label{def:PJcomagrafica}
P = \partial_u \Phi^\pm (u,\theta), \qquad J = \partial_\theta \Phi^\pm (u,\theta),
\end{equation}
where $\Phi^\pm$ are solutions of the Hamilton-Jacobi equation
\begin{equation}
\label{def:HJoriginal}
\wt H \left(u,\partial_u \Phi^\pm (u,\theta), \theta, \partial_\theta \Phi^\pm (u,\theta)\right) = \frac{1}{2} \nu I_0^2,
\end{equation}
with boundary conditions
\begin{equation}
\label{def:condicions_de_contorn}
\lim_{u \to \mp \infty} p_h(u)^{-1} \partial_u \Phi^\pm (u,\theta) = 0.
\end{equation}

\begin{remark}
\label{rem:varietat_estable}
We will find a real analytic solution of~\eqref{def:HJoriginal}, $\Phi^+(u,\theta)$, that will provide a parametrization of the unstable manifold, that is
$\lim_{u\to -\infty} p_h(u)^{-1}\partial_u \Phi^+(u,\theta) = 0$. Since, by assumption,  $V (\theta)$ is even, the function defined by
\[
\Phi^-(u,\theta) = - \Phi^+(-u,-\theta)
\]
is also a solution of~\eqref{def:HJoriginal}, satisfying the boundary condition $\lim_{u\to \infty} p_h(u)^{-1}\partial_u \Phi^-(u,\theta) = 0$. Hence, it provides
a parametrization of the stable manifold.
\end{remark}

It is immediate to check that, up to additive constants, the solution of the \emph{unperturbed} Hamilton-Jacobi equation
\begin{equation*}
\wt H_0 \left(u,\partial_u \Phi^\pm (u,\theta), \theta, \partial_\theta \Phi^\pm (u,\theta)\right) = \frac{1}{2} \nu I_0^2,
\end{equation*}
with boundary conditions~\eqref{def:condicions_de_contorn} is
\begin{equation}
\label{def:Phi0}
\Phi_0(u) = -\frac{u}{2(u^2+1)} + \frac{1}{2}{\arctan u}.
\end{equation}

Introducing the new unknown $\Phi_1$ by $\Phi = \Phi_0+\Phi_1$,
the Hamilton-Jacobi equation~\eqref{def:HJoriginal} becomes
\begin{equation}
\label{eq:HJT1}
\partial_u \Phi_1(u,\theta) + \nu I_0 \partial_\theta \Phi_1(u,\theta) + \frac{1}{2p_h(u)^2} \partial_u \Phi_1(u,\theta)^2
+ \frac{\nu}{2} \partial_\theta \Phi_1(u,\theta)^2 + H_1(q_h(u),\theta) = 0.
\end{equation}
Introducing the operators
\begin{equation}
\label{def:diffopL}
\LL (\Phi_1) = \partial_u \Phi_1 + \nu I_0 \partial_\theta \Phi_1
\end{equation}
and
\begin{equation}
\label{def:F}
\F (\Phi_1)(u,\theta) =  -\frac{1}{2p_h(u)^2} \partial_u \Phi_1(u,\theta)^2
- \frac{\nu}{2} \partial_\theta \Phi_1(u,\theta)^2 - H_1(q_h(u),\theta),
\end{equation}
equation~\eqref{eq:HJT1} becomes
\begin{equation}
\label{eq:HJT1v2}
\LL(\Phi_1) = \F (\Phi_1).
\end{equation}
We remark that
\begin{equation}
\label{def:F0outer}
\F(0)(u,\theta) = -H_1(q_h(u),\theta) = -\frac{1}{2}\frac{1}{(1+u^2)^2} V(\theta).
\end{equation}
To solve~\eqref{eq:HJT1v2}, we will
rewrite it as a fixed point equation by means of a suitable right inverse $\G$ of $\LL$, to be defined later on, that is,
\begin{equation}
\label{eq:HJ1comunpuntfix}
\Phi_1 = \G \circ \F (\Phi_1).
\end{equation}

We will devote the rest of the section to prove the existence of a solution of~\eqref{eq:HJT1} with the boundary conditions as $u\to -\infty$ in \eqref{def:condicions_de_contorn}.

\subsection{Definitions and technical lemmas}

Let $0 < \beta_1 < \beta_2 <\pi/2$ be fixed. For $\kappa >1$ and $\delta \in (0,1/2)$, we consider the complex domain
\begin{equation}\label{def:DominisRaros}
\D^+_{\kappa,\delta}=
\left\{u \in\C\mid \,\right.\left. |\Im u|< - \tan\beta_1\, \Re
u+1-\kappa (\nu I_0)^{-1},
\; |\Im u|> \tan\beta_2\, \Re u+1-\delta\right\}.
\end{equation}
We will be interested in the case where $I_0$ and $\kappa$ are big but satisfy $\kappa (\nu I_0)^{-1} < \delta$. Observe that,  if $\kappa (\nu I_0)^{-1} < \delta$, then
$d(\D^+_{\kappa,\delta}, \pm i) =  C \kappa (\nu I_0)^{-1}$ with $C= \cos \beta_1$.
See Figure~\ref{fig:DomRaro}.

\begin{figure}[h]
\begin{center}
\includegraphics[width=.5\textwidth]{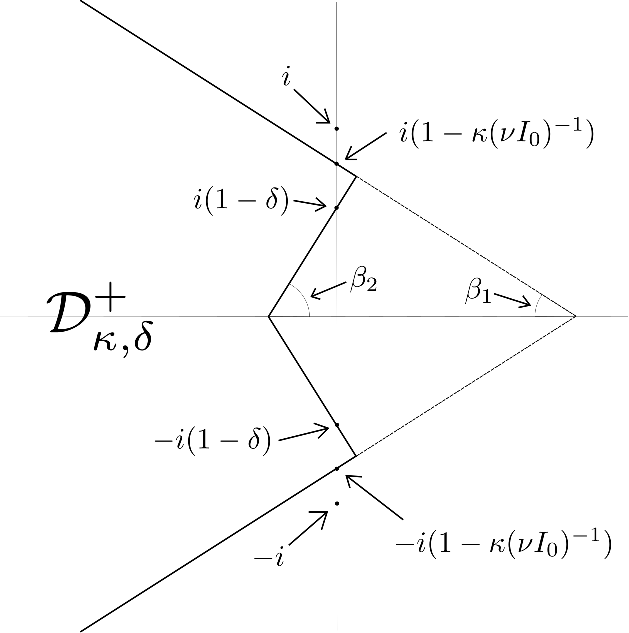}
\end{center}
\caption{The domain  $\D^+_{\kappa,\delta}$ defined in
\eqref{def:DominisRaros}.}\label{fig:DomRaro}
\end{figure}

To solve equation~\eqref{eq:HJT1v2}, for $r,s\in \R$ we introduce the Banach space of real analytic, $2\pi$-periodic in~$\theta$, functions
\begin{equation*}
\X_{r,s} = \{\Psi:\D^+_{\kappa,\delta} \times \T_\sigma \to \C \mid  \, \textrm{$\Psi$ real analytic},\; \|\Psi\|_{r,s} < \infty\}
\end{equation*}
and, taking into account the Fourier expansion $\Psi(u,\theta) = \sum_{k\in \Z} \Psi^{[k]} (u) e^{ik\theta}$, the norm
\[
\|\Psi\|_{r,s} = \sum_{k\in \Z} \|\Psi^{[k]}\|_{r,s} e^{|k|\sigma},
\]
where, for an analytic function $f:\D_{\kappa,\delta}^+ \to \C$,
\[
\| f\|_{r,s} = \max\left\{\sup_{u \in \D^+_{\kappa,\delta}, \Re u \le -u_0} |u^{r} f(u)|,\sup_{u \in \D^+_{\kappa,\delta}, \Re u > -u_0} |(1+u^2)^{s} f(u)|\right\},
\]
where $u_0\in \R$ is chosen such that $u_0 \ge \max \{1, (1-\delta) /\tan \beta_2\}$.

\begin{lemma}
\label{lem:propietatsespaisXrs}
Let $r, r_1, r_2, s, s_1, s_2 \in \R$. We have
\begin{enumerate}
\item[(1)]
If $\Psi \in \X_{r+t_1,s+t_2}$ with  $t_1,t_2 \ge 0$ then $\Psi \in \X_{r,s}$ and
\[
\|\Psi\|_{r,s} \le K \left(\frac{\nu I_0}{\kappa}\right)^{t_2}\|\Psi\|_{r+t_1,s+t_2},
\]
where $K>0$ is independent of $\nu I_0$ and $\kappa$.
\item[(2)]
If $\Psi_1 \in \X_{r_1,s_1}$ and $\Psi_2 \in \X_{r_2,s_2}$ then  $\Psi_1\Psi_2 \in \X_{r_1+r_2,s_1+s_2}$ and
\[
\|\Psi_1 \Psi_2 \|_{r_1+r_2,s_1+s_2} \le \|\Psi_1  \|_{r_1,s_1}\|\Psi_2 \|_{r_2,s_2}.
\]
\end{enumerate}
\end{lemma}

The proof of the previous lemma is straightforward from the ideas in~\cite{Sauzin01}.

We also introduce
\begin{equation*}
\widetilde \X_{r,s} = \{\Psi\in \X_{r,s} \mid  \,
\partial_{u} \Psi\in \X_{r+1,s+1}, \, \partial_{\theta} \Psi\in \X_{r+1,s+1},\,
\lln \Psi \rrn_{r,s} < \infty\},
\end{equation*}
where
\begin{equation*}
\lln \Psi \rrn_{r,s} = \| \Psi\|_{r,s} + \| \partial_u \Psi\|_{r+1,s+1}+\nu I_0 \|\partial_\theta \Psi\|_{r+1,s+1}.
\end{equation*}

Given $\Psi \in \X_{r,s}$, with $r> 0$, $s\ge 0$, we formally define
\begin{equation}
\label{def:Gu}
\G^u(\Psi)(u,\theta) = \int_{-\infty}^0 \Psi (u+\xi, \theta+ \nu I_0 \xi) \, d\xi
= \int_{-\infty}^u \Psi (\xi, \theta+ \nu I_0 (\xi-u)) \, d\xi.
\end{equation}
Clearly, when $\G^u(\Psi)$ is well defined, if $\Psi$ is real analytic, so is $\G^u(\Psi)$.
The operator $\G^u$ formally satisfies $\LL \circ \G^u(\Psi) = \Psi$, where $\LL$ is the differential operator introduced in~\eqref{def:diffopL}.

The next technical lemma will be the main tool to find the desired solutions of equation \eqref{eq:HJ1comunpuntfix}. Its proof is an immediate variation of the arguments in~\cite{GuardiaOS10}.

\begin{lemma}
\label{lem:operadorGu}
Let $\G^u$ be the operator defined in~\eqref{def:Gu}.
\begin{enumerate}
\item[(1)]
If $\Psi \in \X_{r,s}$, with $r>1$, $s\ge 0$, then $\G^u(\Psi) \in \X_{r-1,s}$ and
\[
\|\G^u(\Psi)\|_{r-1,s} \le K \|\Psi\|_{r,s}.
\]
If $r, s >0$ and $\langle \Psi\rangle = 0$,  then $\G^u(\Psi)\in \X_{r,s}$ and
\[
\|\G^u(\Psi)\|_{r,s} \le K (\nu I_0)^{-1}\|\Psi\|_{r,s}.
\]
\item[(2)]
If $\Psi \in \X_{r,s}$, with $r,s> 1$, $\G^u(\Psi)\in \X_{r-1,s-1}$ and
\[
\|\G^u(\Psi)\|_{r-1,s-1} \le K \|\Psi\|_{r,s}.
\]
\item[(3)]
If $\Psi \in \X_{r,s}$, with $r\ge 1$, $s>0$, $\partial_u \G^u(\Psi), \partial_\theta \G^u(\Psi) \in \X_{r,s}$ and
\[
\|\partial_u \G^u(\Psi)\|_{r,s} \le K \|\Psi\|_{r,s}, \qquad
\|\partial_\theta \G^u(\Psi)\|_{r,s} \le K (\nu I_0)^{-1} \|\Psi\|_{r,s}.
\]
\end{enumerate}
As a consequence, if $\Psi \in \X_{r,s}$, with $r,s> 1$, $\G^u(\Psi) \in \wt \X_{r-1,s-1}$ and
\[
\lln \G^u(\Psi) \rrn_{r-1,s-1} \le K \|\Psi\|_{r,s}.
\]
The constant $K$ only depends on $r,s$ and the constants involved in the definition of $\D^+_{\kappa,\delta}$ but it is independent of $\nu I_0$ and $\kappa$.
\end{lemma}

\subsection{Solution of the Hamilton-Jacobi equation~\eqref{def:HJoriginal}}

Let $\F$ and $\G^u$ be the operators defined by~\eqref{def:F} and~\eqref{def:Gu}, respectively.
We formally introduce
\begin{equation}
\label{def:Loutmes}
\begin{aligned}
\LLo(u,\theta)  &=\G^u \circ \F (0)(u,\theta) =
-\int_{-\infty}^{u} H_1(q_h(s),\theta + \nu I_0 (s-u))\, ds \\
& = -\int_{-\infty}^u \frac{1}{2(1+ s^2)^2} V(\theta- \nu I_0 (s-u)) \, ds.	
\end{aligned}
\end{equation}
For later use we also introduce
\[
\LLom(u,\theta)  =-\LLo(-u,-\theta)
 = \int_u^{\infty} \frac{1}{2(1+ s^2)^2} V(\theta+\nu I_0 (s-u)) \, ds	
\]
and
\[
L_{\mathrm{out}}(u,\theta) = \LLo(u,\theta)-\LLom(u,\theta) =
- \int_{-\infty}^{\infty} \frac{1}{2(1+ s^2)^2} V(\theta+\nu I_0 (s-u)) \, ds	.
\]
Observe that $L_{\mathrm{out}} = L $,
where $L$ is the Melnikov potential introduced in~\eqref{def:Melnikovpotential}.

\begin{proposition}
\label{prop:primeraiteracio}
The function $\LLo  \in \X_{4,2}$ in $\D_{\kappa,\delta}^+$ and satisfies
\[
\|\LLo\|_{4,2}, \|\partial_\theta \LLo\|_{4,2}, \|\partial_u \LLo\|_{5,3} \le \frac{K}{\nu I_0},
\]
for some $K>0$, independent of $\nu I_0$ and $\kappa$.
Furthermore, $\langle \LLo \rangle = 0$.
\end{proposition}

\begin{proof}
We recall that
\[
\F(0) (u,\theta) = -H_1(q_h(u),\theta) = -\frac{1}{2} q_h^4(u
) V(\theta) = -\frac{1}{2} \frac{1}{(1+u^2)^2}
V(\theta),
\]
where $H_1$ was introduced in~\eqref{def:H0H1} and the function $V$ in~\eqref{def:potencialdeMorsecorrugat}.
Clearly, $\F(0) \in \X_{4,2}$ and $\langle \F(0) \rangle = 0$. The bounds on $\|\LLo\|_{4,2}$ and on $\|\partial_\theta \LLo\|_{4,2}$ follow from (1) and (3)  of Lemma~\ref{lem:operadorGu}, respectively.
To obtain the bound on $\|\partial_u \LLo\|_{5,3}$ we observe that, by integration by parts, if $k\neq 0$,
\[
\partial_u \int_{-\infty}^u \frac{1}{2} \frac{1}{(1+s ^2)^2} e^{ik\nu I_0(s -u)}
e^{ik \theta}
\, ds
= - 2 \int_{-\infty}^u \frac{s }{(1+s ^2)^3} e^{ik\nu I_0(s -u)}e^{ik \theta}\, ds .
\]
Hence, defining
\[
g (u,\theta) = -\frac{2u}{(1+u^2)^3}
V(\theta),
\]
we have that $\partial_u \LLo = \G^u(g)$. Since $\langle g \rangle = 0$ and $g \in \X_{5,3}$, using again (1) of Lemma~\ref{lem:operadorGu},
$\|\partial_u \LLo\|_{5,3} = \|\G^u(g)\|_{5,3} \le K (\nu I_0)^{-1} \|g\|_{5,3}$.

The last claim follows from the fact that $\langle \F(0) \rangle = 0$.
\end{proof}

We use $\LLo$ to rewrite equation~\eqref{eq:HJ1comunpuntfix} as a new fixed point equation, with better control of its solution.
To do so, we introduce $\Phi_2$ by $\Phi_1 = \LLo + \Phi_2$. Then, $\Phi_1$ is a solution of~\eqref{eq:HJ1comunpuntfix} if and only if $\Phi_2$ satisfies
\begin{equation*}
\Phi_2 = \G^u \circ \wt \F (\Phi_2),
\end{equation*}
where
\begin{equation}
\label{def:wtFequacioouter}
 \wt \F (\Phi_2)  =  \F(\LLo + \Phi_2) - \F(0).
\end{equation}

\begin{proposition}
\label{prop:aproximacioperlahomoclinicaversio2}
There exists $K^*>0$ such that, if $\kappa$ is big enough, the
operator $\G^u \circ \wt \F : \B_{K^* (\nu I_0)^{-2}} \subset \widetilde \X_{5,3} \to \B_{K^* (\nu I_0)^{-2}} $ is well defined and a contraction. Let $\Phi_2^+$ be its fixed point.
\end{proposition}
As a consequence, since $\Phi_2^+$ is the fixed point of $\G^u \circ \wt \F$, then $\Phi_1^+ = \LLo + \Phi_2^+$ is a solution of equation~\eqref{eq:HJT1v2}.
\begin{proof}
We first claim that $\wt \F (0) \in \X_{6,4}$ and $\|\wt \F (0)\|_{6,4} \le  K_1 (\nu I_0)^{-2}$, for some $ K_1$ independent of $\nu I_0$. Indeed,
by the definitions of $\F$ in~\eqref{def:F} and $\wt F$ in~\eqref{def:wtFequacioouter}, using Lemma~\ref{lem:propietatsespaisXrs},
Proposition~\ref{prop:primeraiteracio}, the fact that $p_h^{-2} \in \X_{-2,-2}$ and
\[
\|\partial_\theta \LLo\|_{2,2} \le  \|\partial_\theta \LLo\|_{4,2} \le \frac{K}{\nu I_0},
\]
we have that
\[
\begin{aligned}
\|\wt \F (0)\|_{6,4} & = \left\|\frac{1}{2 p_h^2} (\partial_u \LLo)^2 + \frac{\nu}{2} (\partial_\theta \LLo)^2\right\|_{6,4} \\
& \le \left\|\frac{1}{2 p_h^2} \partial_u \LLo\right\|_{1,1} \|\partial_u \LLo\|_{5,3} +
\frac{\nu}{2} \|\partial_\theta \LLo\|_{2,2} \|\partial_\theta \LLo\|_{4,2} \\
& \le \frac{K}{2} \|p_h^{-2}\|_{-2,-2}  \|\partial_u \LLo\|_{3,3} \frac{1}{\nu I_0} +
\frac{\nu K}{2} \|\partial_\theta \LLo\|_{2,2} \frac{1}{\nu I_0}\\
& \le \frac{K^2}{2} \|p_h^{-2}\|_{-2,-2}   \frac{1}{(\nu I_0)^2} + \frac{\nu K^2}{2(\nu I_0)^2}  \\
& \le \frac{ K_1}{(\nu I_0)^2}.
\end{aligned}
\]

Hence, by the last claim of Lemma~\ref{lem:operadorGu},  $\lln\G^u \circ \wt \F(0)\rrn_{5,3}  \le K \| \wt \F(0) \|_{6,4}  \le \frac{K K_1}{(\nu I_0)^2} < \infty$.

We take $$K^* = 2 K K_1.$$

Let $\Psi \in \B_{K^* (\nu I_0)^{-2}} \subset \wt \X_{5,3}$. In view of~\eqref{def:F} and~\eqref{def:wtFequacioouter}, we write
\begin{equation*}
\wt \F (\Psi) = \wt \F(0) + \wt \F(\Psi) - \wt \F(0) = F_0-
F_1- F_2,
\end{equation*}
where
\[
\begin{aligned}
F_0 &= \wt \F(0), \\
F_1 & = \frac{1}{2p_h^2} \left((\partial_u \LLo +\partial_u \Psi)^2-(\partial_u \LLo)^2\right), \\
F_2 & = \frac{\nu}{2} \left((\partial_\theta \LLo + \partial_\theta \Psi)^2-(\partial_\theta \LLo)^2\right).
\end{aligned}
\]

We start with $F_1$.
Using that $\|p_h^{-2}\|_{-2,-2} < \infty$, $\|\partial_u \LLo\|_{5,3} \le K(\nu I_0)^{-1}$,
$\Psi \in \B_{K^* (\nu I_0)^{-2}}\subset \wt \X_{5,3}$ and  Lemma~\ref{lem:propietatsespaisXrs}, we have that
\begin{multline}
\label{bound:contantdelipschitz1puntfixouter}
\|p_h^{-2}  (2\partial_u \LLo+\partial_u \Psi)\|_{0,0} \le \|p_h^{-2} \|_{-2,-2} \| 2\partial_u \LLo + \partial_u \Psi\|_{2,2} \\
 \le  K \left(2\frac{\nu I_0}{\kappa} \| \partial_u \LLo\|_{5,3} + \frac{(\nu I_0)^2}{\kappa^2} \| \partial_u \Psi\|_{6,4} \right)
\le K \left(\frac{1}{\kappa}  + \frac{K^*}{\kappa^2} \right)
\end{multline}
for some constants $K$ independent of $\nu I_0$ and $\kappa$.

Then,
using again Lemma~\ref{lem:propietatsespaisXrs},
\[
\begin{aligned}
\|F_1\|_{6,4}   & = \frac{1}{2} \|p_h^{-2} (2\partial_u \LLo + \partial_u \Psi) \partial_u \Psi\|_{6,4} \\
& \le \frac{1}{2} \|p_h^{-2} (2\partial_u \LLo + \partial_u \Psi)\|_{0,0}\| \partial_u \Psi\|_{6,4} \\
& \le \frac{ K}{2}  \left(\frac{1}{\kappa}+ \frac{K^*}{\kappa^2}\right) \frac{K^*}{(\nu I_0)^2}.
\end{aligned}
\]

Now we deal with $F_2$. We observe that,  since $\| \LLo\|_{4,2} = K(\nu I_0)^{-1}$ and
$\Psi \in \B_{K^* (\nu I_0)^{-2}}\subset \wt \X_{5,3}$, by Lemma~\ref{lem:propietatsespaisXrs},
\begin{equation}
\label{bound:contantdelipschitz2puntfixouter}
\| 2\partial_\theta \LLo+ \partial_\theta \Psi\|_{0,0}  \le 2K\frac{(\nu I_0)^2}{\kappa^2}\| \partial_\theta \LLo\|_{4,2}+
K\frac{(\nu I_0)^4}{\kappa^4} \| \partial_\theta \Psi\|_{6,4}\le  K \left(\frac{\nu I_0}{\kappa^2} +K^* \frac{\nu I_0}{\kappa^4}\right).
\end{equation}
Then,
\[
\left\|F_2\right\|_{6,4}  \le
\frac{\nu}{2} \|2\partial_\theta \LLo + \partial_\theta \Psi\|_{0,0} \| \partial_\theta \Psi\|_{6,4}
\le
\frac{\nu}{2}  K\left(1 + \frac{K^*}{\kappa^2}\right) \frac{1}{\kappa^2 (\nu I_0)^2}.
\]
Then, $\wt \F (\Psi) \in \X_{6,4}$ and, by Lemma \eqref{lem:operadorGu}, taking $\kappa$ large enough, $\G^u \circ \wt \F(\Psi) \in \B_{K^* (\nu I_0)^{-2}} \subset \wt \X_{5,3}$.

Now we check that $\wt \F$ is Lipschitz with $\Lip \wt \F \le K \kappa^{-1}$. Let $\Psi, \Psi' \in \B_{K^* (\nu I_0)^{-2}} \subset \wt \X_{5,3}$.
Observe that $\|p_h^{-2}(2\partial_u \LLo+ \partial_u \Psi+\partial_u \Psi')\|_{0,0} $ and $\| 2\partial_\theta \LLo+ \partial_\theta \Psi+ \partial_\theta \Psi'\|_{0,0}$ are bounded as in~\eqref{bound:contantdelipschitz1puntfixouter} and~\eqref{bound:contantdelipschitz2puntfixouter}.
Then,
\[
\begin{aligned}
\|\wt \F(\Psi) - \wt \F(\Psi')\|_{6,4} \le & \frac{1}{2}\|p_h^{-2}(2\partial_u \LLo+ \partial_u \Psi+\partial_u \Psi')\|_{0,0} \|\partial_u \Psi-\partial_u \Psi'\|_{6,4} \\
& + \nu \| 2\partial_\theta \LLo+ \partial_\theta \Psi+ \partial_\theta \Psi'\|_{0,0}
\| \partial_\theta \Psi- \partial_\theta \Psi'\|_{6,4} \\
 \le & K \left(\frac{1}{\kappa}  + \frac{2K^*}{\kappa^2} \right) \lln \Psi - \Psi' \rrn_{5,3} +
\nu K \left(\frac{\nu I_0}{\kappa^2} + 2K^* \frac{\nu I_0}{\kappa^4}\right) (\nu I_0)^{-1} \lln \Psi - \Psi' \rrn_{5,3}
\\  \le & \frac{K}{\kappa} \lln \Psi - \Psi' \rrn_{5,3},
\end{aligned}
\]
for some constant $K_2>0$.
Then, by the last claim of Lemma~\ref{lem:operadorGu},
\[
\lln \G^u \circ \wt \F(\Psi) - \G^u \circ \wt \F(\Psi')\rrn_{5,3} \le  K\| \wt \F(\Psi) -\wt \F(\Psi')\|_{6,4} \le
\frac{K}{\kappa} \lln \Psi - \Psi' \rrn_{5,3}.
\]
The claim follows taking $\kappa$  large enough.
Then, we easily check that $\G^u \circ \wt \F$ sends the ball $\B_{K^*(\nu I_0)^{-2} }$ into itself and has a unique fixed point there.
\end{proof}

\section{First extension of the invariant manifold}
\label{sec:extesio_coeficients_de_Fourier}

In Section~\ref{sec:HJouter} we have found the extension of the local  unstable invariant manifold to the complex domain~$\D^+_{\kappa,\delta} \times \T_{\sigma}$ (see Figure~\ref{fig:DomRaro}) by means of the function $\Phi^+ = \Phi_0+ \LLo + \Phi_2^+$, where $\Phi_0$ is introduced in~\eqref{def:Phi0}, $\LLo$ in \eqref{def:Loutmes} and $\Phi_2^+$ in Proposition~\ref{prop:aproximacioperlahomoclinicaversio2}.
Hence, as commented in Remark~\ref{rem:varietat_estable}, the parametrization of the stable manifold is given by $\Phi^-(u,\theta) = - \Phi^+(-u,-\theta)$. It is defined in $\D^-_{\kappa,\delta} \times \T_{\sigma}$, where $\D^-_{\kappa,\delta} = -\D^+_{\kappa,\delta}$. This extension of $\Phi^+$ is not enough for our purposes, because
$\D^+_{\kappa,\delta} \cap \D^-_{\kappa,\delta} \cap \R = \emptyset$ and we cannot compute the difference between the manifolds in the reals.
In this section we will extend $\Phi^+$ to a larger domain, $\D^+_{\kappa,\delta} \cup \D^+_{\kappa,\mathrm{ext}} $, defined below in~\eqref{def:DominisRaros_extesos} (see Figure~\ref{fig:DomRaroextes}). Once extended, the parametrizations of the unstable and stable
manifolds will be defined in a common domain containing an interval of $\R$. We will compute the difference of the manifolds in this common domain in Section~\ref{sec:diferenciaentresolucionseqHJ}.

We remark that we have not been able to find the extension of $\Phi^+$ in a single step because the Hamilton-Jacobi equation~\eqref{def:HJoriginal} is not defined at $u=0$. The reason lies in the fact that the method we have used to find $\Phi^+$ requires computing some integrals along straight lines with some slope in the complex domain where the variable $u$ lives. These straight lines cannot go through $0$ and hence the current method does not allow us to extend $\Phi^+$ beyond $u=0$ directly. In this section we will find an extension
of $\Phi^+$ to $\D^+_{\kappa,\mathrm{ext}} $ by choosing another type of parametrization of the invariant manifold, and then we will go back to the original type of parametrization. However, this extension will not be defined at $u=0$.

\subsection{From Hamilton-Jacobi parametrization to flow parametrization}

Taking into account~~\eqref{def:variablesuP} and \eqref{def:PJcomagrafica}, the solution $\Phi^+ = \Phi_0+ \LLo + \Phi_2^+$ of the Hamilton-Jacobi equation~\eqref{def:HJoriginal} obtained in Section~\ref{sec:HJouter} provides the parametrization of the unstable manifold
\begin{equation}
\label{def:Gammames}
\begin{pmatrix}
	q\\
	p\\
	\theta\\
	J
\end{pmatrix}
=
\Gamma^+(u,\theta) =
\begin{pmatrix}
q_h(u) \\
\frac{1}{p_h(u)} \partial_u \Phi^+(u,\theta) \\
\theta \\
\partial_\theta \Phi^+(u,\theta)
\end{pmatrix} = \Gamma_0 (u,\theta) + \Gamma_1(u,\theta) + \Gamma_2 (u,\theta),
\end{equation}
where
\begin{equation}
\label{def:Gamma0Gamma1Gammames}
\Gamma_0 (u,\theta) = \begin{pmatrix}
q_h(u) \\
\frac{1}{p_h(u)} \partial_u \Phi_0(u) \\
\theta \\
0
\end{pmatrix}, \qquad
\Gamma_1(u,\theta) = \begin{pmatrix}
0 \\
\frac{1}{p_h(u)} \partial_u \LLo (u,\theta) \\
0 \\
\partial_\theta \LLo (u,\theta)
\end{pmatrix}
\end{equation}
and $\Gamma_2= \Gamma^+-\Gamma_0-\Gamma_1 $.
We recall that $p_h(0) = 0$. However, $\Gamma_0$ is analytic at $u=0$ because $p_h(u)^{-1} \partial_u \Phi_0(u) = p_h(u)$.

Let $X = (X_q,X_p,X_\theta,X_J)^\top$ be the vector field corresponding to $H$, in~\eqref{def:rescaledHamiltonianI0}, using the $2$-form in~\eqref{def:rescaledbsymplecticform}, and let  $\phi_t$ be its flow. We look for a change of variables
\begin{equation*}
(u,\theta) = (v+f_1(v,\varphi),\varphi+f_2(v,\varphi))
\end{equation*}
such that conjugates $X$ on the unstable invariant manifold, parametrized by
\begin{equation}
\label{def:parametritzaciocanvidevariablesdeHJaflux}
\wt \Gamma^+ (v,\varphi) = \Gamma^+(v+f_1(v,\varphi),\varphi+f_2(v,\varphi))
\end{equation}
to the vector field $(1,\nu I_0)$, that is,
\[
\phi_t ( \wt \Gamma^+ (v,\varphi) ) =  \wt \Gamma^+ (v+t,\varphi+ \nu I_0 t)
\]
or, equivalently,
\begin{equation}
\label{eq:invarianciapelflux}
\LL (\wt \Gamma^+) = X \circ \wt \Gamma^+,
\end{equation}
where $\LL(\Gamma ) = \partial _v\Gamma  +\nu I_0 \partial _\varphi \Gamma $ is the operator defined in~\eqref{def:diffopL}. An immediate computation shows that~\eqref{def:parametritzaciocanvidevariablesdeHJaflux} satisfies~\eqref{eq:invarianciapelflux} if and only if $f=(f_1,f_2)$ is a solution of
\begin{equation}
\label{eq:canvidevariablesdeHJafluxedpcompleta}
[\partial_u \Gamma^+ \circ( \Id + f)] ( 1+ \LL (f_1))+ [\partial_\theta \Gamma^+\circ( \Id + f)] (\nu I_0+ \LL(f_2) ) = X \circ \Gamma^+\circ( \Id + f).
\end{equation}
We emphasize that the above equation has four components. However, the symplectic character of the vector field $X$ ensures that if two of them are
satisfied, so are the other two. We choose to solve the equations corresponding to the first and third components of~\eqref{eq:canvidevariablesdeHJafluxedpcompleta}.
Taking into account~\eqref{def:Gammames}, the equality $\dot q_h = - q_h p_h$ and the fact that $\partial_u \Phi_0 = p_h^2$, we can write these two equations as
\begin{equation}
\label{eq:canvidevariablesdeHJaflux}
\LL (f) = A \circ (\Id +f),
\end{equation}
where
\begin{equation}
\label{def:Acanvidevariablesaflux}
A = \begin{pmatrix}A_1 \\ A_2 \end{pmatrix} = \begin{pmatrix}
p_h^{-2} (\partial_u \LLo + \partial_u  \Phi^+_2) \\
\nu ( \partial_\theta \LLo + \partial_\theta  \Phi^+_2)
\end{pmatrix}.
\end{equation}
We denote $\NN(f)$ the right hand side of~\eqref{eq:canvidevariablesdeHJaflux}. We have that
\begin{equation}
\label{def:operadorFdeHJaflux}
\NN(f) = A + \DD{A}f+ \RR (f) ,
\end{equation}
where $\DD{A}$ denotes the derivative of $A$ and
$\RR(f) = A\circ (\Id + f) - A - \DD{A} f$. We will see in a moment that equation~\eqref{eq:canvidevariablesdeHJaflux} cannot be solved directly as a fixed point equation because the linear term $\DD{A} f$ in~\eqref{def:operadorFdeHJaflux} is too large.  We will need to rewrite it in a better suited way.

\subsubsection{Preliminaries and technical lemmas to solve equation \eqref{eq:canvidevariablesdeHJaflux}}

To solve equation~\eqref{eq:canvidevariablesdeHJaflux}, we consider $f = (f_1, f_2) \in \X_{r,s} \times \X_{r+1,s+1}$,
with the norm
\begin{equation*}
\|f\|_{r,s} = \|f_1\|_{r,s} +  \|f_2\|_{r+1,s+1}
\end{equation*}
and the operator $\G^u$ in~\eqref{def:Gu}, acting on each component. Also,
given a matrix function
\[
M = \begin{pmatrix} M_{1,1} & M_{1,2} \\ M_{2,1} & M_{2,2}
\end{pmatrix}
\]
with $M_{1,1}, M_{1,2} \in \X_{r,s} $ and $M_{2,1}, M_{2,2}\in \X_{r+1,s+1} $
we define
\begin{equation}
\label{def:normamatricialperalcanvi}
\|M\|_{r,s} = \max\left\{\|M_{1,1}\|_{r,s} + \|M_{2,1}\|_{r+1,s+1},\nu I_0(
\|M_{1,2}\|_{r,s} +  \|M_{2,2}\|_{r+1,s+1} )\right\}.
\end{equation}
It follows immediately from Lemma~\ref{lem:propietatsespaisXrs} that,
if $r,s \in \R $  and $\tilde r,\tilde s \ge 0$,
\begin{equation}
\label{bound:normamatricial}
\|M \|_{r,s}  \le K \left(\frac{\nu I_0}{\kappa}\right)^{\tilde s}\|M\|_{r+\tilde r,s+\tilde s}
\end{equation}
and, for any $r,r',s,s' \in \R$,
\begin{equation}
\label{bound:normamatricialivectorial}
\|M f\|_{r+r',s+s'}  \le \|M\|_{r,s} \|f\|_{r',s'}.
\end{equation}
Indeed, inequality~\eqref{bound:normamatricial} follows immediately from (1) of Lemma~\ref{lem:propietatsespaisXrs}. As for~\eqref{bound:normamatricialivectorial}, using (2) of Lemma~\ref{lem:propietatsespaisXrs}, we have that
\begin{multline*}
\|(Mf)_1\|_{r+r',s+s'} + \|(Mf)_2\|_{r+r'+1,s+s'+1} \\
\begin{aligned}
 & = \|M_{1,1} f_1 + M_{1,2} f_2\|_{r+r',s+s'} +\|M_{2,1} f_1 + M_{2,2} f_2\|_{r+r'+1,s+s'+1}\\
& \le \|M_{1,1} f_1\|_{r+r',s+s'} + \|M_{1,2} f_2\|_{r+r',s+s'} +\|M_{2,1} f_1\|_{r+r'+1,s+s'+1} + \|M_{2,2} f_2\|_{r+r'+1,s+s'+1} \\
& \le \left( \|M_{1,1}\|_{r,s}  + \|M_{2,1}\|_{r+1,s+1} \right)\|f_1\|_{r',s'} +K \frac{\nu I_0}{\kappa} (\|M_{1,2}\|_{r,s}+\|M_{2,2}\|_{r+1,s+1})\|f_2\|_{r'+1,s'+1}\\
& \le \|M\|_{r,s} \|f\|_{r',s'},
\end{aligned}
\end{multline*}
if $\kappa \ge K$.

We will look for the solution of equation~\eqref{eq:canvidevariablesdeHJaflux} in a domain slightly smaller than
$\D^+_{\kappa,\delta}\times \T_\sigma$. The norms we use depend on the choice of the domains.  In particular, they depend on $\kappa$, $\delta $ and $\sigma$.
Below, we will increase $\kappa $ to $\tilde \kappa$ and decrease $\delta$ to $\tilde\delta$. In order to have the distance from the points of the boundary of $\D^+_{\tilde \kappa,\tilde \delta} $ to $\D^+_{\kappa,\delta}$ constant,   when taking $\tilde \kappa > \kappa $ we will take
$ \delta-\tilde\delta =\frac{\tilde \kappa - \kappa}{\nu I_0} \frac{\cos \beta _1}{\cos \beta _2} $, sometimes without explicit mention of it.
  To clarify the exposition, till the end of the section, we include $\sigma$ and $\kappa$ as subscripts in the norms, but not $\delta$, since we will understand that when changing the domain the previous rule applies. In this way $\|f\|_{r,s,\kappa,\sigma}$ will denote the norm $\|f\|_{r,s}$ for $f$ defined either
in $\D^+_{\kappa,\delta}$ or in $\D^+_{\kappa,\delta} \times \T_{\sigma}$, depending on the setting. We remark that if $\tilde \kappa > \kappa$ (with $\tilde \kappa (\nu I_0)^{-1}< \tilde \delta$)
and $0 < \tilde \sigma < \sigma$, then $\D^+_{\tilde \kappa, \tilde \delta} \times \T_{\tilde \sigma} \subset \D^+_{\kappa,\delta} \times \T_{\sigma}$.
Analogously, we will  denote the spaces by $\X_{r,s,\kappa,\sigma}$ to clarify their dependence on the domain. The reduction of domain will only be done a finite number of times.

\begin{lemma}
\label{lem:tecnicderivadesmateixanorma}
For all $\tilde \kappa > \kappa$ and $0 < \tilde \sigma <\sigma$, there exists $C>0$ such that if $B \in \X_{r,s,\kappa,\sigma}$, then, for all $m,n \in \N$,  $\partial_u^{m} \partial_\theta^n B \in \X_{r,s,\tilde \kappa,\tilde \sigma}$  with
\[
\|\partial_u^{m}  B\|_{r,s, \tilde \kappa,\tilde \sigma} \le  C
\left(\frac{\tilde \kappa}{\kappa}\right)^{s} \frac{(\nu I_0)^m m! }{(\tilde \kappa -\kappa)^m } \frac{1}{(\cos \beta_1)^m}\|B\|_{r,s, \kappa, \sigma}, \qquad m\ge 0,
\]

\[
\|\partial_u^{m} \partial_\theta^n B\|_{r,s, \tilde \kappa,\tilde \sigma} \le  C
\left(\frac{\tilde \kappa}{\kappa}\right)^{s} \frac{(\nu I_0)^m m! n!}{(\tilde \kappa -\kappa)^m (\sigma-\tilde \sigma)^{n+1}} \frac{1}{(\cos \beta_1)^m}\|B\|_{r,s, \kappa, \sigma}, \qquad m\ge 0, \quad n\ge 1.
\]
\end{lemma}

When applying the previous lemma, we will choose $\tilde \kappa = \kappa + \kappa_0$, where $\kappa_0 >0$ is fixed. In this way,
$\tilde \kappa / \kappa = 1 + \kappa_0 /\kappa < 2$, if $\kappa$ is large enough.

\begin{proof}[Proof of Lemma~\ref{lem:tecnicderivadesmateixanorma}]
We note that
\[
(\partial_u^{m} \partial_\theta^n B)^{[k]} = (i k)^n \partial_u^{m}  B^{[k]}.
\]
By Cauchy formula,
\[
\partial_u^{m}  B^{[k]}(u)  = \frac{m!}{2\pi i} \int_{\gamma_u} \frac{B^{[k]}(z)}{(z-u)^{m+1}}\, dz,
\]
taking $\gamma_u$ to be the circle of radius $(\tilde \kappa-\kappa) (\nu I_0)^{-1}\cos \beta_1$, we immediately have, for some $C_1>0$,
\[
\|\partial_u^{m}  B^{[k]}\|_{r,s, \tilde \kappa,\tilde \sigma} \le C_1 \left(\frac{\tilde \kappa}{\kappa}\right)^s\frac{m!(\nu I_0)^{m}}{(\tilde \kappa-\kappa)^m}\frac{1}{(\cos \beta_1)^m}\|B^{[k]}\|_{r,s,\kappa, \sigma}.
\]
Hence, since $\|B^{[k]}\|_{r,s,\kappa, \sigma} \le \|B\|_{r,s,\kappa, \sigma} e^{-|k|\sigma}$,  we have that
\[
\|\partial_u^{m}  \partial_\theta^n B^{[k]}\|_{r,s, \tilde \kappa,\tilde \sigma} \le C_1 \left(\frac{\tilde \kappa}{\kappa}\right)^s\frac{m! |k|^n(\nu I_0)^{m}}{(\tilde \kappa-\kappa)^m} \frac{1}{(\cos \beta_1)^m}e^{-|k|\sigma}\|B\|_{r,s,\kappa, \sigma}.
\]
Therefore,
\begin{align*}
\|\partial_u^{m} \partial_\theta^n B\|_{r,s, \tilde \kappa,\tilde \sigma} &= \sum_{k \in \Z} \|\partial_u^{m} \partial_\theta^n B^{[k]}\|_{r,s, \tilde \kappa,\tilde \sigma} e^{|k|\tilde \sigma}  \\
& \le \frac{C_1}{(\cos \beta_1)^m} \left(\frac{\tilde \kappa}{\kappa}\right)^s\frac{m! (\nu I_0)^{m}}{(\tilde \kappa-\kappa)^m} \|B\|_{r,s,\kappa, \sigma} \sum_{k \in \Z} |k|^n e^{-|k| (\sigma-\tilde \sigma)}\\
& \le \frac{C_1 C_2}{(\cos \beta_1)^m} \left(\frac{\tilde \kappa}{\kappa}\right)^s\frac{m! n!(\nu I_0)^{m}}{(\tilde \kappa-\kappa)^m (\sigma-\tilde \sigma)^{n+1}} \|B\|_{r,s,\kappa, \sigma},
\end{align*}
where we have used that $\sum_{k\ge 0} k^n e^{-b k} \le C_2 n!/b^{n+1}$, for $b>0$ and $n \ge 1$, for some $C_2$.
\end{proof}

\begin{lemma}
\label{lem:tecniccomposicioproperaalaidentitat} Let $C$ be the constant given by Lemma~\ref{lem:tecnicderivadesmateixanorma} and $r'\ge0$ and $s' >0$. Given $K>0$ and $\delta >0$, for any $0 < \tilde \sigma <\sigma$, there exist $\kappa_0 >0$ such that for any $\tilde \kappa > \kappa+\delta > \kappa \ge \kappa_0$,  if $B \in \X_{r,s,\kappa,\sigma}$ and $f\in \X_{r',s',\tilde \kappa,\tilde \sigma}\times \X_{r'+1,s'+1,\tilde \kappa,\tilde \sigma}$ with $\|f\|_{r',s',\tilde \kappa,\tilde \sigma} \le K/ (\nu I_0)^{s'+1}$, then $B \circ (\Id +f) \in \X_{r,s,\tilde \kappa,\tilde \sigma}$
and
\[
\|B \circ (\Id +f)\|_{r,s,\tilde \kappa,\tilde \sigma} \le  \frac{2C}{\sigma-\tilde \sigma}
\left(\frac{\tilde \kappa}{\kappa}\right)^s \|B\|_{r,s,\kappa,\sigma}.
\]
\end{lemma}

\begin{proof}
First of all, by (1) of Lemma~\ref{lem:propietatsespaisXrs}, we have that
\[
\begin{aligned}
\|f_1\|_{0,0,\tilde \kappa,\tilde \sigma} & \le K \frac{(\nu I_0)^{s'}}{\tilde \kappa^{s'}}
\|f_1\|_{r',s',\tilde \kappa,\tilde \sigma} \le \frac{K}{\tilde \kappa^{s'}\nu I_0 }, \\
\|f_2\|_{0,0,\tilde \kappa,\tilde \sigma} & \le K \frac{(\nu I_0)^{s'+1}}{\tilde \kappa^{s'+1}}
\|f_2\|_{r'+1,s'+1,\tilde \kappa,\tilde \sigma} \le \frac{K}{\tilde \kappa^{s'+1} }.
\end{aligned}
\]
Then, expanding $B$ in Taylor series and using Lemma~\ref{lem:tecnicderivadesmateixanorma}, we have that,
writing $\alpha _1= \cos \beta _1$,
\[
\begin{aligned}
\|B \circ (\Id +f)\|_{r,s,\tilde \kappa,\tilde \sigma} & \le \sum_{j\ge 0} \frac{1}{j!}\sum_{\ell=0}^j \binom{j}{\ell}
\|\partial_u^{j-\ell} \partial_\theta^\ell B f_1^{j-\ell} f_2^{\ell}\|_{r,s,\tilde \kappa,\tilde \sigma} \\
& \le \sum_{j\ge 0} \sum_{\ell=0}^j \frac{1}{(j-\ell)! \ell!}
\|\partial_u^{j-\ell} \partial_\theta^\ell B\|_{r,s,\tilde \kappa,\tilde \sigma}\| f_1\|_{0,0,\tilde \kappa,\tilde \sigma}^{j-\ell}\| f_2\|^{\ell}_{0,0,\tilde \kappa,\tilde \sigma} \\
& \le \frac{C}{\sigma-\tilde \sigma}
\left(\frac{\tilde \kappa}{\kappa}\right)^s \sum_{j\ge 0} \sum_{\ell=0}^j
\left(\frac{K}{\tilde \kappa^{s'}\nu I_0 }\right)^{j-\ell}\left(\frac{K}{\tilde \kappa^{s'+1} }\right)^{\ell}
 \frac{(\nu I_0)^{j-\ell}}{(\alpha_1(\tilde \kappa -\kappa))^{j-\ell}(\sigma-\tilde \sigma)^\ell} \|B\|_{r,s, \kappa, \sigma} \\
 & \le  \frac{C}{\sigma-\tilde \sigma}
\left(\frac{\tilde \kappa}{\kappa}\right)^s \sum_{j\ge 0} K^j \sum_{\ell=0}^j
\frac{1}{(\tilde \kappa^{s'}(\alpha_1(\tilde \kappa -\kappa))^{j-\ell}} \frac{1}{(\tilde \kappa^{s'+1}(\sigma -\tilde \sigma))^{\ell}}\|B\|_{r,s, \kappa, \sigma} \\
& \le  \frac{C}{\sigma-\tilde \sigma}
\left(\frac{\tilde \kappa}{\kappa}\right)^s \sum_{j\ge 0} \left(K \left(\frac{1}{\tilde \kappa^{s'}\alpha_1(\tilde \kappa -\kappa)}+ \frac{1}{\tilde \kappa^{s'+1}(\sigma -\tilde \sigma)}\right)\right)^j \|B\|_{r,s, \kappa, \sigma} \\
& \le \frac{2C}{\sigma-\tilde \sigma}
\left(\frac{\tilde \kappa}{\kappa}\right)^s \|B\|_{r,s, \kappa, \sigma},
\end{aligned}
\]
where we have chosen $\kappa_0$ such that
\[
\frac{K}{\tilde \kappa^{s'}} \left(\frac{1}{\alpha_1(\tilde \kappa -\kappa)}+ \frac{1}{\tilde \kappa(\sigma -\tilde \sigma)}\right)
\le \frac{1}{2}.
\]
\end{proof}

\subsubsection{Solution of equation~\eqref{eq:canvidevariablesdeHJaflux}}

Let $\G^u$ be the operator defined in~\eqref{def:Gu}. We formally
define
\begin{equation}
\label{def:f0F0}
f^0 = \G^u(A), \qquad F^0 = \G^u(\DD{A}).
\end{equation}

\begin{lemma}
\label{lem:fitesdeADAiF0}
Let $A$ be the function defined in~\eqref{def:Acanvidevariablesaflux}. Let $\tilde \kappa > \kappa$ and $0 <\tilde \sigma < \sigma$, with $\kappa $ big.
Then, there exists $K>0$ such that
\begin{enumerate}
\item[(1)]
$\|A\|_{3,1,\kappa,\sigma}\le K(\nu I_0)^{-1}$,
\item[(2)] $\|f^0\|_{3,1,\kappa,\sigma} \le  K(\nu I_0)^{-2}$,
\item[(3)]  $\|\DD{A}\|_{3,1,\tilde \kappa, \tilde \sigma} \le K$,
\item[(4)] $\|F^0\|_{3,1,\tilde \kappa, \tilde \sigma} \le K(\nu I_0)^{-1}$.
\end{enumerate}
\end{lemma}

\begin{proof}
We write $A = A^0+A^1$, where
\[
A^0 =  \begin{pmatrix}
p_h^{-2} \partial_u \LLo  \\
\nu  \partial_\theta \LLo
\end{pmatrix}, \qquad
A^1 =  \begin{pmatrix}
p_h^{-2} \partial_u  \Phi_2^+ \\
\nu  \partial_\theta \Phi_2^+
\end{pmatrix}.
\]
Taking into account that $p_h^{-2} \in \X_{-2,-2,\kappa,\sigma}$ and the properties of $\LLo$ in Proposition~\ref{prop:primeraiteracio},
we have that $\|A^0\|_{3,1,\kappa,\sigma} \le K (\nu I_0)^{-1}$ and $\langle A^0\rangle = 0$. Also, by the properties of $\Phi_2^+$ in Proposition~\ref{prop:aproximacioperlahomoclinicaversio2},
we have that $\|A^1\|_{4,2,\kappa,\sigma} \le K  (\nu I_0)^{-2}$.
Hence, $\|A^1\|_{3,1,\kappa,\sigma} \le K \kappa^{-1} (\nu I_0)^{-1}$.
This proves (1).

We bound $\|\G^u (A^0) \|_{3,1,\kappa,\sigma}$ and $\|\G^u (A^1) \|_{3,1,\kappa,\sigma}$ separately.

Since $\langle A^0 \rangle = 0$, by (1) of Lemma~\ref{lem:operadorGu},
\[
\|\G^u (A^0) \|_{3,1,\kappa,\sigma} \le \frac{K}{\nu I_0} \|A^0\|_{3,1,\kappa,\sigma} \le \frac{K}{(\nu I_0)^{2}}.
\]

Using again that $p_h^{-2}  \in \X_{-2,-2}$, the properties of $\Phi_2^+$ in Proposition~\ref{prop:aproximacioperlahomoclinicaversio2}
imply that $\|A^1\|_{4,2,\kappa,\sigma} \le K (\nu I_0)^{-2}$. Then, by (2) of Lemma~\ref{lem:operadorGu},
\[
\|\G^u (A^1) \|_{3,1,\kappa,\sigma} \le K \|A^1\|_{4,2,\kappa,\sigma} \le \frac{K }{(\nu I_0)^{2}},
\]
from which (2) follows.

We write $\DD{A} = \DD{A^0} + \DD{A^1}$. From the previous bounds for $A^0$
and $A^1$, applying Lemma~\ref{lem:tecnicderivadesmateixanorma} and the definition of the
 matrix norm~\eqref{def:normamatricialperalcanvi}, we have that, in the reduced domain,
\begin{equation*}
\begin{aligned}
\|\DD{A^0}\|_{3,1,\tilde \kappa, \tilde \sigma} & \le K,\\
\|\DD{A^1}\|_{4,2,\tilde \kappa, \tilde \sigma} & \le \frac{K}{\kappa \nu I_0}.
\end{aligned}
\end{equation*}
Here the constants depend on $\tilde \kappa -\kappa$ and $\sigma -\tilde \sigma$.
This proves (3).

We finally prove (4).
Since $\langle A^0\rangle = 0$, we have that $\langle \DD{A^0} \rangle = 0$. By (1) of Lemma~\ref{lem:operadorGu},
\[
\|\G^u(\DD{A^0})\|_{3,1,\tilde \kappa, \tilde \sigma} \le \frac{K}{\nu I_0}.
\]
Moreover, by (2) of Lemma~\ref{lem:operadorGu},\[
\|\G^u(\DD{A^1})\|_{3,1,\tilde \kappa, \tilde \sigma} \le K \|\DD{A^1}\|_{4,2,\tilde \kappa, \tilde \sigma} \le  \frac{K }{\nu I_0}.
\]
Since $F^0 = \G^u(\DD{A^0}) + \G^u(\DD{A^1})$, the claim follows.
\end{proof}

We introduce
$\tilde f$ by
$$
f = f^0+(\Id + F^0) \tilde f.
$$
Observe that, by (4) of Lemma~\ref{lem:fitesdeADAiF0} and (1) of Lemma~\ref{lem:propietatsespaisXrs},
\[
\|F^0\|_{0,0,\tilde \kappa,\tilde \sigma} \le K \frac{\nu I_0 }{\kappa} \|F^0\|_{3,1,\tilde \kappa,\tilde \sigma} \le \frac{K}{\kappa}.
\]
Hence, if $\kappa$ is large enough, $\Id + F^0$ is invertible and $\|(\Id + F^0)^{-1}\|_{0,0,\kappa, \sigma} \le 2$.
Using that, by~\eqref{def:f0F0}, $\LL(f^0) = A$ and $\LL (F^0)  = \DD{A}$,  we rewrite equation~\eqref{eq:canvidevariablesdeHJaflux}
as
\begin{equation}
\label{eq:canvidevariablesdeHJaflux2}
\LL (\tilde f) = \wt \NN (\tilde f),
\end{equation}
where
\begin{equation*}
\wt \NN (\tilde f) =
(\Id + F^0)^{-1}\DD{A} \,f^0 +
(\Id+F^0)^{-1} \DD{A} \,F^0 \tilde f
+(\Id+F^0)^{-1} \RR\left(f^0+(\Id+F^0)\tilde f\right)
\end{equation*}
and $\RR$ was introduced in~\eqref{def:operadorFdeHJaflux}.
Using the operator $\G^u$, we rewrite equation~\eqref{eq:canvidevariablesdeHJaflux2} as the fixed point equation \begin{equation}
\label{eq:canvidevariablesdeHJaflux2puntfix}
\tilde f = \G^u \circ \wt \NN (\tilde f).
\end{equation}

\begin{proposition}
\label{prop:canvidevariablesdeHJaflux}
For all $\hat \kappa > \tilde \kappa > \kappa$ and $0 < \hat \sigma <\tilde \sigma <\sigma$, with $\kappa$ big enough,
equation~\eqref{eq:canvidevariablesdeHJaflux2puntfix} has a solution $\tilde f^+$, defined in $\D^+_{\hat \kappa,\hat \delta} \times \T_{\hat \sigma}$, and $\tilde f^+ \in \X_{3,1,\hat \kappa,\hat \sigma}\times \X_{4,2,\hat \kappa,\hat \sigma}$ with
$\|\tilde f^+\|_{3,1,\hat \kappa,\hat \sigma} \le C (\nu I_0)^{-2}$.
Consequently, $\wt \Gamma^+ = \Gamma^+ \circ (\Id +f^0+(I+F^0)\tilde f^+)$ satisfies the invariance equation~\eqref{eq:invarianciapelflux} in  $\D^+_{\hat \kappa,\delta} \times \T_{\hat \sigma}$.
\end{proposition}

\begin{proof}[Proof of Proposition~\ref{prop:canvidevariablesdeHJaflux}]

We claim that there exists $K_1>0$ such that
\[
\|\G^u \circ \wt \NN(0) \|_{3,1,\hat \kappa,\hat \sigma} \le \frac{K_1}{ (\nu I_0)^2}.
\]
Indeed, first we notice that
\begin{multline*}
\wt \NN(0) =  (\Id+F^0)^{-1}(\DD A \,f^0 + \RR(f^0)) = (\Id+F^0)^{-1} \left[\DD{A} \, f^0 + \int_0^1 (1-s) D^2 A \circ (\Id + s f^0) \,ds \,(f^0)^{\otimes 2} \right] \\
= (\Id+F^0)^{-1} \left[\DD{A} \, f^0 + \int_0^1 (1-s)\left( f_1^0 \partial_u \DD{A} \circ (\Id + s f^0)  f^0 +
f_2^0 \partial_\theta \DD{A} \circ (\Id + s f^0)  f^0 \right) \,ds \right].
\end{multline*}
Next, we note that
\[
\|\DD{A} \, f^0 \|_{4,2,\hat \kappa,\hat \sigma} \le
\|\DD{A}  \, f^0 \|_{6,2,\hat \kappa,\hat \sigma} \le
\|\DD{A} \|_{3,1,\hat \kappa,\hat \sigma}
\|f^0 \|_{3,1,\hat \kappa,\hat \sigma}\le \frac{K}{(\nu I_0)^2}.
\]

By (3) of Lemma~\ref{lem:fitesdeADAiF0} and Lemma~\ref{lem:tecnicderivadesmateixanorma}, for $\hat \kappa > \tilde \kappa$ and $0< \hat \sigma < \tilde \sigma$, we have
\begin{equation}
\label{bound:partialuDapartialthetaDA}
\begin{aligned}
\|\partial_u \DD{A}\|_{3,1,\hat \kappa,\hat \sigma} & \le K \nu I_0 \|\DD{A} \|_{3,1,\tilde \kappa,\tilde \sigma} \le  K \nu I_0, \\
\|\partial_\theta \DD{A}\|_{3,1,\hat \kappa,\hat \sigma} & \le K  \|\DD{A} \|_{3,1,\tilde \kappa,\tilde \sigma} \le K.
\end{aligned}
\end{equation}
Moreover, by (1) of Lemma~\ref{lem:propietatsespaisXrs}
and (2)  of Lemma~\ref{lem:fitesdeADAiF0},
\[
\|f_1^0 \partial_u \DD{A} \circ (\Id + s f^0)  f^0 \|_{4,2,\hat \kappa,\hat \sigma} \le K \frac{\nu I_0}{\kappa}
\|f_1^0 \partial_u \DD{A} \circ (\Id + s f^0)  f^0 \|_{9,3,\hat \kappa,\hat \sigma} \le \frac{K}{\kappa(\nu I_0)^2}
\]
and, since $f^0_2\in \X_{4,2,\hat \kappa,\hat \sigma}$,
\[
\|f_2^0 \partial_\theta \DD{A} \circ (\Id + s f^0)  f^0\|_{4,2,\hat \kappa,\hat \sigma}\le K \left(\frac{\nu I_0}{\kappa}\right)^2
\|f_2^0 \partial_v \DD{A} \circ (\Id + s f^0)  f^0 \|_{10,4,\hat \kappa,\hat \sigma} \le \frac{K}{\kappa^2(\nu I_0)^2}.
\]
Hence,
\[
\|\G^u \circ \wt \NN (0) \|_{3,1,\hat \kappa,\hat \sigma} \le
 K \|\wt \NN (0) \|_{4,2,\hat \kappa,\hat \sigma} \le \frac{K_1}{\kappa (\nu I_0)^2}.
\]
Let $ C = 2 K_1$.

Let $f, f' \in \X_{3,1,\hat \kappa,\hat \sigma}\times \X_{4,2,\hat \kappa,\hat \sigma}$ with
$\|f\|_{3,1,\hat \kappa,\hat \sigma},\|f'\|_{3,1,\hat \kappa,\hat \sigma} \le \wt C$. We have that
\begin{equation}
\label{eq:wtFtildefmenyswtFhatf}
\wt \NN(f) - \wt \NN(f') = (I+F^0)^{-1} \left(M_1 +M_2 \right)(f -f'),
\end{equation}
where
\[
\begin{aligned}
M_1 & = \DD{A} \, F^0, \\
M_2(f, f') & = \int_0^1 \int_0^1 D^2 A \circ (\Id + t u_{s}(f, f') )\,dt \,  u_s(f, f') \,ds \, (I+F^0)
\end{aligned}
\]
with $u_{s} (f, f')= f^0+(\Id +F^0) f' +s(I+F^0)(f- f')$.

Then, using (1) of Lemma~\ref{lem:propietatsespaisXrs} and (3) and (4) of Lemma~\ref{lem:fitesdeADAiF0}, we have that
\begin{equation}
\label{bound:DAF0}
\|M_1 \|_{1,1,\hat \kappa, \hat \sigma} \le K\frac{\nu I_0}{\kappa} \|\DD{A} F^0 \|_{6,2,\hat \kappa, \hat \sigma} \le
K \frac{\nu I_0}{\kappa} \|\DD{A}\|_{3,1,\hat \kappa, \hat \sigma}  \| F^0 \|_{3,1,\hat \kappa, \hat \sigma} \le\frac{K}{\kappa}.
\end{equation}
Observe that, by (2) of Lemma~\ref{lem:fitesdeADAiF0} and the hypotheses on $f$, $f'$,
\[
\|u_s(f, f')\|_{3,1,\hat \kappa, \hat \sigma} \le  K(\nu I_0)^{-2}.
\]
Hence,  also using~\eqref{bound:partialuDapartialthetaDA}
\begin{align*}
\|u_{s,1} \partial_u \DD{A}\circ (\Id + t u_s)\|_{1,1,\hat \kappa, \hat \sigma}  & \le
K\frac{\nu I_0}{\kappa} \| u_{s,1}\partial_u \DD{A}\circ (\Id + t u_s)\|_{6,2,\hat \kappa, \hat \sigma} \\
 & \le
K \frac{\nu I_0}{\kappa} \| u_{s,1}\|_{3,1,\hat \kappa, \hat \sigma} \|\partial_u \DD{A}\circ (\Id + t u_s)\|_{3,1,\hat \kappa, \hat \sigma}
\le \frac{K}{\kappa}
\end{align*}
and
\begin{align*}
\|u_{s,2} \partial_\theta \DD{A}\circ (\Id + t u_s)\|_{1,1,\hat \kappa, \hat \sigma}  & \le
K \left(\frac{\nu I_0}{\kappa}\right)^2 \| u_{s,2}\partial_\theta \DD{A}\circ (\Id + t u_s)\|_{7,3,\hat \kappa, \hat \sigma} \\
 & \le
K \left(\frac{\nu I_0}{\kappa}\right)^2 \| u_{s,2}\|_{4,2,\hat \kappa, \hat \sigma} \|\partial_\theta \DD{A}\circ (\Id + t u_s)\|_{3,1,\hat \kappa, \hat \sigma}
 \le \frac{K}{\kappa^2},
\end{align*}
since $D^2 A\circ (\Id + t u_s) u_s = u_{s,1} \partial_u \DD{A}\circ (\Id + t u_s) + u_{s,2} \partial_\theta \DD{A}\circ (\Id + t u_s)$,
we have that
\begin{equation}
\label{bound:M2}
\|M_2(f, f')\|_{1,1,\hat \kappa, \hat \sigma} \le \frac{K}{\kappa}.
\end{equation}
Combining~\eqref{bound:DAF0} and~\eqref{bound:M2} with~\eqref{eq:wtFtildefmenyswtFhatf}, we obtain that
\[
\|\wt \NN(f) - \wt \NN(f')\|_{4,2,\hat \kappa, \hat \sigma} \le \|(I+F^0)^{-1}\|_{0,0,\hat \kappa, \hat \sigma} \| M_1 +M_2 \|_{1,1,\hat \kappa, \hat \sigma}
\|f -f'\|_{3,1,\hat \kappa, \hat \sigma} \le \frac{K}{\kappa} \|f -f'\|_{3,1,\hat \kappa, \hat \sigma}.
\]
Then, using (2) of Lemma~\ref{lem:operadorGu} and taking $\kappa$ large enough, $\|\G^u \wt \NN(f) - \G^u\wt \NN(f')\|_{3,1,\hat \kappa, \hat \sigma} \le \frac{K}{\kappa}
\|f -f'\|_{3,1,\hat \kappa, \hat \sigma}$ and hence
$\G^u \circ \wt \NN$ is a contraction in $\B _{C(\nu I_0)^{-2} }$.
\end{proof}

\subsection{Extension of the flow parametrization}

In this section we extend the parametrization $\wt \Gamma^+$ given by Proposition~\ref{prop:canvidevariablesdeHJaflux}, defined in the domain $\D^+_{\hat \kappa,\hat \delta} \times \T_{\hat \sigma}$,
to $\D^+_{\hat \kappa,\mathrm{ext}}\times \T_{\hat \sigma}$, defined in~\eqref{def:DominisRaros_extesos}, below. See Figure~\ref{fig:DomRaroextes}. To do so, we deal with the decomposition  $H = H_0+H_1$, in~\eqref{def:H0H1}, and denote $X_0$ and $X_1$ the vector fields corresponding to $H_0$ and $H_1$, respectively.
We perform this extension close to the invariant manifold of the unperturbed Hamiltonian $H_0$ in~\eqref{def:H0H1},
which is given by $\Gamma_0$, defined in~\eqref{def:Gamma0Gamma1Gammames}.
It satisfies
\[
\LL(\Gamma_0) = X_0 \circ \Gamma_0.
\]
We emphasize that, since $\Gamma_0$ only depends on $\theta$ through its $\theta$ component and $H_0$ does not depend on $\theta$, $X_0 \circ \Gamma_0$
does not depend on $\theta$.

To obtain the analytic continuation of $\wt \Gamma^+$, which is a solution of~\eqref{eq:invarianciapelflux}, we will solve this equation in the domain
$\D^+_{\kappa,\mathrm{ext}} \times \T_\sigma$, with appropriate initial conditions that guarantee that the solution we obtain coincides with $\wt \Gamma^+$ in an open set.

Since we want to look for $\wt \Gamma ^+$ close to $\Gamma_0$, we introduce the new unknown $\wt \Gamma_1$ by $\wt \Gamma^+ = \Gamma_0 + \wt \Gamma_1$. Equation~\eqref{eq:invarianciapelflux} becomes
\begin{equation}
\label{eq:invarianciapelfluxapropGamma0}
\wh \LL (\wt \Gamma_1) = \Fext (\wt \Gamma_1),
\end{equation}
with
\begin{equation}
\label{def:whL}
\wh \LL (\Gamma) = \LL(\Gamma) - (D X \circ \Gamma_0) \Gamma
\end{equation}
and
\begin{equation*}
\Fext (\Gamma)  = X_1 \circ \Gamma_0+  X \circ ( \Gamma_0+\Gamma) - X \circ \Gamma_0 - (DX \circ \Gamma_0) \Gamma,
\end{equation*}
where $X = X_0+X_1$.
We look for an appropriate right inverse of $\wh \LL$.

We start by defining the Banach spaces we will work with and the technical lemmas we will need.
We introduce the domain,
\begin{equation}
\label{def:DominisRaros_extesos}
\D^+_{\kappa,\mathrm{ext}} = \left\{u \in\C\mid \,  |\Im u|< - \tan\beta_1\, \Re
u+1-\kappa (\nu I_0)^{-1},
\; |\Im u|< \tan\beta_2\, \Re u+1-\delta\right\}.
\end{equation}
See Figure~\ref{fig:DomRaroextes}.

\begin{figure}[h]
\begin{center}
\includegraphics[width=.5\textwidth]{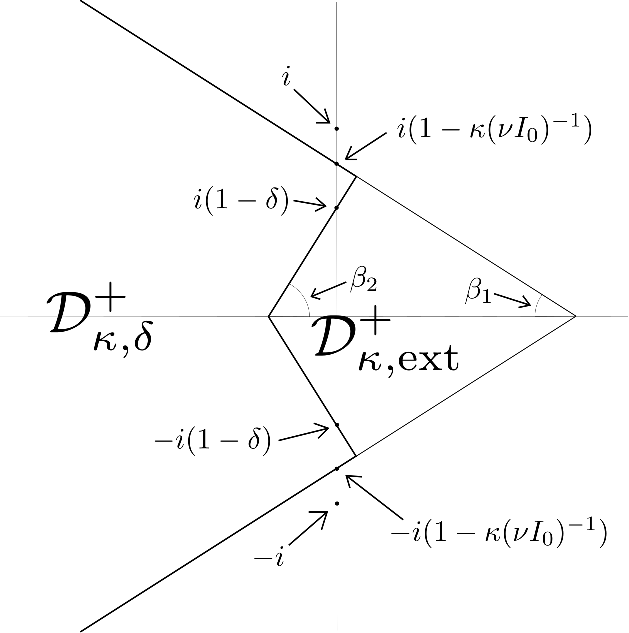}
\end{center}
\caption{The domain  $\D^+_{\kappa,\delta} \cup \D^+_{\kappa,\mathrm{ext}}$ defined in
\eqref{def:DominisRaros_extesos}.}\label{fig:DomRaroextes}
\end{figure}

For $s\in \R$, we introduce  the Banach space of $2\pi$-periodic in~$\theta$, real analytic functions
\begin{equation*}
\X_{s} = \{\Psi:\D^+_{\kappa,\mathrm{ext}} \times \T_\sigma \to \C \mid  \, \textrm{$\Psi$ real analytic},\; \|\Psi\|_{s} < \infty\},
\end{equation*}
with the norm,  taking into account the Fourier expansion of $\Psi(u,\theta) = \sum_{k\in \Z} \Psi^{[k]} (u) e^{ik\theta}$,
\[
\|\Psi\|_{s} = \sum_{k\in \Z} \|\Psi^{[k]}\|_{s} e^{|k|\sigma},
\]
where, for an analytic function $f:\D^+_{\kappa,\mathrm{ext}} \to \C$,
\[
\| f\|_{s} = \sup_{u \in \D^+_{\kappa,\mathrm{ext}}} |(1+u^2)^{s} f(u)|.
\]
Observe that, since for $\sup_{z \in \D^+_{\kappa,\mathrm{ext}}}|1+u^2|^{-1} < \infty$, the spaces $\X_s$ coincide for all $s$ and the norms are equivalent. However, the concrete values of the norms depend on $s$.

The following lemma is analogous to Lemma~\ref{lem:propietatsespaisXrs}.

\begin{lemma}
\label{lem:propietatsespaisXs}
Let $s, s_1, s_2 \in \R$.
There exists $C,K>0$ such that the following holds.
\begin{enumerate}
\item[(1)]
If $\Psi \in \X_{s+t}$ with $t \ge 0$, then $\Psi \in \X_{s}$ and
\[
\|\Psi\|_{s} \le C K^{t}\|\Psi\|_{s+t}.
\]
\item[(2)]
If $\Psi_1 \in \X_{s_1}$ and $\Psi_2 \in \X_{s_2}$, then $\Psi_1\Psi_2 \in \X_{s_1+s_2}$ and
\[
\|\Psi_1 \Psi_2 \|_{s_1+s_2} \le \|\Psi_1  \|_{s_1}\|\Psi_2 \|_{s_2}.
\]
\end{enumerate}
\end{lemma}

For $\Gamma \in \X_s^4$, we define
\[
\|\Gamma\|_s = \sum_{i=1}^{4} \|\Gamma_i\|_s.
\]
Also, if $M = (M_{i,j})_{1\le i,j\le 4} \in \M_{4,4}(\X_s)$, that is, with $M_{i,j} \in \X_s$, its norm is
\[
\|M\|_s = \max_{1\le j \le 4} \sum_{i=1}^4 \|M_{i,j}\|_s.
\]
Clearly, this matrix norm satisfies
$\|M \Gamma \|_{s_1+s_2} \le \|M  \|_{s_1}\|\Gamma \|_{s_2}$.

 Let $\Gf$ be the right inverse of $\LL$ formally defined
through the Fourier coefficients of the image $\Gf(\Gamma)$ by
\begin{equation}
\label{def:Gperaestendreelflux}
\begin{aligned}
\Gf(\Gamma)^{[k]}(u) & = \int_{u_0}^u e^{ik \nu I_0 (s-u)} \Gamma^{[k]}(s)\,ds, & \qquad \text{if $k> 0$}, \\
\Gf(\Gamma)^{[0]}(u) & = \int_{\rho}^u  \Gamma^{[0]}(s)\,ds, &  \\
\Gf(\Gamma)^{[k]}(u) & = \int_{\overline{u_0}}^u e^{ik\nu I_0 (s-u)} \Gamma^{[k]}(s)\,ds, & \qquad \text{if $k< 0$,}
\end{aligned}
\end{equation}
where $u_0$ and $\rho$ are the topmost and leftmost points of the domain $\D^+_{\kappa,\mathrm{ext}}$, respectively.

The following lemma is a simplified version of Lemma~\ref{lem:operadorGu}.

\begin{lemma}
\label{lem:operadorGflow}
Let $\Gf$ be the operator defined in~\eqref{def:Gperaestendreelflux}. Let $s\ge 0$. There exists $K>0$ such that, if $\Gamma \in \X_{s}^4$, then $\Gf(\Gamma)\in \X^4_s$ and
\[
\|\Gf(\Gamma)\|_{s} \le K \|\Gamma\|_{s}.
\]
If, furthermore, $\langle \Gamma \rangle = 0$, then
\[
\|\Gf(\Gamma)\|_{s} \le K (\nu I_0)^{-1}\|\Gamma\|_{s}.
\]
The constant $K$ only depends on $s$ and the constants involved in the definition of $\D^+_{\kappa,\mathrm{ext}}$ but it is independent of $\nu I_0$.
\end{lemma}

In order to solve~\eqref{eq:invarianciapelfluxapropGamma0}, we need an appropriate matrix solution $\wt M$ of
$\wh \LL (\wt M) = 0$, where $\wh \LL$ was introduced in~\eqref{def:whL}.
To do so, first we take $M(u)$, any real analytic fundamental matrix of the ordinary differential equation
$\frac{dx}{du} = DX_0 \circ \Gamma_0 \cdot x$ (recall that $DX_0 \circ \Gamma_0$ only depends on $u$). It is well defined and uniformly bounded in $\D^+_{\kappa,\mathrm{ext}}$,
because so is $DX_0 \circ \Gamma_0$ and the equation is linear. In particular, $\|M\|_0$ is bounded. The following lemma provides a matrix solution of $\wh \LL (\wt M) = 0$.

\begin{lemma}
\label{lem:soluciovariacionalscompletes}
There exists $N \in \M_{4,4}(\X_s)$ with $\|N\|_0 \le K (\nu I_0)^{-1}$, such that $\wt M = M (\Id + N)$ satisfies
$\wh \LL (\wt M) = 0$. In particular, $\wt M$ is invertible with $\|\wt M\|_0, \|\wt M^{-1}\|_0 \le K$,
$\|M- \wt M\|_0 \le K (\nu I_0)^{-1}$ and $\|M^{-1}- \wt M^{-1}\|_0 \le K (\nu I_0)^{-1}$.
\end{lemma}

\begin{proof}
Let $N_0 = \Gf ( M^{-1} D X_1 \circ \Gamma_0 M)$. Since $\|M^{-1} D X_1 \circ \Gamma_0 M\|_0 = \OO(1)$ and
$\langle M^{-1} D X_1 \circ \Gamma_0 M \rangle = 0$, Lemma~\ref{lem:operadorGflow} implies that $\|N_0 \|_0 \le K  (\nu I_0)^{-1}$.
In particular, if $\nu I_0$ is sufficiently big, $\Id + N_0$ is invertible.

Now, introducing the new unknown $\wt N$ by $\wt M = M (\Id + N_0)(\Id + \wt N)$ we have that $\wt M$ satisfies $\wh \LL (\wt M) = 0$ if and only if
\begin{equation}
\label{eq:matriufonamentalsencera}
\LL (\wt N) = \A(\wt N),
\end{equation}
where
\[
\A(\wt N) = (\Id+N_0)^{-1} M^{-1} DX_1 \circ \Gamma_0 M N_0 (\Id + \wt N).
\]
Indeed, on the one hand
\[
\begin{aligned}
\LL(\wt M ) & = \LL (M) (\Id + N_0)(\Id + \wt N) + M \LL(N_0) (\Id + \wt N) + M (\Id + N_0)\LL(\wt N) \\
& = DX_0 \circ \Gamma_0 M (\Id + N_0)(\Id + \wt N) +  D X_1 \circ \Gamma_0 M (\Id + \wt N) + M (\Id + N_0)\LL(\wt N)
\end{aligned}
\]
while, on the other,
\begin{multline*}
D(X_0+X_1) \circ \Gamma_0 \wt M \\
 = DX_0 \circ \Gamma_0 M (\Id + N_0)(\Id + \wt N)+ D X_1 \circ \Gamma_0 M (\Id + \wt N)
+DX_1 \circ \Gamma_0 M N_0(\Id + \wt N).
\end{multline*}

Using the operator $\Gf$ we consider the fixed point equation 
\[
\wt N = \Gf \circ \A (\wt N).
\]
A solution of it is also a solution of \eqref{eq:matriufonamentalsencera}.

First, we observe that
\[
\|\A(0)\|_0 = \|(\Id+N_0)^{-1} M^{-1} DX_1 \circ \Gamma_0 M N_0\|_0 \le
\|(\Id+N_0)^{-1} M^{-1} DX_1 \circ \Gamma_0 M \|_0 \|N_0\|_0\le K (\nu I_0)^{-1}.
\]
Hence, by Lemma~\ref{lem:operadorGflow}, $\|\Gf (\A(0))\|_0 \le K (\nu I_0)^{-1}$.

Second, for any $N, N' \in \M_{4,4}(\X_0)$, we have that
\[
\|\A(N) - \A(N')\|_0 \le \|(\Id+N_0)^{-1} M^{-1} DX_1 \circ \Gamma_0 M\|_0 \| N_0\|_0 \|N-N'\|_0 \le \frac{K}{\nu I_0} \|N-N'\|_0.
\]
Lemma~\ref{lem:operadorGflow} implies that, if $\nu I_0$ is sufficiently big, $\Gf \circ \A$ is a contraction in $\M_{4,4}(\X_0)$. Hence, equation~\eqref{eq:matriufonamentalsencera} has a unique solution $N^*$, with $\|N^*\|_0 \le K (\nu I_0)^{-1}$. The claim follows taking $N = N^* + N_0(\Id + N^*)$ and using that $M^{-1}-\wt M^{-1} = M^{-1} (\wt M - M)\wt M^{-1}$.
\end{proof}

Let $\wt M$ be the matrix given by Lemma~\ref{lem:soluciovariacionalscompletes}. Then, it is immediate to check that
\begin{equation}
\label{def:whG0}
\wh \G_0 (\Gamma) := \wt M \Gf ( \wt M^{-1} \Gamma)
\end{equation}
is a right inverse of the operator $\wh \LL$ in~\eqref{def:whL}. Furthermore, if $g$ satisfies $\LL(g) = 0$, $\wt M g + \wh \G_0 (\Gamma)$ is also a right inverse of $\wh \LL$.
In order to have the analytic continuation of $\wt \Gamma^+$, we choose $g$ by setting $g(u,\theta)
= \sum_{k\in \Z} g^{[k]}(u) e^{ik\theta}$
and
\begin{equation}
\label{def:coeficientsgcondicioinicial}
\begin{aligned}
g^{[k]}(u) & = e^{ik \nu I_0 (u_0-u)} \left(\wt M^{-1} (\wt \Gamma^+ - \Gamma_0)\right)^{[k]}(u_0) & \quad \text{if}\;\;  k >0, \\
g^{[0]}(u) & = \left(\wt M^{-1} (\wt \Gamma^+ - \Gamma_0)\right)^{[0]}(\rho), &  \\
g^{[k]}(u) & = e^{ik \nu I_0 (\overline{u_0}-u)} \left(\wt M^{-1} (\wt \Gamma^+ - \Gamma_0)\right)^{[k]}(\overline{u_0}) & \quad \text{if}\;\;  k <0
\end{aligned}
\end{equation}
where $\wt \Gamma^+$ is the solution given in Proposition \ref{prop:canvidevariablesdeHJaflux}.

With all this in mind, we define  $\wh \G_g (\Gamma) = \wt M g + \wh \G_0  ( \Gamma)$, a right inverse of $\wh \LL$, and rewrite equation~\eqref{eq:invarianciapelfluxapropGamma0} as the fixed point equation
\begin{equation}
\label{eq:extensiopelfluxpuntfix}
\wt \Gamma_1 = \wh \G_g \circ \Fext (\wt \Gamma_1).
\end{equation}

\begin{proposition}
\label{prop:wtGammames}
Equation~\eqref{eq:extensiopelfluxpuntfix} has a solution $\wt \Gamma_1^+ \in \X_0^4$, with $\|\wt \Gamma_1^+\|_0 \le K (\nu I_0)^{-1}$.
\end{proposition}

\begin{proof}
We observe that
\[
\|\Fext(0) \|_0 = \| X_1 \circ \Gamma_0\| = \OO(1),
\]
that  $\langle \Fext(0) \rangle = 0$ and $\langle M^{-1} \Fext(0) \rangle = 0$ .
We write $\wt M^{-1} = M^{-1} + ((\Id + N)^{-1}- \Id) M^{-1} $ and note that, by Lemma~\ref{lem:soluciovariacionalscompletes},   $\|(\Id + N)^{-1}- \Id\|_0 = \OO ((\nu I_0)^{-1})$.
Then
\[
\wh\G_0 \circ \Fext (0) = \wt M \Gf (\wt M^{-1} \Fext(0) )= \wt M \Gf ( M^{-1} \Fext(0) +((\Id + N)^{-1}- \Id) M^{-1} \Fext(0)   )
\]
and therefore,  by Lemmas~\ref{lem:operadorGflow} and~\ref{lem:soluciovariacionalscompletes},
\[
\|\wh \G_0 \circ \Fext (0) \|_0= \| \wt M\|_0 (\|K (\nu I_0)^{-1} +K (\nu I_0)^{-1} \|M^{-1}\|_0 \| \Fext(0)\|_0   )
\le \frac{K_1}{\nu I_0}.
\]
Let $K^* = 2 K_1$. Take $\Gamma, \Gamma' \in \B_{K^* (\nu I_0)^{-1}} \subset \X_0^4$.
We have  $D^2 X \circ \Gamma_0 = \OO(1)$ and,
since for $0 < s < 1$,
\[
\Gamma'+s(\Gamma-\Gamma')  = s\Gamma + (1-s) \Gamma' \in \B_{K^* (\nu I_0)^{-1}},
\]
we have
\[
\begin{aligned}
	\|\Fext(\Gamma) - \Fext(\Gamma') \|_0 & \le \int_0^1 \int_0^1 \| D^2 X \circ (\Gamma_0+t(\Gamma'+s(\Gamma-\Gamma'))) (\Gamma'+s(\Gamma-\Gamma'))\|_0 \,dtds \|\Gamma-\Gamma'\|_0 \\
	& \le K^* \frac{1}{\nu I_0} \|\Gamma-\Gamma'\|_0,
\end{aligned}
\]
which, by Lemma~\ref{lem:operadorGflow}, implies that $\wh \G_g \circ \Fext$ is a contraction, with Lipschitz constant bounded by $ K^* \frac{1}{\nu I_0}$, in the ball $\B_{K^* (\nu I_0)^{-1}} \subset \X_0^4$. With the usual argument, we can check that $\wh \G_g \circ \Fext$ sends $\B_{K^* (\nu I_0)^{-1}}$ to itself and, hence, possesses a unique fixed point, $\wt \Gamma_1^+$. This proves the  claim.
\end{proof}

Since $\wt \Gamma_1^+$ is a solution of equation~\eqref{eq:extensiopelfluxpuntfix} ,  $\wt \Gamma^+ = \Gamma_0 + \wt \Gamma_1^+$, which is a solution of~\eqref{eq:invarianciapelfluxapropGamma0}, provides an extension of $\wt \Gamma^+$ to the domain $\D^+_{\hat \kappa,\mathrm{ext}}\times \T_{\hat \sigma}$
(see Figure~\ref{fig:DomRaroextes}).

\subsection{From flow parametrization to Hamilton-Jacobi}

We finally change the parametrization obtained in the previous section to a Hamilton-Jacobi parametrization in the domain $\wt \D^+_{\delta} \times \T_\sigma$,  where
\begin{equation}
\label{def:wtDeltamesdelta}
\wt \D^+_{\delta} = B_\delta^+ \cup \overline{B_\delta^+},
\end{equation}
where the bar indicates complex conjugation, and
\begin{multline*}
B_\delta^+ = \{u \in \C \mid \Im u \ge 0, \\
\max\{0,(1-\delta) - \tan \beta_2 \, \Re u \} < \Im u < \min \{  (1-\delta) + \tan \beta_2 \, \Re u,
(1-\kappa (\nu I_0)^{-1} -\tan \beta_1 \,  \Re u\} \}.
\end{multline*}
See Figure~\ref{fig:Dominibumerang}.

\begin{figure}[h]
\begin{center}
\includegraphics[width=.6\textwidth]{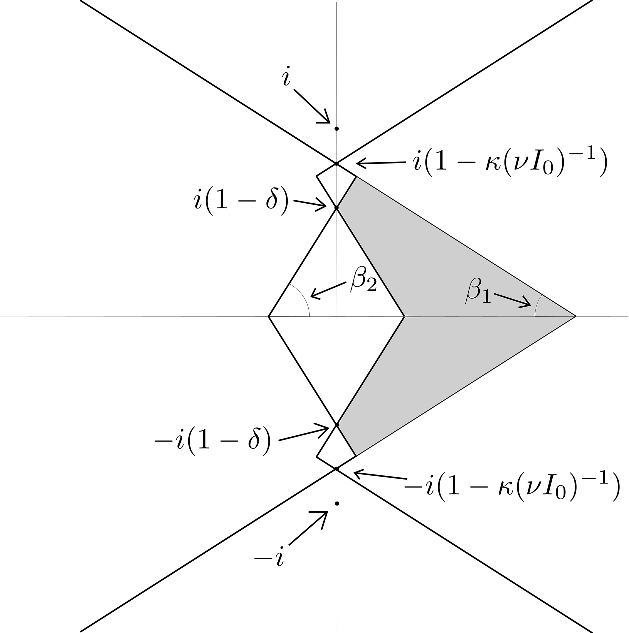}
\end{center}
\caption{The domain $\wt \D^+_{\delta}$ defined in
\eqref{def:wtDeltamesdelta}, in gray. Compare with Figure~\ref{fig:DomRaroextes}.}\label{fig:Dominibumerang}
\end{figure}

We look for a change of variables $\Id + g$, defined in $\wt \D^+_{\delta} \times \T_\sigma$, such that $\Gamma^+ = \wt \Gamma^+ \circ (\Id + g)$ is again a solution of the Hamilton-Jacobi equation~\eqref{eq:HJT1}. We will see that $(\Id +g) \circ (\Id + f) = \Id$, where $\Id + f$ is
given by Proposition~\ref{prop:canvidevariablesdeHJaflux} in their common domain. Hence, the Fourier coefficients of $\Gamma^+$ will be the analytic continuation of the ones already defined in $\D^+_{\kappa,\delta}$ (see~\eqref{def:DominisRaros}).

We recall that $\wt \Gamma^+ = \Gamma_0 + \wt \Gamma_1^+$, where $\Gamma_0$ is defined in~\eqref{def:Gamma0Gamma1Gammames} and
$\wt \Gamma_1^+$ is given by Proposition~\ref{prop:wtGammames}. It is the parametrization of  a Lagrangian manifold. From these two facts we deduce that
$\Gamma^+ = \wt \Gamma^+ \circ (\Id + g)$ is a solution of the Hamilton-Jacobi equation~\eqref{eq:HJT1} if and only if
\begin{equation}
\label{eq:defluxagrafica}
\left\{
\begin{aligned}
q_h (u) & = q_h  (u+g_1(u,\theta)) + \pi_q \wt \Gamma_1^+ \circ (\Id + g) (u,\theta), \\
\theta & = \theta+g_2(u,\theta) + \pi_\theta \wt \Gamma_1^+ \circ (\Id + g) (u,\theta).
\end{aligned}
\right.
\end{equation}

\begin{proposition}
\label{prop:extensiodominibumerang}
Equation~\eqref{eq:defluxagrafica} admits a solution $g\in \X_0$ with $\|g\|_0 \le \OO((\nu I_0)^{-1})$. Furthermore,
$(\Id + f) \circ (\Id + g) = \Id$, where the composition is well defined.
\end{proposition}

\begin{proof}
We rewrite equation~\eqref{eq:defluxagrafica} as the fixed point equation
\begin{equation*}
g = \PP (g),
\end{equation*}
where
\begin{equation}
\PP(g) =
\begin{pmatrix}
(\dot q_h)^{-1} \left( - \pi_q \wt \Gamma_1^+ \circ (\Id + g) - q_h\circ (\Id + g) -q_h - \dot q_h g_1 \right) \\
-\pi_\theta \wt \Gamma_1^+ \circ (\Id + g)
\end{pmatrix}.
\end{equation}
Proposition~\ref{prop:wtGammames} implies that $\|\PP(0)\|_0 = \OO((\nu I_0)^{-1})$. Since $\wt \D^+_{\delta} \times \T_\sigma$ is at a distance
$\OO(1)$ of $\pm i$, the Lipschitz constant of $\PP$ in the ball of radius $\OO((\nu I_0)^{-1})$ is also $\OO((\nu I_0)^{-1})$.

Finally, since the expression of the invariant manifold as a graph is unique,  $(\Id + f) \circ (\Id + g) = \Id$  where the composition is well defined.  This proves the claim.
\end{proof}

\section{The \emph{inner} equation}
\label{sec:equacio_inner}

In order to find a formula for the splitting of the invariant manifolds, we need to find a good approximation of the manifolds
up to distance $(\nu I_0)^{-1}$ of $\pm i$. However, the previous discussion suggests that the unperturbed separatrix is not
close enough to the true manifolds there (since the error term becomes large). It is necessary to find another approximation of the manifolds when the separatrix ceases to be close enough. We perform a local  study around $i$. The behavior at $-i$ will be obtained by the real analyticity of the parametrizations of the invariant manifolds.
To do so, we introduce the new variable $v$ by
\[
u-i = (\nu I_0)^{-1} v,
\]
and the new unknown $S(v,\theta)= (\nu I_0)^{-1} \Phi (i+(\nu I_0)^{-1} v,\theta)$. These changes transform Equation~\eqref{def:HJoriginal} into
\begin{multline}
\label{eq:HJT1innervariables}
\partial_\theta S + \frac{1}{2} \nu (\partial_\theta S)^2  + \frac{1}{2} \frac{v^2(2i+(\nu I_0)^{-1} v)^2}{(i+(\nu I_0)^{-1} v)^2} (\partial_v S)^2 \\
 - \frac{1}{2} \frac{(i+(\nu I_0)^{-1}v)^2}{v^2(2i+(\nu I_0)^{-1}v)^2}
+ \frac{1}{2} \frac{1}{v^2(2i+(\nu I_0)^{-1}v)^2}   V(\theta) = 0.
\end{multline}

Taking limit when $\nu I_0 \to \infty$ in~\eqref{eq:HJT1innervariables} we obtain a new equation which is independent of $I_0$. We call it \emph{inner equation}. It collects the lowest order terms in $(\nu I_0)^{-1}$. It is
\begin{equation}
\label{eq:inner}
\partial_\theta T + \frac{1}{2} \nu (\partial_\theta T)^2 + 2 v^2 (\partial_v T)^2 - \frac{1}{8v^2}
- \frac{1}{8v^2}  V(\theta) = 0.
\end{equation}

By the definition of $\Phi_0$ in~\eqref{def:Phi0}, we have that
\[
(\nu I_0)^{-1} \Phi_0 (i+(\nu I_0)^{-1} v,\theta) = T_0(v) + (\nu I_0)^{-1} \OO(\log((\nu I_0)^{-1} v)),
\]
where
\begin{equation}
	\label{def:T0}
	T_0(v,\theta) = -\frac{1}{4v}.
\end{equation}
Denoting $\E(T)$ the left hand side of \eqref{eq:inner} we have
\begin{equation}
\label{bound:error_term_inner}
\E(T_0)(v,\theta) = \partial_\theta T_0 + \frac{\nu}{2} (\partial_\theta T_0)^2 + 2 v^2 (\partial_v T_0)^2 - \frac{1}{8v^2}
- \frac{1}{8v^2} V(\theta) = - \frac{1}{8v^2}   V(\theta).
\end{equation}

Introducing the new unknown $T_1$ through $T = T_0+T_1$, we can rewrite equation~\eqref{eq:inner} as
\begin{equation}
\label{eq:innerwhT}
\LLin(T_1) = \whFi (T_1),
\end{equation}
where
\begin{equation}
\label{def:LLin}
\LLin T_1 = \partial_v T_1 + \partial_\theta T_1
\end{equation}
and
\begin{equation*}
\whFi (T_1) = - \frac{1}{2} \nu (\partial_\theta T_1)^2- 2v^2 (\partial_v T_1)^2 + \frac{1}{8v^2} V(\theta).
\end{equation*}

To find a suitable solution of equation~\eqref{eq:inner}, we will look for
an auxiliary fixed point equation. However, we will need a better approximation than just $T_0$ of the solution of  \eqref{eq:innerwhT}. 
For that we first define
\begin{equation}
\label{def:Guinner}
\Gi^u(R)(v,\theta) = \int_{-\infty}^0 R (v+s, \theta+   s) \, ds
= \int_{-\infty}^v R (s, \theta+ s-v) \, ds.
\end{equation}
The operator $\Gi^u$ formally satisfies $\LLin \circ \Gi^u(R) = R$.
Using $\Gi^u$, we introduce
\begin{equation}
\label{def:melnikovinner}
\LLi (v,\theta) = \Gi^u \circ \whFi(0) = -\Gi^u (\E(T_0))(v,\theta) =
  \sum_{k\in \Z} \frac{V^{[k]}}{8} e^{ik \theta} \int_{-\infty}^v  \frac{1}{s^2} e^{i k  (s-v)} \,ds.
\end{equation}
We recall that $V^{[0]} = 0$.

Now we introduce the new unknown $T_2$ by $T = T_0+\LLi+T_2$. $T$ is a solution of~\eqref{eq:inner} if and only if $T_2$ is a solution of
\begin{equation}
\label{eq:inner2}
\LLin (T_2) =
\Fi(T_2),
\end{equation}
where
\begin{equation*}
\begin{aligned}
\LLin (T_2) & = \partial_v T_2 +  \partial_\theta T_2, \\
\Fi(T_2) & =-\left(
\frac{\nu}{2}(\partial_\theta \LLi)^2 + 2 v^2 (\partial_v \LLi)^2
+ \nu \partial_\theta \LLi \partial_\theta T_2 +
\frac{\nu}{2}  (\partial_\theta T_2)^2 + 4 v^2 \partial_v \LLi \partial_v T_2 +2 v^2 (\partial_v T_2)^2 \right),
\end{aligned}
\end{equation*}
where we have used that $\LLin(\LLi) = -\E(T_0)$.

\subsection{Spaces and technical lemmas}
To deal with equation \eqref{eq:inner2} we have to introduce some function spaces and provide basic properties of the operator $\LLin$.
We consider the domain
\begin{equation}\label{def:Dominiinner1}
\D^+_{\kappa,\mathrm{in}}=
\{v\in\C \mid\; \Im v <- (\tan\beta_1 )\, \Re v -\kappa \} \cup
\{v\in\C\mid \; \Im v > (\tan\beta_1) \, \Re v + \kappa \},
\end{equation}
with  $\kappa>1$. See Figure~\ref{fig:DomRaroinner}.
We notice that the image of $\D^+_{\kappa,\delta} \cup \D^+_{\kappa,\mathrm{ext}}$ by the transformation $u= i+(\nu I_0)^{-1} v$ is contained in $\D^+_{\kappa,\mathrm{in}}$.
\begin{figure}[h]
\begin{center}
\includegraphics[width=.6\textwidth]{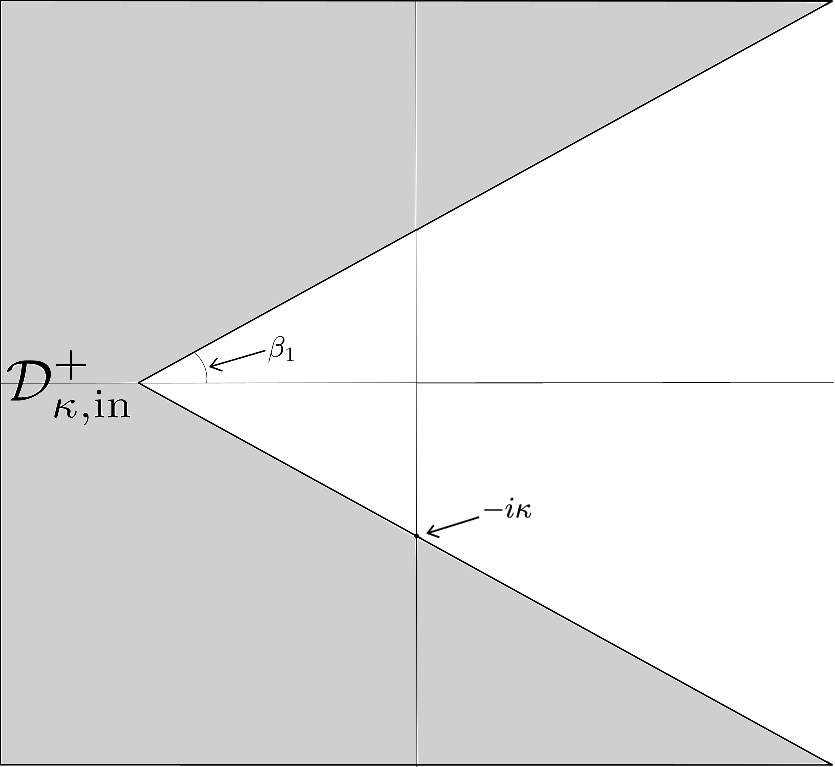}
\end{center}
\caption{The domain $\D^+_{\kappa,\mathrm{in}}$ defined in
\eqref{def:Dominiinner1}, shaded in gray.}\label{fig:DomRaroinner}
\end{figure}

To solve equation~\eqref{eq:inner2}, for $r\in \R$, we introduce  the Banach space of  $2\pi$-periodic in~$\theta$, analytic functions
\begin{equation}
\label{def:Yr}
\Y_r = \{R:\D^+_{\kappa,\mathrm{in}} \times \T_\sigma \to \C \mid  \, \|R\|_{r} < \infty\},
\end{equation}
where, taking into account the Fourier expansion of $R(v,\theta) = \sum_{k\in \Z} R^{[k]} (v) e^{ik\theta}$,
\begin{equation}
\label{def:normYr}
\|R\|_{r} = \sum_{k\in \Z} \|R^{[k]}\|_{r} e^{|k|\sigma},
\end{equation}
and, for an analytic function $f:\D^+_{\kappa, \mathrm{in}} \to \C$,
\[
\| f\|_{r} = \sup_{v \in \D^+_{\kappa,\mathrm{in}}} |v^r f(v)|.
\]

\begin{lemma}
\label{lem:propietatsespaisYr}
Let $r,r_1, r_2\in \R$. 
\begin{enumerate}
\item[(1)]
If $R \in \Y_{r+s}$, $s \ge 0$, then $R \in \Y_{r}$ and there exists $K>0$ such that
\[
\|R\|_{r} \le K \frac{1}{\kappa^s}\|R\|_{r+s}.
\]
\item[(2)]
If $R_1 \in \Y_{r_1}$ and $R_2 \in \Y_{r_2}$, then $R_1R_2 \in \Y_{r_1+r_2}$ and
\[
\|R_1 R_2 \|_{r_1+r_2} \le \|R_1  \|_{r_1}\|R_2 \|_{r_2}.
\]
\end{enumerate}
\end{lemma}

We also introduce the Banach space
\begin{equation}
\label{def:tildeYR}
\widetilde \Y_{r} = \{R\in \Y_r \mid\, \partial_v R , \, \partial_\theta R\in \X_{r+1},\, \lln R \rrn_{r} < \infty\},
\end{equation}
with the norm
\begin{equation}
\label{def:normapotetesinner}
\lln R \rrn_{r} = \| R\|_{r} + \|\partial_v R \|_{r+1}+\|\partial_\theta R\|_{r+1}.
\end{equation}

We will use the next technical lemma several times.
It is analogous to Lemma~\ref{lem:operadorGu}.

\begin{lemma}
\label{lem:operadorGuinner}
Let $\Gi^u$ be the operator defined in~\eqref{def:Guinner}. There exists $K>0$ such that
\begin{enumerate}
\item[(1)]
If $R \in \Y_{r}$ with $r>0$ and $\langle R \rangle = 0$, then
$\Gi^u(R) \in \Y_r$ and
\[
\|\Gi^u(R)\|_{r} \le K \|R\|_{r}.
\]
\item[(2)]
If $R \in \Y_{r}$ with $r> 1$, then $\Gi^u(R)\in \Y_{r-1}$ and
\[
\|\Gi^u(R)\|_{r-1} \le K \|R\|_{r}.
\]
\item[(3)]
If $R \in \Y_{r}$ with $r>0$, then $\partial_v \Gi^u(R), \partial_\theta \Gi^u(R) \in \Y_{r}$ and
\[
\|\partial_v \Gi^u(R)\|_{r} \le K \|R\|_{r}, \qquad
\|\partial_\theta \Gi^u(R)\|_{r} \le K \|R\|_{r}.
\]
\end{enumerate}
As a consequence, if $R \in \Y_{r}$, with $r> 1$, $\Gi^u(R) \in \wt \Y_{r-1}$ and
\[
\lln \Gi^u(R) \rrn_{r-1} \le K \|R\|_{r}.
\]
The constant $K$ only depends on $r$ and the constants involved in the definition of $\D^+_{\kappa,\mathrm{in}}$.
\end{lemma}

\subsection{Fixed point equation}

We consider the equation
\begin{equation}
\label{eq:equacio_inner_punt_fix}
T_2 = \Gi^u \circ \Fi (T_2).
\end{equation}
We decompose 
\begin{equation}
\Fi (T_2) = -(
F_0+F_1+F_2+F_3+F_4),
\end{equation}
with
\[
\begin{aligned}
F_0(v,\theta)  & = \frac{\nu}{2}(\partial_\theta \LLi(v,\theta))^2 +
2 v^2 (\partial_v \LLi(v,\theta))^2,\\
F_1(v,\theta) &= \nu \partial_\theta \LLi (v,\theta)\partial_\theta T_2(v,\theta), \\
F_2(v,\theta) &= \frac{\nu}{2}  (\partial_\theta T_2(v,\theta))^2 , \\
F_3(v,\theta) &=  4 v^2 \partial_v \LLi (v,\theta)\partial_v T_2(v,\theta), \\
F_4(v,\theta) &= 2 v^2 (\partial_v T_2(v,\theta))^2.
\end{aligned}
\]
If $T_2$ is a solution of the fixed point equation \eqref{eq:equacio_inner_punt_fix}
it is also a solution of~\eqref{eq:inner2}.
In view of~\eqref{def:melnikovinner}, since $V^{[0]} = 0$, we have that $(\LLi)^{[0]} = 0$. The following proposition is analogous to Proposition~\ref{prop:primeraiteracio}. Its proof follows exactly the same lines.

\begin{proposition}
\label{prop:migmelnikovinner}
We have
$\LLi, \partial_\theta \LLi \in \Y_{2}$ and $\partial_v \LLi \in \Y_{3}$. Moreover, there exists $ K>0$, independent of $\kappa$, such that
$\|\LLi\|_2, \|\partial_\theta \LLi\|_2,  \|\partial_v \LLi\|_3\le  K$.
\end{proposition}

Motivated by Proposition~\ref{prop:migmelnikovinner}, we introduce the constant
\begin{equation}
\label{def:mida_Melnikov_inner}
\nLi = \left(\|\LLi\|_2^2+\|\partial_\theta \LLi\|_2^2 +  \|\partial_v \LLi\|_3^2\right)^{1/2}.
\end{equation}
\begin{remark}
If one replaces $r_i$ by $\varepsilon \tilde r_i$ in the definition of $V$ in~\eqref{def:potencialdeMorsecorrugat},
then $\nLi = \OO(\varepsilon)$.
\end{remark}

\begin{proposition}
\label{prop:solucions_inner}
There exist $K_1 >0$ and $\kappa_0$  such that, if $\kappa > \kappa_0$,
the operator $\Gi^u \circ \Fi : \B_{K_1 \nLi^2 }\subset \widetilde \Y_{3} \to \B_{K_1 \nLi^2}$ is well defined and a contraction. Hence,
equation~\eqref{eq:equacio_inner_punt_fix} has a solution $T_2^+ \in  \widetilde \Y_{3}$ satisfying
\begin{equation}\label{fita:T2entermesdeL}
\lln T_2^+ \rrn_3 \le K_1 \nLi^2.
\end{equation}
\end{proposition}
\begin{proof}
We start by bounding $F_0$. By Proposition~\ref{prop:migmelnikovinner},
\begin{equation}
\label{normaF0}
\|F_0\|_4 \le  \left\|\frac{\nu}{2}(\partial_\theta \LLi)^2\right\|_4+
2 \| v^2 (\partial_v \LLi)^2\|_4 \le
\frac{\nu}{2} \|\partial_\theta \LLi\|_2^2 + 2 \|v^2\|_{-2} \|\partial_v \LLi\|_3^2 < K_2 \nLi^2,
\end{equation}
for some $K_2>0$.
We take $K_1 = 2\|\Gi^u\|  K_2  $ and $K^* =  K_1 \nLi^2$. In this proof $\|\Gi^u\|$ stands for the operator norm of $\Gi^u:\Y_{4} \to \wt \Y_3$.
Let $T_2 \in \B_{K^*} \subset \wt \Y_3 $.
We claim that $\Fi (T_2) \in   \Y_{4}$ and
\begin{equation}
\label{bound:FdeR}
\|\Fi (T_2)\|_4 \le \|F_0\|_4+\left(K\frac{\nu \|\partial_\theta \LLi\|_2  }{\kappa^{2}} + K\frac{\nu K^*}{2\kappa^{4}} + K \frac{4\| \partial _v \LLi\|_3 }{\kappa}
+K\frac{2K^*}{\kappa^2}\right) K^*.
\end{equation}
Indeed, next we deal with $F_j$, $j=1,\dots,4$. By Proposition~\ref{prop:migmelnikovinner} and
Lemma~\ref{lem:propietatsespaisYr}
\[
\|F_1\|_{4} = \|\nu \partial_\theta \LLi \partial_\theta T_2\|_{4} = \nu \|\partial_\theta \LLi\|_0
\|\partial_\theta T_2\|_{4} \le K \frac{\nu \|\partial_\theta \LLi\|_2 }{\kappa^{2}}\lln T_2\rrn_3 \le K\frac{\nu \|\partial_\theta \LLi\|_2 K^* }{\kappa^{2}},
\]
since $\|\partial_\theta T_2\|_{4}\le \lln T_2\rrn_{3} \le K^*$. Again by Lemma~\ref{lem:propietatsespaisYr},
\[
\|F_2\|_{4} = \left\|\frac{\nu}{2}  (\partial_\theta T_2(v,\theta))^2\right\|_{4} \le \frac{\nu}{2 } \|\partial_\theta T_2\|_0
\|\partial_\theta T_2\|_{4} \le K\frac{\nu}{2 \kappa^{4}} \|\partial_\theta T_2\|_{4}
\|\partial_\theta T_2\|_{4} \le K
\frac{\nu (K^*)^2}{2\kappa^{4}}.
\]
Now, by Proposition~\ref{prop:migmelnikovinner}, Lemma~\ref{lem:propietatsespaisYr} and the fact that $\|\partial_v T_2\|_{4}\le \lln T_2\rrn_{3} \le K^* $,
\[
\|F_3\|_{4} = \left\| 4 v^2 \partial_v \LLi \partial_v T_2\right\|_{4} \le
4 \left\| v^2 \|_{-2}\|\partial_v \LLi\|_2 \|\partial_v T_2\right\|_{4} \le
K \frac{4\| \partial _v \LLi\|_3 }{\kappa} K^*.
\]
Finally,
\begin{align*}
\|F_4\|_{4} & = \left\| 2 v^2 (\partial_v T_2(v,\theta))^2\right\|_{4} \le
\left\| 2 v^2 \partial_v T_2(v,\theta)\right\|_{0}\left\| \partial_v T_2(v,\theta)\right\|_{4}
 \\ & \le
\| 2 v^2 \|_{-2}\|\partial_v T_2(v,\theta)\|_{2}\| \partial_v T_2(v,\theta)\|_{4} \le
K \frac{2}{\kappa^2} \|\partial_v T_2(v,\theta)\|_{4 }\|\partial_v T_2(v,\theta)\|_{4 }
\le
K \frac{2(K^*)^2}{\kappa^2}.
\end{align*}
This proves~\eqref{bound:FdeR}. Now we choose $\kappa_0 >1$ such that
\[
K \left(\frac{\nu \|\partial_\theta \LLi\|_2  }{\kappa_0^{2}} + \frac{\nu K^*}{2\kappa_0^{4}} + \frac{4\| \partial_v\LLi\|_3 }{\kappa_0}
+\frac{2K^*}{\kappa_0^2}\right) K^*<  K_2 \nLi^2.
\]
Hence, for any $\kappa > \kappa_0$, if $T_2\in \B_{K^*}$, by the last claim of Lemma~\ref{lem:operadorGuinner},
\[
\lln \Gi^u \circ \Fi(T_2)\rrn_{3} \le \|\Gi^u\| \|\Fi(T_2)\|_4
\le  2\|\Gi^u\| K_2 \nLi^2  = K^*.
\]
Next we check that $\Gi^u \circ \Fi$ is a contraction on $\B_{K^*}$.
We claim that there exists $ K_3>0$ such that, if $T_2, T'_2 \in \B_{K^*}$,
\begin{equation}
	\label{bound:FinnerLipschitz}
	\|\Fi(T_2)-\Fi(T'_2)\|_{4} \le \frac{K_3}{\kappa} \lln T_2 - T'_2 \rrn_{3}.
\end{equation}
To prove this claim, we write
\[
\Fi(T_2)-\Fi(T'_2) = -\sum_{i=1}^4 \left(\F_i(T_2)-\F_i(T'_2) \right),
\]
where
\[
\begin{aligned}
	\F_1(R) &= \nu \partial_\theta \LLi \partial_\theta R, \\
	\F_2(R) &= \frac{1}{2} \nu (\partial_\theta R)^2 , \\
	\F_3(R)(v,\theta) &= 4v^2 \partial_v \LLi(v,\theta) \partial_v R(v,\theta), \\
	\F_4(R)(r,\theta) &= 2 v^2 (\partial_v R(v,\theta))^2.
\end{aligned}
\]
We bound each difference separately. First, by Proposition~\ref{prop:migmelnikovinner} and Lemma~\ref{lem:propietatsespaisYr}, we have that
\[
\|\F_1(T_2)-\F_1(T'_2)\|_{4}  \le \|\nu \partial_\theta \LLi\|_0 \|\partial_\theta T_2 -\partial_\theta T'_2\|_{4} \le K \frac{\nu \| \partial_\theta \LLi\|_2}{\kappa^2} \lln T_2-T'_2 \rrn_{3}.
\]
Second,
\[
\begin{aligned}
	\|\F_2(T_2)-\F_2(T'_2)\|_{4} &  =  \frac{\nu}{2} \|  (\partial_\theta T_2+\partial_\theta T'_2) (\partial_\theta T_2 -\partial_\theta T'_2)\|_{4} \\
	& \le \frac{\nu}{2}\| (\partial_\theta T_2+\partial_\theta T'_2)\|_0 \|\partial_\theta T_2 -\partial_\theta T'_2\|_{4} \\
	& \le K \frac{\nu}{2} \frac{1}{\kappa^{4}} \|\partial_\theta T_2 +\partial_\theta T'_2\|_{4}\|\partial_\theta T_2 -\partial_\theta T'_2\|_{4} \\
	& \le  \frac{K K^* \nu}{\kappa^{4}} \lln T_2 - T'_2 \rrn_{3}.
\end{aligned}
\]
Third, using again Proposition~\ref{prop:migmelnikovinner} and Lemma~\ref{lem:propietatsespaisYr},
\[
\|\F_3(T_2)-\F_3(T'_2)\|_{4}  \le \left\|4v^2 \partial_v \LLi\right\|_0\|\partial_v T_2 - \partial_v T'_2\|_{4}
\le \frac{4\left\|\partial_v \LLi\right\|_3}{\kappa} \lln T_2 - T'_2 \rrn_{3}.
\]
And fourth,
\[
\begin{aligned}
	\|\F_4(T_2)-\F_4(T'_2)\|_{4}  & \le \|2 v^2 (\partial_v T_2 + \partial_v T'_2)\|_0\|\partial_v T_2 - \partial_v T'_2\|_{4} \\
	& \le \|2 v^2\|_{-2} \|\partial_v T_2 + \partial_v T'_2\|_2 \|\partial_v T_2 - \partial_v T'_2\|_{4} \\
	& \le K\frac{2}{\kappa^2} \|\partial_v T_2 + \partial_v T'_2\|_{4} \|\partial_v T_2 - \partial_v T'_2\|_{4} \\
	& \le \frac{4K K^*}{\kappa^2}\lln  T_2 -  T'_2\rrn_{3}.
\end{aligned}
\]
Inequality~\eqref{bound:FinnerLipschitz} follows from combining these four steps, since $\kappa > \kappa_0>1$ and $\kappa_0$ is big.

Finally, we check that the operator $\Gi^u \circ \Fi$ is a contraction in $\B_{K^*}$. Indeed, if $T_2, T'_2 \in \B_{K^*}$, using the last statement of Lemma~\ref{lem:operadorGuinner} and~\eqref{bound:FinnerLipschitz},
\[
\lln \Gi^u \circ \Fi(T_2) - \Gi^u \circ \Fi(T'_2)\rrn_{3}
 \le \|\Gi^u\| \| \Fi(T_2) -  \Fi(T'_2)\|_{4}
\le\|\Gi^u\| \frac{  K_3 }{\kappa} \lln T_2 -T'_2\rrn_{3}.
\]
Hence, by the standard argument, taking $\kappa_0$ large enough, $\Gi^u \circ \Fi$ sends $\B_{K^*}$ into itself and has a unique fixed point $T_2^+$ in this ball. The bound \eqref{fita:T2entermesdeL} follows from the definition
of $K^*$.
\end{proof}

\begin{corollary}
\label{cor:solutioninnerequation}
The inner equation~\eqref{eq:inner} admits a solution
\[
T^+ = T_0 + \LLi + T_2^+,
\]
where $T_0$ is given by~\eqref{def:T0}, $\LLi$ is defined in~\eqref{def:melnikovinner} and $T_2^+$ is given by Proposition~\ref{prop:solucions_inner}.
\end{corollary}

\section{Approximation of the manifold in the inner domain}
\label{sec:extensio_al_domini_inner}

Let $T^+= T_0  + \LLi+ T_2^+$ be the solution of the inner equation~\eqref{eq:inner} in the domain
$\D^+_{\kappa,\mathrm{in}}$
given by Corollary~\ref{cor:solutioninnerequation}.

From Proposition~\ref{prop:aproximacioperlahomoclinicaversio2}, we have that the solution $\Phi^+(u,\theta)$ of the Hamilton-Jacobi equation~\eqref{def:HJoriginal} and its extension to the domain
$\D^+_{\kappa,\delta} \cup \D^+_{\kappa,\mathrm{ext}}\setminus B^*_\rho$ (where $B^*_\rho$ is a fixed small ball centered at the origin),
with vanishing contour conditions at $\Re u=-\infty$,
  is given by
\[
\Phi^+
 = \Phi_0 + \LLo + \Phi_2^+,
\]
where $\Phi_0$ was introduced in~\eqref{def:Phi0}, $\LLo$ in \eqref{def:Loutmes}  and $\Phi_2^+ \in \wt \X_{5,3}$, hence
\[
\Phi_2^+(u,\theta) \sim \frac{(\nu I_0)^{-2}}{  (u-i)^3}.
\]
The approximation of $\Phi^+$ by $\Phi_0+ \LLo$ is not good enough if $u-i \sim (\nu I_0)^{-1}$, although Proposition~\ref{prop:wtGammames} ensures that the manifolds can be extended to $\D^+_{\kappa,\delta} \cup \D^+_{\kappa,\mathrm{ext}}\setminus B^*_\rho$. In order to obtain a better approximation of $\Phi^+$, we compare it with the solution of the inner equation~\eqref{eq:inner}, $T^+ =T_0+ \LLi + T_2^+$, given by Corollary~\ref{cor:solutioninnerequation}.

We recall that the function
\[
S^+(v,\theta)= (\nu I_0)^{-1} \Phi^+ (i+(\nu I_0)^{-1} v,\theta)
\]
is a solution of the  Hamilton-Jacobi equation~\eqref{eq:HJT1innervariables} in the transformed domain.

Actually, we want to compare $S^+$ with the solution $T^+$ of the \textit{inner equation}. Both functions are already determined. We write 
$$
S^+ = T_0+ \LLi+T_2^+ + S_2^+
$$ 
and derive bounds and properties of $S_2^+$ from the equation it satisfies.

Using that $T^+= T_0+ \LLi+T_2^+$ is a solution of the \textit{inner equation}~\eqref{eq:inner}, we have that $S_2^+$ is a solution of
\begin{equation}
\label{eq:HJdifferenceinnervariables}
\LLin(S_2) = \Fii(S_2),
\end{equation}
where $\LLin(S_2) = \partial _vS_2 + \partial _\theta S_2$ was defined in~\eqref{def:LLin} and
\begin{equation}
\label{def:LiFi}
\Fii(S_2)  = -\left(\wt \E + \nu \partial_\theta T^+ \partial_\theta S_2 +
\frac{1}{2} \nu (\partial_\theta S_2)^2 +(2 A_1 \partial_v T^+ -1) \partial_v S_2 +  A_1 (\partial_v S_2) ^2 \right)
\end{equation}
with
\begin{equation}
\label{def:wteA1A2}
\begin{aligned}
A_1(v) & = \frac{1}{2} \frac{v^2(2i+(\nu I_0)^{-1} v)^2}{(i+(\nu I_0)^{-1} v)^2}, \\
A_2(v,\theta) & = -\frac{1}{2} \frac{(i+(\nu I_0)^{-1}v)^2}{v^2(2i+(\nu I_0)^{-1}v)^2}
+ \frac{1}{2} \frac{1}{v^2(2i+(\nu I_0)^{-1}v)^2}  V(\theta), \\
\wt \E (v,\theta)& = (A_1(v)-2v^2) (\partial_v T_0^+)^2 + A_2(v,\theta) +\frac{1}{8v^2}+\frac{1}{8v^2} V(\theta).
\end{aligned}
\end{equation}

Let $\beta_1 >0$ be the angle involved in the definition of $\D^+_{\kappa,\delta}$ (see~\eqref{def:DominisRaros}). Fix $\beta_3$ and $\beta_4$ satisfying  $0 < \beta_3 < \beta_4 < \beta_1$. For $\alpha \in(0,1)$ and $\kappa >1$, we define the domain
\begin{equation}\label{def:DominisRaros2}
	\D^+_{\kappa,\alpha}=
	\{v\in\C \mid\; -(\nu I_0)^{\alpha} - \tan \beta_4 \,\Re v < \Im v < \min \{-\kappa - \tan \beta_1 \,\Re u, -\kappa - \tan \beta_3 \,\Re u\}\}.
\end{equation}
See Figure~\ref{fig:DomRaro2}. Observe that $d(\D^+_{\kappa,\alpha}, 0) = \kappa \cos \beta_3 $ and $|v| \le K  (\nu I_0)^{\alpha}$ for $v \in \D^+_{\kappa,\alpha}$.

\begin{figure}[h]
	\begin{center}
		\includegraphics[width=.6\textwidth]{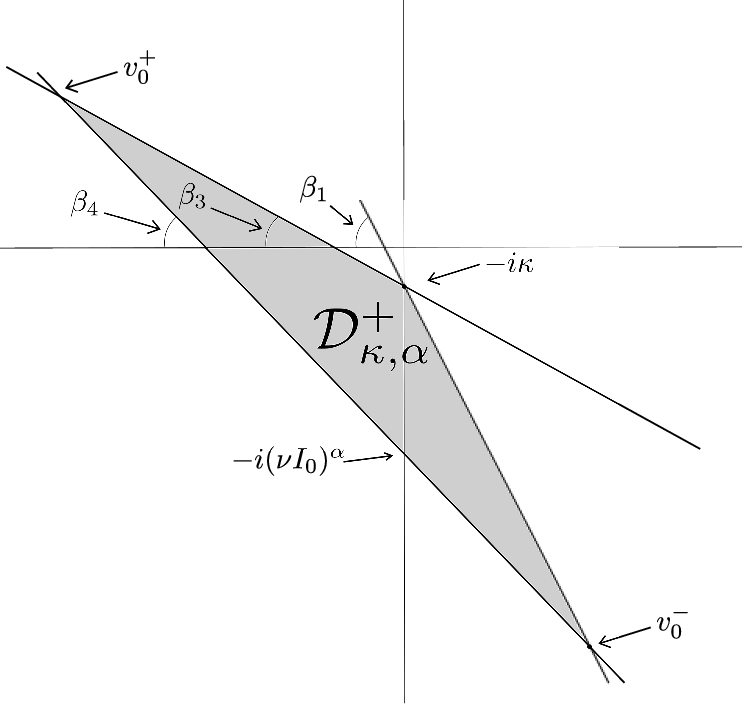}
	\end{center}
	\caption{The domain  $\D^+_{\kappa,\alpha}$ defined in
		\eqref{def:DominisRaros2}, shaded in gray.}\label{fig:DomRaro2}
\end{figure}

In order to solve equation~\eqref{eq:HJdifferenceinnervariables}, we consider the following right inverses, $\wh \G^u$ and
$\wh \G_s^u$, of $\LLin$. First, $\wh \G^u$ is determined by the expression of the Fourier coefficients of the image
$\wh \G^u(S_2)$:
\begin{equation}
	\label{def:whG}
	\begin{aligned}
		\wh \G^u(S_2)^{[k]}(v) & = \int_{v_0^+}^v e^{ik (t-v)} S_2^{[k]}(t)\,dt \qquad \text{if \; $k\ge 0$}, \\
		\wh \G^u(S_2)^{[k]}(v) & = \int_{v_0^-}^v e^{ik(t-v)} S_2^{[k]}(t)\,dt \qquad \text{if \; $k< 0$},
	\end{aligned}
\end{equation}
where
\begin{equation*}
	\begin{aligned}
		v_0^- & =  \frac{(\nu I_0)^\alpha-\kappa}{\tan \beta_1-\tan \beta_4}
		+ \left(-\left( \frac{\tan \beta_4}{\tan \beta_1-\tan \beta_4}+1\right) (\nu I_0)^\alpha
		+ \frac{\tan \beta_4}{\tan \beta_1-\tan \beta_4} \kappa\right) i,  \\
		v_0^+ & = \frac{-(\nu I_0)^\alpha+\kappa}{\tan \beta_4-\tan \beta_3}
		+ \left(\left( \frac{\tan \beta_4}{\tan \beta_4-\tan \beta_3}-1\right) (\nu I_0)^\alpha
		- \frac{\tan \beta_4}{\tan \beta_4-\tan \beta_3} \kappa\right) i,
	\end{aligned}
\end{equation*}
with $\alpha \in(0,1)$ and $0 < \beta_3 < \beta_4 < \beta_1 < \pi$. See Figure \ref{fig:DomRaro2}. 
The most general right inverse of  $\LLin$ has the form 
\begin{equation}
\label{def:Gtilde}
\wh \G^u_s(S_2) = S_{2,0} + \wh \G^u(S_2),
\end{equation}
with $S_{2,0} \in \ker \LLin$.
With this operator we consider the fixed point equation
\begin{equation} \label{eq:puntfixS2}
	S_2 = \wh \G^u_s \circ \Fii(S_2)
\end{equation} 
and we choose $ S_{2,0} $ so that the solution $S^+_2$ of \eqref{eq:puntfixS2} is the analytical continuation of 
$S^+ -T_0-\LLi-T_2^+$ to the domain $\D^+_{\kappa , \alpha} $. Note again that  if $S_2^+$ is a fixed point of $\wh \G^u_s \circ \Fii$, it is also a solution of \eqref{eq:HJdifferenceinnervariables}.

Since $S_{2,0} \in \ker \LLin$, 
\[
S_{2,0}(v,\theta) = \sum_{k\in \Z} S_{2,0}^{[k]}(v) e^{ik\theta} = \sum_{k\in \Z} g^{[k]} e^{ik(\theta-v)},
\] 
which defines the coefficients $g^{[k]}$. 
If, moreover,  $S_{2}$ is a solution of \eqref{eq:puntfixS2}, 
we have
\[
S_2(v,\theta)  = S_{2,0}(v,\theta) + \wh \G^u \circ \Fii(S_2) (v,\theta).
\]
For $k\ge 0$, the Fourier coefficients satisfy  
$$
S_2^{[k]}(v) = S_{2,0}^{[k]}(v)  + \int_{v_0^+}^v e^{ik (t-v)} (\Fii(S_2))^{[k]}(t)\,dt
$$
and, taking into account that, at the points $(v_0^+,\theta )$ of the boundary,  
$$
S_2(v,\theta)=
S^+(v,\theta) -T_0(v,\theta)-\LLi(v,\theta)-T_2^+(v,\theta),
$$ 
evaluating at $v=v_0^+$, we get
$$
 g^{[k]} =  S_{2,0}^{[k]}(v_0^+)e^{ik v_0^+}  = (S^+ -T_0-\LLi-T_2^+)^{[k]}(v_0^+)  e^{ik v_0^+} 
$$
and finally, 
\begin{equation} \label{defSkpos}
S^{[k]}_{2,0}(v) = g^{[k]} e^{-ik v} = (S^+ -T_0-\LLi-T_2^+)^{[k]}(v_0^+) e^{-ik (v-v_0^+ )}.
\end{equation}
Analogously, we obtain that for $k<0$
\begin{equation} \label{defSkneg}
S^{[k]}_{2,0}(v) = g^{[k]} e^{-ik v} = (S^+ -T_0-\LLi-T_2^+)^{[k]}(v_0^-) e^{-ik (v-v_0^- )}.
\end{equation}
Therefore, we choose 
\begin{equation} \label{def:V0}
S_{2,0} (v, \theta) = \sum_{k\in \Z} S_{2,0}^{[k]}(v) e^{ik\theta}
\end{equation}
with  the coefficients \eqref{defSkpos} and \eqref{defSkneg}.

\subsection{Spaces and technical lemmas}

To solve equation~\eqref{eq:puntfixS2}, for $r\in \R$, we introduce  the Banach space of analytic, $2\pi$-periodic in~$\theta$ functions
\begin{equation*}
\ZZ_r = \{R:\D^+_{\kappa,\alpha} \times \T_\sigma \to \C \mid  \, \|R\|_{r} < \infty\},
\end{equation*}
with the norm $ \|\cdot\|_{r} $ defined as follows. 
Using that $R(v,\theta) = \sum_{k\in \Z} R^{[k]} (v) e^{ik\theta}$,
\[
\|R\|_{r} = \sum_{k\in \Z} \|R^{[k]}\|_{r} e^{|k|\sigma},
\]
where, for an analytic function $f:\D^+_{\kappa,\alpha} \to \C$,
\[
\| f\|_{r} = \sup_{v \in \D^+_{\kappa,\alpha}} |v^r f(v)|.
\]

The following lemma is very similar to~Lemma~\ref{lem:propietatsespaisYr}.

\begin{lemma}
\label{lem:propietatsespaisZr}
Let $r, r_1, r_2 \in \R$.
There exists $K>0$ such that the following holds.
\begin{enumerate}
\item[(1)]
If $R \in \ZZ_{r+s}$ with $s \ge 0$, then $R \in \ZZ_{r}$ and
\[
\|R\|_{r} \le K \frac{1}{\kappa^s}\|R\|_{r+s}.
\]
\item[(2)]
If $R_1 \in \ZZ_{r_1}$ and $R_2 \in \ZZ_{r_2}$, then $R_1R_2 \in \ZZ_{r_1+r_2}$ and
\[
\|R_1 R_2 \|_{r_1+r_2} \le \|R_1  \|_{r_1}\|R_2 \|_{r_2}.
\]
\end{enumerate}
\end{lemma}

We also introduce
\begin{equation*}
\widetilde \ZZ_{r} = \{R \in \ZZ_r \mid  \, \partial_v R, \,\partial_\theta R \in \ZZ_{r+1},\, \lln R \rrn_{r} < \infty\},
\end{equation*}
with the norm
\begin{equation*}
\lln R \rrn_{r} = \| R\|_{r} + \| \partial_v R\|_{r+1}+\|\partial_\theta R\|_{r+1}.
\end{equation*}

The following result is completely analogous to Lemma~\ref{lem:operadorGuinner}.
The only difference is the domain of the functions of the space.
\begin{lemma}
\label{lem:operadorGuinnermatching}
Let $\wh \G^u$ be the operator defined in~\eqref{def:whG}.
There exists $K>0$ such that 
\begin{enumerate}
\item[(1)]
If $R \in \ZZ_{r}$, with $r>0$, and $R^{[0]} = 0$,
\[
\|\wh \G^u(R)\|_{r} \le K \|R\|_{r}.
\]
\item[(2)]
If $R \in \ZZ_{r}$ with $r> 1$, $\wh \G^u(R)\in \ZZ_{r-1}$ and
\[
\|\wh \G^u(R)\|_{r-1} \le K \|R\|_{r}.
\]
\item[(3)]
If $R \in \ZZ_{r}$ with $r>0$, then $\partial_v \wh \G^u(R), \partial_\theta \wh \G^u(R) \in \ZZ_{r}$ and
\[
\|\partial_v \wh \G^u(R)\|_{r} \le K \|R\|_{r}, \qquad
\|\partial_\theta \wh \G^u(R)\|_{r} \le K \|R\|_{r},
\]
\end{enumerate}
As a consequence, if $R \in \ZZ_{r}$, with $r> 1$, $\wh \G^u(R) \in \wt \ZZ_{r-1}$ and
\[
\lln \wh \G^u(R) \rrn_{r-1} \le K \|R\|_{r}.
\]
The constant $K$ only depends on $r$ and the constants involved in the definition of $\D^+_{\kappa,\alpha}$.
\end{lemma}

\subsection{The fixed point equation}

Before solving the fixed point equation \eqref{eq:puntfixS2} we deal with some preliminary estimates.

\begin{lemma}
\label{lem:boundsforthefixedpointextensioninnerdomain}
Let $A_1$ and $\wt \E$ be the functions introduced in~\eqref{def:wteA1A2}.
There exists $K>0$ such that, in $\D^+_{\kappa,\alpha}$, we have 
\begin{enumerate}
\item[(1)]
$\wt \E \in \ZZ_{2}$ with $\|\wt \E\|_2 \le K  (\nu I_0)^{\alpha -1}$.
\item[(2)]
$A_1 \in \ZZ_{-2}$ with $\|A_1\|_{-2} \le K  $.
\item[(3)]
$2 A_1 \partial_v T^+ -1 \in \ZZ_0$ with $\|2 A_1 \partial_v T^+ -1 \|_0 
\le K( \nu I_0)^{\alpha-1}+K \kappa^{-1}$.
\item[(4)]
$\partial_\theta T^+ \in \ZZ_2$ with $\|\partial_\theta T^+\|_2 \le K$.
\end{enumerate}
\end{lemma}

\begin{proof}
We recall that in the domain $\D^+_{\kappa,\alpha}$,
$K_1 \kappa < |v| < K_2(\nu I_0)^{\alpha}$, for some $K_1, K_2>0$.
We start by proving (1).
We have that
\begin{equation}
\label{eq:A1menys2v2}
A_1(v) - 2 v^2  = \frac{1}{2} \frac{v^2(2i+(\nu I_0)^{-1} v)^2}{(i+(\nu I_0)^{-1} v)^2} - 2 v^2  = 2 v^2 \frac{ (\nu I_0)^{-1} v(i  + \frac{3}{4}(\nu I_0)^{-1} v) }{(1-i(\nu I_0)^{-1} v)^2},
\end{equation}
which implies that  $\|A_1(v) - 2 v^2\|_{-2} \le K(\nu I_0)^{\alpha-1}$. Recall that  $0<\alpha<1$.

We claim that there exists a constant $ K$ such that $\|\partial_v T^+\|_2 \le  K$.
Indeed, since  $T^+ = T_0+\LLi+T_2^+$, by Propositions~\ref{prop:migmelnikovinner} and~\ref{prop:solucions_inner}, $\|\partial_v \LLi\|_3 \le K$ and $\lln T_2^+\rrn_{3} \le K$.
Hence, by (1) of Lemma~\ref{lem:propietatsespaisZr},
$\|\partial_v \LLi\|_2 \le K \kappa^{-1}$ and
$\|\partial_v T_2^+\|_2 \le  K \kappa^{-2}$. Since $\partial_v T_0 = 1/(4v^2)$, the claim follows immediately from the previous bounds.
Then,
by (2) of Lemma~\ref{lem:propietatsespaisZr},
\[
\left\|(A_1(v) - 2 v^2) (\partial_v T^+)^2\right\|_2  \le
\|A_1(v) - 2 v^2\|_{-2}\| \partial_v T^+\|_2^2   \le  K^3 (\nu I_0)^{\alpha-1}.
\]

By the definition of $A_2$ in~\eqref{def:wteA1A2}, we have that
\[
\begin{aligned}
A_2(v,\theta) +\frac{1}{8v^2}+\frac{1}{8v^2} V(\theta) & =
-\frac{1}{2} \frac{(i+(\nu I_0)^{-1}v)^2}{v^2(2i+(\nu I_0)^{-1}v)^2}
+\frac{1}{8v^2}
+ \left(\frac{1}{2} \frac{1}{v^2(2i+(\nu I_0)^{-1}v)^2} +\frac{1}{8v^2}\right) V(\theta) \\
& = \frac{1}{v^2} \OO((\nu I_0)^{-1} v)
\end{aligned}
\]
which implies
\[
\left\|A_2(v,\theta) +\frac{1}{8v^2}+\frac{1}{8v^2} V(\theta)\right\|_2 \le K  (\nu I_0)^{\alpha -1}.
\]
Hence, (1) follows. (2) is an immediate consequence of~\eqref{eq:A1menys2v2}.

Now we deal with (3). We have that
\[
2A_1 \partial_v T^+ -1 = 2A_1 \partial_v T_0 -1 + 2A_1 \partial_v(\LLi+T_2^+)
= -\frac{1}{4} \frac{(\nu I_0)^{-1}v (4i+3(\nu I_0)^{-1}v)}{(i+(\nu I_0)^{-1} v)^2}+ 2A_1 \partial_v(\LLi+T_2^+)
\]
which implies that
\[
\begin{aligned}
\|2A_1 \partial_v T^+ -1\|_0 & \le
\|2A_1 \partial_v T_0 -1\|_0 + \|2A_1 \partial_v (\LLi+T_2^+) \|_0\\
& \le K(\nu I_0)^{\alpha-1} + \|2A_1\|_{-2} \|\partial_v (\LLi+T_2^+) \|_{2} \\
& \le K(\nu I_0)^{\alpha-1} + \frac{K}{\kappa} \|\partial_v (\LLi+T_2^+ ) \|_{3}.
\end{aligned}
\]
(4) Follows directly from Propositions  \ref{prop:migmelnikovinner} and \ref{prop:solucions_inner}.
\end{proof}

\begin{proposition}
\label{prop:normofV0}
The function $S_{2,0}$ defined in~\eqref{def:V0} satisfies $S_{2,0} \in \ZZ_0$
and
\[
\|S_{2,0}^{[k]}\|_1  \le K \max\{(\nu I_0)^{-2\alpha}, (\nu I_0)^{-1+\alpha} \log (\nu  I_0)\} e^{-|k| \sigma}, \qquad k\neq 0.
\]
Consequently, in  $\D^+_{\kappa,\alpha} \times \T_{\sigma'}$ with 
$0<\sigma'< \sigma$, 
\[
\|S_{2,0}\|_1  \le  K \max\{(\nu I_0)^{-2\alpha}, (\nu I_0)^{-1+\alpha} \log (\nu I_0)\}.
\]
\end{proposition}

\begin{proof}
We have that
\begin{multline*}
(S^+-T_0-\LLi-T_2^+)(v,\theta)
 = (\nu I_0)^{-1} \Phi_0(i+(\nu I_0)^{-1}v) - T_0(v) \\ + (\nu I_0)^{-1} \LLo(i+(\nu I_0)^{-1}v,\theta)-\LLi(v,\theta)+
(\nu I_0)^{-1} \Phi_2^+(i+(\nu I_0)^{-1}v,\theta) - T_2^+(v,\theta),
\end{multline*}
where $\Phi_0$ was introduced in~\eqref{def:Phi0}, $T_0$ in~\eqref{def:T0},  $\LLo$  in
\eqref{def:Loutmes},  $\LLi$ in~\eqref{def:melnikovinner}, $\Phi_2^+$  in Proposition~\ref{prop:aproximacioperlahomoclinicaversio2} and $T_2^+$ in Proposition~\ref{prop:solucions_inner}.

First, an explicit computation shows that, in the domain under consideration,
\[
(\nu I_0)^{-1} \Phi_0(i+(\nu I_0)^{-1}v) - T_0(v) = \frac{i}{4} (\nu I_0)^{-1}\log((\nu I_0)^{-1} v) + (\nu I_0)^{-1}\OO((\nu I_0)^{-1} v).
\]

Next, also an explicit computation shows that
\[
\begin{aligned}
(\nu I_0)^{-1} (\LLo)^{[k]}(i+(\nu I_0)^{-1}v)-(\LLi)^{[k]}(v) & = - V^{[k]}\int_{-\infty}^v \frac{4 (\nu I_0)^{-1} i+(\nu I_0)^{-2} s}{8s(2-(\nu I_0)^{-1}is)^2} e^{ik (s-v)}\,ds \\
& = \OO\left(\frac{(\nu I_0)^{-1}}{v}\right)(1+ \OO((\nu I_0)^{-1}v)) V^{[k]}.
\end{aligned}
\]

Moreover, since $\Phi_2^+ \in \wt \X_{5,3}$ with $\|\Phi_2^+\|_{5,3} \le K(\nu I_0)^{-2}$,
\[
|(\nu I_0)^{-1} (\Phi_2^+)^{[k]}(i+(\nu I_0)^{-1}v) | \le K|v|^{-3} e^{-|k| \sigma}.
\]

Finally, since, by Proposition~\ref{prop:solucions_inner},  $T_2^+ \in \wt \Y_{3}$ with $\|T_2^+\|_{3} \le K$,
\[
|(T_2^+)^{[k]}(v)| \le  K |v|^{-3} e^{-|k| \sigma}.
\]
Since $|v_0^\pm | = \OO((\nu I_0)^\alpha)$, the bounds of the Fourier coefficients follow. To get the bound in the $\|\cdot\|_1$ norm we have to restrict the domain to  $\D^+_{\kappa,\alpha} \times \T_{\sigma'}$ with 
$0<\sigma'< \sigma < \sigma_0$, where $\sigma_0$ was introduced in~\eqref{fitadelsVk}.
\end{proof}

We consider the fixed point equation
\begin{equation}
\label{eq:fixedpointforthemanifoldintheinnerdomain}
S_2 = \wh \G^u_s \circ \Fii(S_2),
\end{equation}
where $\wh \G^u_s$ and $\Fii$ were defined in~\eqref{def:Gtilde} and~\eqref{def:LiFi}, respectively.
\begin{proposition}
\label{prop:éxtensioaldominiinner}
Let $\alpha \in (0,1)$. Then, equation~\eqref{eq:fixedpointforthemanifoldintheinnerdomain} has a solution $S_2^+\in \wt \ZZ_1$ in $\D^+_{\kappa,\alpha}\times \T_{\sigma'} $  
with $\|S_2^+\|_1 \le K(\nu I_0)^{-2\alpha} + K(\nu I_0)^{-1+\alpha} \log (\nu I_0)$. 
\end{proposition}

\begin{proof}
We first check that there exists ${K^*}>0$ such that $\B_{K^*} \subset \wt  \ZZ_1$ satisfies $\wh \G^u_s \circ \Fii (\B_{K^*}) \subset \B_{K^*}$.
To do so, given $S_2 \in \B_{K^*}$, we write
\[
\Fii(S_2) = F_0+ F_1+F_2+F_3+F_4,
\]
with
\[
\begin{aligned}
F_0 & = \wt \E, \\
F_1 & = \nu \partial_\theta T^+ \partial_\theta S_2, \\
F_2 & = \frac{1}{2} \nu (\partial_\theta S_2)^2, \\
F_3 & = (2 A_1 \partial_v T^+ -1) \partial_v S_2, \\
F_4 & = A_1 (\partial_v S_2) ^2,
\end{aligned}
\]
where $\wt \E$, $A_1$ and $A_2$ were introduced in~\eqref{def:wteA1A2}. We claim that $F_i \in \ZZ_2$, $i=0,\dots,4$, and there exists $K>0$ such that 
\begin{align}
\label{bound:F0}
\|F_0\|_2 & \le K(\nu I_0)^{\alpha-1}, \\
\label{bound:F1}
\|F_1\|_2 & \le K\kappa^{-2} {K^*}, \\
\label{bound:F2}
\|F_2\|_2 & \le K\kappa^{-2} ({K^*})^2, \\
\label{bound:F3}
\|F_3\|_2 & \le K\left((\nu I_0)^{\alpha-1} +\kappa^{-1}\right){K^*}, \\
\label{bound:F4}
\|F_4\|_2 & \le K ({K^*})^2.
\end{align}

Bound~\eqref{bound:F0} is simply (1) of Lemma~\ref{lem:boundsforthefixedpointextensioninnerdomain}.

Since $\|\partial_\theta T^+\|_2 \le K$, by (1) of Lemma~\ref{lem:propietatsespaisZr}, we have that
\[
\|\nu \partial_\theta T^+ \partial_\theta S_2\|_2 \le \nu \|\partial_\theta T^+\|_1 \|\partial_\theta S_2\|_1 \le
\frac{\nu}{\kappa^2} \|\partial_\theta T^+\|_2 \|\partial_\theta S_2\|_2
\le K\kappa^{-2} {K^*},
\]
which proves~\eqref{bound:F1}.

Bound~\eqref{bound:F2} follows analogously from
\[
\left\|\frac{1}{2} \nu (\partial_\theta S_2)^2\right\|_2 \le
\frac{\nu}{2}\|\partial_\theta S_2\|_{0} \|\partial_\theta S_2\|_2
\le \frac{\nu}{2} \frac{1}{\kappa^2} ({K^*})^2.
\]

By (3) of Lemma~\ref{lem:boundsforthefixedpointextensioninnerdomain}, bound~\eqref{bound:F3} follows from
\[
\|(2 A_1 \partial_v T^+ -1) \partial_v S_2\|_2 \le
\|2 A_1 \partial_v T^+ -1\|_0 \| \partial_v S_2\|_2 \le K \left((\nu I_0)^{\alpha-1} +\kappa^{-1} \right){K^*}.
\]
Bound~\eqref{bound:F4} follows from (2) of Lemma~\ref{lem:boundsforthefixedpointextensioninnerdomain} and
\[
\|A_1 (\partial_v S_2) ^2\|_2 \le
\|A_1\|_{-2}\| \partial_v S_2\|_2^2 \le K({K^*})^2.
\]
In particular, from~\eqref{bound:F0}, Proposition~\ref{prop:normofV0} and the last claim of Lemma~\ref{lem:operadorGuinnermatching}, since $\Fii(0) = F_0$, we deduce that
\begin{multline*}
\lln \wh \G^u_s \circ \Fii (0 )\rrn_1 = \lln S_{2,0} + \wh{\G}^u \circ \Fii(0)\rrn_1 \le \lln S_{2,0}\rrn_1 + \lln \wh{\G}^u \circ \Fii(0)\rrn_1 \\ \le K \max\{(\nu I_0)^{-2\alpha}, (\nu I_0)^{-1+\alpha} \log (\nu I_0)\} + K \| F_0\|_2
 \le K_1\max \{ (\nu I_0)^{-2\alpha}, (\nu I_0)^{-1+\alpha} \log (\nu I_0)\},
\end{multline*}
for some $K_1>0$.
By the last claim of Lemma~\ref{lem:operadorGuinnermatching}, the same type of computations imply that, if $S_2,S'_2 \in \B_{K^*}$,
\[
\begin{aligned}
\lln \wh \G^u_s \circ \Fii(S_2) - \wh \G^u_s \circ \Fii(S'_2)\rrn_ 1
& \le K\| \Fii(S_2) - \Fii(S'_2)\|_2 \\
& \le K \left( \kappa^{-1}+ (\nu I_0) ^{\alpha-1}+ {K^*} \right)
\lln S_2-S'_2\rrn_1.
\end{aligned}
\]
Then, taking
\[
K^* = 2 \lln \wh \G^u_s \circ \Fii (0 )\rrn_1 = 2K_1\max\{(\nu I_0)^{-2\alpha}, (\nu I_0)^{-1+\alpha} \log (\nu I_0)\},
\]
the claim follows with the usual argument, taking $\kappa$ and $\nu I_0$ large enough so that \\ $ K \left( \kappa^{-1}+ (\nu I_0) ^{\alpha-1} + K^* \right) <1/2$.
\end{proof}

\section{Difference between solutions of the inner equation}
\label{sec:diferencia_solucions_inner}

In this section we compute the difference between the functions $T^\pm = T_0 + \LLipm + T_2^\pm$, with $T^+$ is given in $\D^+_{\kappa,\mathrm{in}}\times \T_\sigma$ 
where  $ \D^+_{\kappa,\mathrm{in}}$ is defined in~\eqref{def:Dominiinner1}
and
$T_0(v)=-1/(4v)$ is introduced in~\eqref{def:T0}, $\LLi$ is defined in~\eqref{def:melnikovinner} and $T_2^+$ is given by Proposition~\ref{prop:solucions_inner}. Moreover, the function $T^-$ is defined by
\begin{equation}
\label{def:T0mLLinmRm}
T^-(v,\theta) = - \overline{T^+(- \bar v,-\bar \theta)} = T_0(v)-\overline{\LLi(-\bar v ,- \bar \theta)}-\overline{T_2^+(-\bar v,-\bar \theta)}
\end{equation}
in the domain $\overline{(- \D^+_{\kappa,\mathrm{in}})}\times \T_\sigma$.

We also define
\begin{equation} \label{def:L-,T2-}
\LLim(v,\theta) = -\overline{\LLi(-\bar v ,- \bar \theta)} \qquad \text{and} \qquad  T_2^-( v, \theta) =
-\overline{T_2^+(-\bar v,-\bar \theta)}.
\end{equation}
The following lemma is an immediate computation. It follows from the fact that $V$ restricted to $\R$ is an even function, i.e. only depends on cosinus. We recall that $\E(T_0)( v, \theta) = -  V(\theta)/8v^2$.
\begin{lemma}
\label{lem:Linnm}
We have that
\[
\LLim(v,\theta)  = \int_{v}^\infty \E(T_0) (s, \theta+s-v) \, ds =
  -\sum_{k\in \Z} \frac{V^{[k]}}{8} e^{ik( \theta- v)} \int_{v}^\infty  \frac{1}{s^2} e^{i k  s} \,ds.
\]
It satisfies the analogous bounds of $\LLi$ in Proposition~\ref{prop:migmelnikovinner}.
\end{lemma}

In view of Lemma~\ref{lem:Linnm}, and taking into account that $V^{[0]} = 0$, for $\Im v<0$ we define the \emph{inner Melnikov potential} as
\[
\LLic (v,\theta)=  \LLi(v,\theta) - \LLim(v,\theta).
\]
Simple computations show that, for $\Im v<0$,   
\begin{equation}
\label{def:LLic}
\LLic (v,\theta)= -\int_{-\infty}^\infty \E(T_0) (s, \theta+s-v) \, ds
= 
\sum_{k\in \Z} \frac{V^{[k]}}{8} e^{ik( \theta- v)} \int_{-\infty}^\infty  \frac{1}{s^2} e^{i k  s} \,ds
 = -\sum_{k \ge 1} \frac{ \pi  k V^{[k]}}{4} e^{ik( \theta- v)},
\end{equation}
where the integral above is computed along a horizontal line with $\Im s < 0$.

Next we consider the 
difference
\begin{equation*}
	\wtDeltainn = T^+-T^-
\end{equation*}
in the domain
\begin{equation}\label{def:Dominissolucioinner}
\D_{\kappa,\mathrm{in}}\times \T_\sigma, \qquad \text{where}\quad  \D_{\kappa,\mathrm{in}}= \D^+_{\kappa,\mathrm{in}} \cap \overline{(- \D^+_{\kappa,\mathrm{in}})} \cap \{\Im v <0\}.
\end{equation}
See Figure~\ref{fig:Dominissolucioinner}.

\begin{figure}[h]
\begin{center}
\includegraphics[width=.6\textwidth]{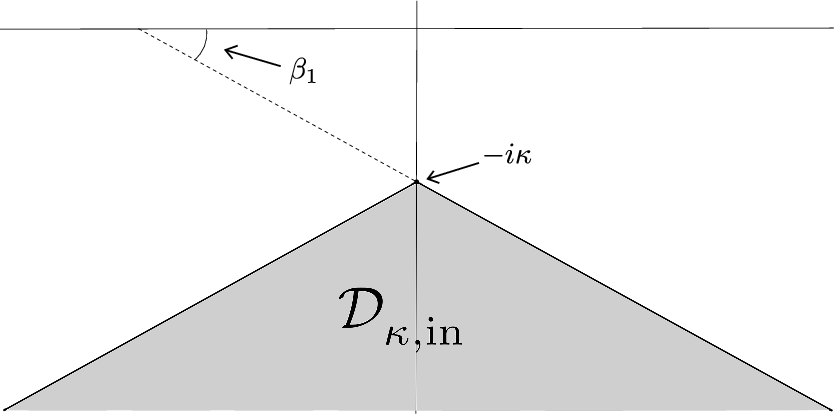}
\end{center}
\caption{The domain  $\D_{\kappa,\mathrm{in}}$ defined in~\eqref{def:Dominissolucioinner}, shaded in gray.}
\label{fig:Dominissolucioinner}
\end{figure}

Both functions $T^\pm$ are solutions of~\eqref{eq:inner}. Then, $\wtDeltainn$ satisfies
\begin{equation}
\label{eq:equationDelta}
\wt \LLin \wtDeltainn  = 0,
\end{equation}
where
\begin{equation*}
\wt \LLin \wtDeltainn =  \frac{1+B}{1+A} \partial_v  \wtDeltainn + \partial_\theta  \wtDeltainn,
\end{equation*}
and
\begin{equation}
\label{def:AiBinner}
\begin{aligned}
A & = \frac{\nu}{2} \left( \partial_\theta T^+ + \partial_\theta T^-\right), \\
B & = 2 v^2 \left( \partial_v T^+ + \partial_v T^-\right) -1.
\end{aligned}
\end{equation}
Note that, since $T^+, T^- \in \wt \Y_3$, if $\kappa $ is big, $1+A\ne 0$.
We look for a change of variables of the form
\[
v = w + W(w,\theta)
\]
such that
$$
\Deltainn(w,\theta) := \wtDeltainn(w+W(w,\theta),\theta)
$$
satisfies
\[
\LLin \Deltainn = 0,
\]
where $\LLin \Deltainn = \partial_v \Deltainn +\partial_\theta \Deltainn $ was introduced in~\eqref{def:LLin}. Notice that, $\wtDeltainn$ satisfies~\eqref{eq:equationDelta} if and only if $\Deltainn$ satisfies
\[
 \frac{1}{1+\partial_w W} \left(\left( \frac{1+B}{1+A}\right)_{\mid (v,\theta) = (w+W(w,\theta),\theta)}- \partial_\theta W \right) \partial_w \Deltainn + \partial_\theta \Deltainn = 0.
\]
Hence, we need to impose that $W$  satisfies
\begin{equation}
\label{eq:redressantLL}
\LLin W = \left(\frac{B-A}{1+A}\right)_{\mid (v,\theta) = (w+W(w,\theta),\theta)}.
\end{equation}
To solve this equation, we define a formal right inverse, $\Gid$,  of the operator $\LLin$ through the Fourier coefficients of the image of a function, $\Gid(W)$:
\begin{align}
\Gid(W)^{[k]}(w) & = \int_{-i \kappa}^w e^{ik(s-w)} W^{[k]}(s)\,ds  \qquad \text{if} \;\; k\ge0, \\
\Gid(W)^{[k]}(w) & = \int_{- i \infty}^w e^{ik(s-w)} W^{[k]}(s)\,ds  \qquad \text{if} \;\; k< 0.
\end{align}
We notice that $\Gid$ is well defined in $\Y_1$.
With the aid of $\Gid$, we
consider the fixed point equation
\begin{equation}
\label{eq:puntfixperaredressarloperador}
W = \Gid \circ \Fd (W),
\end{equation}
where
\[
\Fd(W)(w,\theta) = \left(\frac{B-A}{1+A}\right)_{\mid (v,\theta) = (w+W(w,\theta),\theta)}.
\]
If $W$ is a fixed point of~\eqref{eq:puntfixperaredressarloperador} then it is a solution of  \eqref{eq:redressantLL}.

We consider the spaces of functions $\Y_r$
and $\wt \Y_r$, defined in~\eqref{def:Yr} and~\eqref{def:tildeYR}, with the norms defined in~\eqref{def:normYr} and~\eqref{def:normapotetesinner}
changing $\D^+_{\kappa,\mathrm{in}}$ for $\D_{\kappa,\mathrm{in}}$. Consequently, Lemmas~\ref{lem:propietatsespaisYr} and~\ref{lem:operadorGuinner} hold equally with $\Gid$ instead of~$\Gi^u$.

We define
\begin{equation} \label{eq:perW}
W_0 = \Gid \circ \Fd(0).
\end{equation}

\begin{lemma}
\label{lem:GcircF0redressantloperador}
The functions $A$ and $B$ in~\eqref{def:AiBinner} satisfy $A \in \Y_2$, $B =B_1+B_2 $ with
$B_1\in \Y_1$, $\langle B_1 \rangle =0 $ and $B_2\in \Y_2$ and, moreover,
\[
\|A\|_2 , \|B_1\|_1, \|B_2\|_2 , \left\| \frac{B-A}{1+A}\right\|_1  \le K \nLi (1+\nLi)
\]
for some constant $K >0$, where $\nLi$ was introduced in~\eqref{def:mida_Melnikov_inner}. Furthermore,
$W_0 \in \Y_1$ and 
$$
\|W_0\|_1 \le K \nLi (1+\nLi).
$$
\end{lemma}

\begin{proof}
We recall that, by \eqref{def:mida_Melnikov_inner} and Propositions~\ref{prop:migmelnikovinner} and \ref{prop:solucions_inner}, $\|\partial_\theta \LLi \|_2, \|\partial_v \LLi \|_3 \le  \nLi$, $T_2^+ \in \wt \Y_3$ and $\| \partial_\theta  T_2^+ \|_4 ,\| \partial_v  T_2^+ \|_4\le K\nLi^2$. Then, by the definitions of $\LLim$ and $T_2^-$,
\[
A = \frac{\nu}{2} \left( \partial_\theta T^+ + \partial_\theta T^-\right)
= \frac{\nu}{2} \partial_\theta (\LLi+\LLim) +  \frac{\nu}{2} \partial_\theta \left(T_2^+ + T_2^-\right) \in \Y_2
\]
and $\|A\|_2 \le \nu \nLi (1+K \nLi)$. Moreover,
\[
\begin{aligned}
B & = 2 v^2 \left( \partial_v T^+ + \partial_v T^-\right) -1 \\
& = 2 v^2 \left(2 \frac{1}{4v^2} + \partial_v (\LLi+\LLim)+ \partial_v (T_2^+ +T_2^-) \right) -1 \\
& = 2 v^2 \partial_v (\LLi+\LLim) + 2v^2 \partial_v (T_2^+ +T_2^-)=B_1+B_2 ,
\end{aligned}
\]
where $B_1=2 v^2 \partial_v (\LLi+\LLim)$ and
$B_2= 2v^2 \partial_v (T_2^+ +T_2^-)$. Clearly $\|B_1\|_1 \le 4K \nLi$ and
$\|B_2\|_2 \le 4\nLi^2$.
Then,
$$
\frac{B-A}{1+A} = B_1 + \frac{B_2-A(B_1+1)}{1+A}
$$
and
$\| \frac{B_2-A(B_1+1)}{1+A}\|_2 \le
\|B_2\|_2\|\frac{1}{1+A}\|_0 +
\|A\|_2\|B_1+1\|_0\|\frac{1}{1+A}\|_0 \le K \nLi (1+\nLi)$.

From~(1) and~(2) of Lemma~\ref{lem:operadorGuinner}, $\| \Gid \circ \Fd(0)\|_1 \le K \|B_1\|_1 + K \|\frac{B_2-A(B_1+1)}{1+A}\|_2 \le K \nLi (1+\nLi) $.
\end{proof}

In order to solve equation~\eqref{eq:puntfixperaredressarloperador}, we introduce  $\wt W$ by $W= W_0 + \wt W$, where $W_0 = \Gid \circ \Fd(0)$ is defined in \eqref{eq:perW}.
Then, $W$ is a solution of~\eqref{eq:puntfixperaredressarloperador} if $\wt W$ satisfies
\begin{equation}
\label{eq:puntfixperaredressarloperadorversio2}
\wt W = \Gid \circ \wtFd (\wt W),
\end{equation}
where
\[
\wtFd(\wt W)(w,\theta) = \left( \frac{B-A}{1+A}\right)_{\mid (v,\theta) = (w+W_0(w,\theta) + \wt W(w,\theta),\theta)} - \left(\frac{B-A}{1+A}\right)_{\mid (v,\theta) = (w,\theta)}.
\]

The needed properties of the operator $\wtFd$ are contained in the following technical lemma.
It deals with the control of near identity changes of variables of the form $v=w+W(w,\theta)$ with $W\in \Y_s$.
In the following lemma we introduce the angle parameter $\beta_1$ in the notation of the domains and we write
$\D_{\kappa, \beta_1, \mathrm{in}}$.
\begin{lemma}
\label{lem:composicioambW}
Let $C \in \Y_r$ in $\D_{\kappa_0, \beta_1, \mathrm{in}} \times \T_\sigma$ and  $W\in \Y_s$
in $\D_{\kappa, \tilde  \beta_1,\mathrm{in}} \times \T_\sigma$ with $r, s\ge 1$ and $\beta_1, \tilde \beta_1 \in (0,\pi/2) $.
We define
$$
C_W (w,\theta) = C(w+W(w,\theta),\theta).
$$
Then, there exist $K_0,K_1>0$ and $\kappa_0$ such that, for any $\tilde \beta_1 > \beta_1$ and  $\kappa > \kappa_0$  in the definition of $\D_{\kappa,\tilde \beta_1, \mathrm{in}}$ and the spaces $\Y_s$, the following properties hold.
\begin{enumerate}
\item[(1)] $\partial_v^j C \in \Y_{r+j}$ in $\D_{\kappa, \tilde  \beta_1, \mathrm{in}}\times \T_\sigma$ and $\|\partial_v^j C\|_{r+j} \le \wt K_1 j! \wt K_0^j\|C\|_{r}$ for some $\wt K_0, \wt K_1$ and for all $j\ge 0$.
\item[(2)]  If $W \in \Y_s$ in $\D_{\kappa, \tilde  \beta_1, \mathrm{in}}\times \T_\sigma$, then $C_W \in \Y_r$ in $\D_{\kappa, \tilde  \beta_1, \mathrm{in}}\times \T_\sigma$  and
\[
\|C_W \|_r \le \|C\|_r \frac{K_1}{1-\frac{K_0\|W\|_s}{ \kappa^{s+1}}}.
\]
\item[(3)]  If $W, W'\in \Y_s$ in $\D_{\kappa, \tilde  \beta_1, \mathrm{in}}\times \T_\sigma$, with  $\|W\|_s, \|W'\|_s \le K$, then
$C_W - C_{W'} \in \X_{r+1}$ and
\[
\|C_W - C_{W'}\|_{r+1} \le \frac{K_0 K_1\|C\|_r}{\kappa^{s}} \|W-W'\|_s.
\]
\end{enumerate}
\end{lemma}

\begin{proof}
(1) is an immediate consequence of Cauchy estimates in the reduced domain.

Using (1) of Lemma~\ref{lem:propietatsespaisYr}, (2) follows from
\begin{multline*}
\|C_W(w,\theta) \|_r = \left\| \sum_{j\ge 0} \frac{1}{j!} \partial_v^j C(w,\theta) W^j(w,\theta)\right\|_r
\le \sum_{j\ge 0} \frac{1}{j!} \left\|  \partial_v^j C(w,\theta) W^j(w,\theta)\right\|_r \\
\le
\sum_{j\ge 0} \frac{1}{j!} \left\|  \partial_v^j C(w,\theta)\right\|_r  \left\|W(w,\theta)\right\|^j_0 \le 
 K\wt K_1 \sum_{j\ge 0} \left(\frac{K \wt K_0\|W\|_s}{ \kappa^{s+1}}\right)^j \|C\|_r   = \|C\|_r \frac{K_1}{1-\frac{K_0\|W\|_s}{ \kappa^{s+1}}}.
\end{multline*}

Finally, (3) follows from
\begin{multline*}
C(w+W(w,\theta),\theta)-C(w+W'(w,\theta),\theta) \\
= \int_0^1 \partial_v C(w+W'(w,\theta)+s(W(w,\theta)-W'(w,\theta)),\theta)\, ds \,
(W(w,\theta)-W'(w,\theta))
\end{multline*}
and (1) and (2).
\end{proof}

\begin{proposition}
\label{prop:Wqueredressa}
There exists $\kappa_0$ such that, for any $\kappa > \kappa_0$, equation~\eqref{eq:puntfixperaredressarloperadorversio2} has a unique
solution $\wt W_1 \in  \Y_1$ in $\D_{\kappa, \mathrm{in}}\times \T_\sigma$ with $\|\wt W_1\|_1 \le K \|W_0\|_1^2/\kappa$. As a consequence, $W = W_0 + \wt W_1 \in \Y_1$ is a solution of~\eqref{eq:puntfixperaredressarloperador} satisfying $\|W\|_1 \le K \|W_0\|_1(1+\nLi(1+\nLi))$.
\end{proposition}

\begin{proof}
Since $W_0 \in \Y_1$, by Lemma~\ref{lem:GcircF0redressantloperador} and (3) of Lemma~\ref{lem:composicioambW}, taking 
$C= \frac{B-A}{1+A} $, $W=W_0$, 
$W' = 0$, $r=1$ and $s=1$, we have that
$\wtFd(0) \in \Y_2$ with
\[
\| \wtFd (0) \|_2 \le \frac{K}{\kappa}\left\| \frac{B-A}{1+A}\right\|_1 \|W_0\|_1 \le \frac{K}{\kappa}\nLi(1+\nLi)\|W_0\|_1.
\]
Hence, by (2) of Lemma~\ref{lem:operadorGuinner}, $\|\Gid \circ \wtFd(0)\|_1 \le \|\Gid\| \| \wtFd (0) \|_2$, where
$\|\Gid\|$ stands for the norm of $\Gid: \Y_2 \to \Y_{1}$.
Let $K^*= 2 \| \Gid \| \frac{K}{\kappa} \nLi(1+\nLi) \|W_0\|_1$.
With the same argument, if $W, W' \in \Y_1$ with $\|W\|_1,\|W'\|_1 \le K^*$,
\[
\|\wtFd(W) - \wtFd(W') \|_2 \le \frac{K}{\kappa}\left\| \frac{B-A}{1+A}\right\|_1 \|W-W'\|_1
\le \frac{K}{\kappa}\nLi(1+\nLi) \|W-W'\|_1.
\]
Finally, using again (2) of Lemma~\ref{lem:operadorGuinner},
\begin{multline*}
\| \Gid \circ \wtFd(W) - \Gid \circ \wtFd(W') \|_1 \\
\le \|\Gid\| \|  \wtFd(W) -  \wtFd(W') \|_2 \le  \|\Gid\| \frac{K}{\kappa}\nLi(1+\nLi) \|W-W'\|_1,
\end{multline*}
which proves that $\Gid \circ \wtFd$ is a contraction in a ball of radius $K^*$, if $\kappa_0$ is large enough.
With the usual argument we can show that, if $\kappa_0$ is big enough, $\Gid \circ \wtFd$ sends $\B_{K^*} \subset \Y_1$ into itself.
Then, the fixed point $\wt W_1$ belongs to $\B_{K^*} \subset \Y_1$ and $K^*= \OO ( \nLi(1+\nLi)\|W_0\|_1)$.
Finally, the claim follows from $W= W_0 + \wt W_1$.
\end{proof}

\begin{proposition}
\label{prop:inversa_canvi_inner}
Let $W $ be the map given in Proposition~\ref{prop:Wqueredressa}.
If $\kappa$ is large enough, the map
\[
\Psi(w,\theta) = \begin{pmatrix}
w + W(w,\theta) \\
\theta
\end{pmatrix}
\]
is a well defined change of variables in $\D_{\kappa,\mathrm{in}} \times \T_\sigma$.
Its inverse,
\[
\Psi^{-1}(v,\theta) = \begin{pmatrix}
v + Z(v,\theta) \\
\theta
\end{pmatrix},
\]
is well defined in $\D_{\tilde \kappa,\mathrm{in}} \times \T_\sigma$, for some $\tilde \kappa > \kappa$, where $Z \in \Y_1$
in $\D_{\tilde \kappa,\mathrm{in}} \times \T_\sigma$,
with 
$$\|Z\|_1 \le K \|W_0\|_1(1+\|W_0\|_1).
$$
\end{proposition}

\begin{proof}
The function  $Z$ is the solution of the fixed point equation
\[
Z = \PP(Z),
\]
where
\[
\PP(Z) (v,\theta) = - W(v+Z(v,\theta),\theta).
\]
It is immediate to check that, taking $\tilde \kappa> \kappa$ large enough and
$Z \in \Y_1$, in the domain $\D_{\tilde \kappa, \mathrm{in}}$,  $\PP (Z) \in \Y_1$.
Taking the ball $\B_{K^*}\subset \Y_1$ with $K^*= 2\|W\|_1 = 2K \|W_0\|_1(1+ \nLi(1+\nLi))$,  using Cauchy estimates we get that $\PP $ is Lipschitz with Lipschitz constant
bounded by $\frac{1}{\kappa(\tilde \kappa- \kappa)}$. Hence, it is a contraction if $\kappa(\tilde \kappa- \kappa)$ is large enough.
Moreover, we can check that $\PP (\B_{K^*} )\subset \B_{K^*}$, and hence it has a fixed point $Z\in \B_{K^*}$.
\end{proof}
Now, we define
$$
\Deltainn(w,\theta) = \wtDeltainn (w+W(w,\theta),\theta),
$$
where $\wtDeltainn (v,\theta) = T^+(v,\theta) - T^-(v,\theta)$ and $W$ is the function given by Proposition~\ref{prop:Wqueredressa}.
\begin{theorem}
\label{thm:diferenciadesolucionsdelainner}
There exists $\kappa_0>0$ such that, if $\kappa > \kappa_0$, the following holds.

\begin{enumerate}
\item[(1)]
The function $\Deltainn$ satisfies $\LLin(\Deltainn) = 0$.
\item[(2)]
For any $(v,\theta) \in \D_{\kappa,\mathrm{in}}$,
\[
\wtDeltainn (v,\theta) = \sum_{k\ge 1} f_k e^{i k (\theta - v -  Z(v,\theta))},
\]
for some $f_k \in \C$, with $f_1 \in \R$, where $Z \in \Y_1$  is the function given by Proposition~\ref{prop:inversa_canvi_inner}.
Furthermore,
\begin{equation}
\label{fita:inner_menys_Melnikov_inner}
\left|f_k + \frac{\pi k V^{[k]}}{4} \right|\le \frac{K}{\kappa_0^3} \nLi^2 e^{k \kappa_0}, \qquad k \ge 1,
\end{equation}
where $\nLi$ is the constant introduced in~\eqref{def:mida_Melnikov_inner}.
\end{enumerate}
\end{theorem}

\begin{remark}
Bound~\eqref{fita:inner_menys_Melnikov_inner} can be interpreted as follows.
If one replaces $r_i$ by $\varepsilon \tilde r_i$ in the definition of $V$ in~\eqref{def:potencialdeMorsecorrugat}, then $\nLi = \OO(\varepsilon)$  and, by Proposition \ref{prop:solucions_inner},  $\wtDeltainn - \LLic = \OO(\varepsilon^2)$. That is the difference of solutions of the inner equation, $\wtDeltainn$, is well approximated by the inner Melnikov potential $\LLic$. On the contrary, if $V$ is not small, this bound, which is rather optimal, suggests that then $\LLic$ is not the leading term of $\wtDeltainn$.
\end{remark}

\begin{proof}
By the characterization of the function $W$ in \eqref{eq:puntfixperaredressarloperador} and its properties  given in Proposition~\ref{prop:Wqueredressa} we have that  $\LLin(\Deltainn) = 0$. Moreover,
by the form of $\LLic$ (see~\eqref{def:LLic}), $\LLin(\LLic) = 0$. Hence, $\Deltainn-\LLic \in \ker \LLin$, that is,
$(\Deltainn-\LLic ) (w,\theta) = \tilde f(\theta- w)$, where $\tilde f$ is a $2\pi$-periodic function, and hence
\begin{equation}
\label{eq:DeltaLLicigualaf}
(\Deltainn-\LLic ) (w,\theta) = \sum_{k\in \Z} \tilde f_k e^{ik(\theta- w)}.
\end{equation}

We claim that $\lim_{\Im w \to -\infty} (\Deltainn-\LLic)(w,\theta) = 0$. Indeed, we have
\begin{equation*}
\begin{aligned}
(\Deltainn-\LLic)(w,\theta)  = & \LLi (w+W(w,\theta),\theta)- \LLi (w,\theta)  - \LLim (w+W(w,\theta),\theta)+\LLim (w,\theta) \\
& + T_2^+(w+W(w,\theta),\theta) - T_2^-(w+W(w,\theta),\theta).
\end{aligned}
\end{equation*}
By Propositions~\ref{prop:migmelnikovinner} and~\ref{prop:solucions_inner},
$\LLi \in \Y_2$ and $T_2^+ \in \Y_3$, which implies that $\lim_{\Im w \to -\infty} \LLi(w,\theta) =
\lim_{\Im w \to -\infty} T_2^+(w,\theta) = 0$. Moreover, by
the definitions of $\LLim$ and $T_2^-$ in~\eqref{def:L-,T2-}, the same limits are 0 for $\LLim$ and $T_2^-$.
Proposition~\ref{prop:Wqueredressa} also implies that $W\in \Y_1$ and then  $\lim_{\Im w \to -\infty} W(w,\theta) = 0$. This proves the claim.

Moreover,  $\lim_{\Im w \to -\infty} (\Deltainn-\LLic)(w,\theta) = 0$ implies that $\tilde f_k = 0$ if $k\le 0$. Taking into account the definition of $\LLic$ in~\eqref{def:LLic}, we have that $\Deltainn^{[k]} (w) = 0$, if $k\le 0$
and $\Deltainn^{[k]} (w) = \LLic^{[k]}(w)+\tilde f_k e^{-ik w} $, if $k\ge 1$.

To finish the proof we only need to obtain the bound for $\tilde f_k$ .
We claim that there exists $K_2>0$ such that, if $\kappa_0$ is big enough,
\begin{equation}
\label{bound:DeltamenysLLicnorma0}
\|\Deltainn-\LLic\|_0 \le \frac{K_2}{\kappa_0^3}\nLi^2, \qquad \text{in}\quad \D_{\kappa_0,\mathrm{in}}.
\end{equation}
To prove this claim, we observe that, by Propositions~\ref{prop:migmelnikovinner} and~\ref{prop:Wqueredressa}, $ \LLi \in \Y_3$ and $W\in \Y_1$.
Then, by (1) of Lemma~\ref{lem:propietatsespaisYr}, (3) of Lemma~\ref{lem:composicioambW}  and Lemma~\ref{prop:Wqueredressa},
\[
\begin{aligned}
\|\LLi (w+W(w,\theta),\theta)- \LLi (w,\theta)\|_0 & \le \frac{K}{\kappa_0^3} \|\LLi (w+W(w,\theta),\theta)- \LLi (w,\theta)\|_3 \\
& \le \frac{K\|\LLi\|_2}{\kappa_0^4} \|W\|_1 \le \frac{K \nLi}{\kappa_0^4} \|W_0\|_1 (1+\nLi(1+\nLi)).
\end{aligned}
\]
By the definition of $\LLim$, the same bounds are true for $\LLim (w+W(w,\theta),\theta)- \LLim (w,\theta)$.

Also,
since $T_2^+ \in \Y_3$ and $\|T_2^+\|_3\le K \nLi^2$, by (2) of Lemma~\ref{lem:composicioambW} and (1) of Lemma~\ref{lem:propietatsespaisYr},
\[
\|T_2^+(w+W(w,\theta),\theta)\|_0 \le \frac{K}{\kappa_0^3} \|T_2^+(w+W(w,\theta),\theta)\|_3 \le \frac{K \nLi^2}{\kappa_0^3} \frac{K_1}{1-\frac{K_0\|W\|_1}{ \kappa_0^2}}.
\]
The same bound applies to $T_2^-(w+W(w,\theta),\theta)$. This proves~\eqref{bound:DeltamenysLLicnorma0}.

Now, evaluating the equality
\[
(\Deltainn-\LLic)^{[k]}(w) = \tilde f_k e^{-i k w}
\]
at $w = -i\kappa_0$ and taking into account that the left hand side is bounded by $K_2\nLi^2/\kappa_0^3$, we obtain
that
\[
|\tilde f_k | \le \frac{K_2}{\kappa_0^3} \nLi^2  e^{ k \kappa_0}, \qquad k \ge 1.
\]

Now,
\begin{align*}
\wtDeltainn(v,\theta) & = \Deltainn(v+Z(v,\theta), \theta) =
\sum_{k\ge 1} \left(\LLic^{[k]} (v+ Z(v,\theta ) ) + \tilde f_k e^{-ik (v+ Z(v,\theta ) )} \right)e^{ik\theta} \\
& =
\sum_{k\ge 1} \left(-\frac{\pi k V^{[k]}}{4} + \tilde f_k \right)  e^{ik (\theta-v- Z(v,\theta ) )}
= \sum_{k\ge 1}f_k   e^{ik (\theta-v- Z(v,\theta ) )} ,
\end{align*}
where
$f_k= -\frac{\pi k V^{[k]}}{4} + \tilde f_k $.

The definition of $T^-$ in~\eqref{def:T0mLLinmRm} implies that $\tilde f_1 \in \R$, since $\overline{ \wtDeltainn (-\overline{v},-\overline{\theta})} = \wtDeltainn(v,\theta)$.
Indeed, 
\begin{align*}
0 = & [\overline{\wtDeltainn (-\overline{v},-\overline{\theta})} 
- \wtDeltainn ({v},{\theta})] e^{-i(\theta-v)} \\
= & [\overline{f_1} \overline{e^{i(-\overline \theta+\overline v
+ Z(-\overline{v},-\overline{\theta})) }	}
-f_1 e^{i(\theta - v - Z(v,\theta)) }]  e^{-i(\theta-v)}
+ \sum_{k\ge 2}\cdots
\\
= & \overline{f_1}e^{i \overline{Z(-\overline{v},-\overline{\theta})}}
-f_1 e^{-i Z(v,\theta) } + \sum_{k\ge 2}\cdots .
\end{align*}	
Taking $\Im v \to -\infty$ we get $\overline{f_1}-f_1=0$.

\end{proof}


\section{Difference of the solutions of the Hamilton-Jacobi equation}
\label{sec:diferenciaentresolucionseqHJ}

We have obtained two solutions $\Phi^{\pm} $ of the Hamilton-Jacobi equation~\eqref{def:HJoriginal},
\begin{equation}
\label{eq:Phimouter}
\Phi^+ (u,\theta) = \Phi_0(u)+ \LLo(u,\theta)+ \Phi_2^+(u,\theta),
\end{equation}
defined in $(\D_{\kappa,\delta}^+ \cup \D^+_{\kappa,\mathrm{ext}} \setminus B^*_{\rho})\times \T_\sigma$,
where $B^*_{\rho} = \{|u| < \rho\}$, with $0 < \rho < 1/2$ and
where $\Phi_0$ is introduced in~\eqref{def:Phi0}, $\LLo$ in \eqref{def:Loutmes} and $\Phi_2$ is given in Proposition~\ref{prop:aproximacioperlahomoclinicaversio2},
 and
\begin{equation}
\label{def:Phimenysouter}
\Phi^-(u,\theta) = - \Phi^+ (-u,-\theta) = \Phi_0(u)+ \LLom(u,\theta)+ \Phi_2^-(u,\theta),
\end{equation}
with
\begin{equation*}
\begin{aligned}
\LLom(u,\theta) & = - \LLo(-u,-\theta), \\
\Phi_2^-(u,\theta) & = -\Phi_2^+(-u,-\theta),
\end{aligned}
\end{equation*}
defined for $u \in \D_{\kappa,\delta}^- := - \D_{\kappa,\delta}^+$. In particular, both $\Phi^+$ and $\Phi^-$ are defined in the  domain
\begin{equation}
\label{def:dominiDkappa}
\D_{\kappa} = (\D_{\kappa,\delta}^+ \cup \D^+_{\kappa,\mathrm{ext}} \setminus B^*_{\rho}) \cap \D_{\kappa,\delta}^-.
\end{equation}
See Figure~\ref{fig:Domimicomu}.

\begin{figure}[h]
\begin{center}
\includegraphics[width=.5\textwidth]{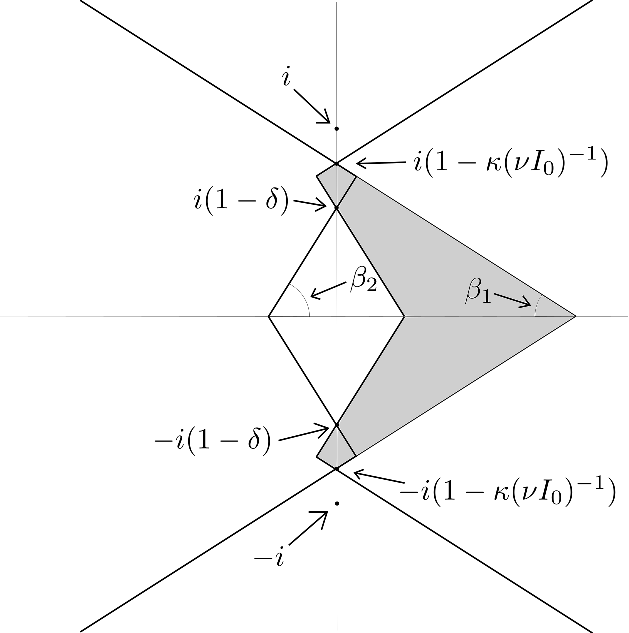}
\end{center}
\caption{The domain  $\D_{\kappa}$ defined in
\eqref{def:dominiDkappa}, shaded in gray. Compare with Figure~\ref{fig:Dominibumerang}.}\label{fig:Domimicomu}
\end{figure}

We recall from \eqref{def:Loutmes} that
\[
\LLo(u,\theta) = -\int_{-\infty}^u \frac{1}{2(1+ s^2)^2} V(\theta- \nu I_0 (s-u)) \, ds,
\]
and, from the fact that $V$ is even, that
\begin{equation}
\LLom(u,\theta) = \int^{\infty}_u \frac{1}{2(1+ s^2)^2} V(\theta- \nu I_0 (s-u)) \, ds.
\end{equation}
The purpose of this section is to obtain an expression for the difference between $\Phi^+$ and $\Phi^-$,
\begin{equation}
	\label{def:wtDeltasencera}
	\wtDeltaout(u,\theta) = \Phi^-(u,\theta)-\Phi^+(u,\theta),
\end{equation}
in their domain $\D_\kappa$.

\subsection{Straightening the linear operator}

Since $\Phi^{\pm}$ are solutions of~\eqref{def:HJoriginal}, $\wtDeltaout$ satisfies
\begin{equation}
\label{eq:c9-equaciowtDeltasencera}
\wt \LL \wtDeltaout  = 0,
\end{equation}
where
\begin{equation*}
\wt \LL \wt \Delta = \frac{1+B}{1+A} \partial_u  \wt \Delta + \nu I_0 \partial_\theta  \wt \Delta
\end{equation*}
and
\begin{equation}
\label{def:AiBouter}
\begin{aligned}
A & = \frac{\nu}{2 \nu I_0} \left( \partial_\theta \Phi^+ + \partial_\theta \Phi^-\right), \\
B & =  \frac{1}{2}\frac{(1+u^2)^2}{u^2} \left( \partial_u \Phi^+ + \partial_u \Phi^-\right) -1.
\end{aligned}
\end{equation}

We look for a change of variables of the form $u = v + X(v,\theta)$ such that
\[
\Deltaout(v,\theta) = \wtDeltaout (v+X(v,\theta),\theta)
\] satisfies
\[
\LL \Deltaout = 0,
\]
where $\LL \Delta =  \partial_u \Delta + \nu I_0 \partial_\theta \Delta$ was introduced in~\eqref{def:diffopL}. Notice that, provided $v\mapsto v + X(v,\theta)$ is invertible, $\wtDeltaout$ satisfies~\eqref{eq:c9-equaciowtDeltasencera} if and only if $\Delta$ satisfies
\[
 \frac{1}{1+\partial_v X} \left(\left(\frac{1+B}{1+A}\right)_{\mid (u,\theta) = (v+X(v,\theta),\theta)} - \nu I_0 \partial_\theta X \right) \partial_v \Delta + \nu I_0 \partial_\theta \Delta= 0.
\]
Hence, we need to impose that $X$ satisfies
\begin{equation}
\label{eq:redressantLLouter}
\LL X = \left(\frac{B-A}{1+A}\right)_{\mid (u,\theta) = (v+X(v,\theta),\theta)}.
\end{equation}

To solve this equation, we consider the inverse $\God$ of the operator $\LL$ defined by the Fourier coefficients of
the image $\God(X)$:
\begin{equation}
\label{def:Gouterdiferencia}
\begin{aligned}
\God(X)^{[k]}(u) & = \int_{u_0^+}^u e^{ik\nu I_0(s-u)} X^{[k]}(s) \,ds, \qquad & k >0, \\
\God(X)^{[0]}(u) & = \int_{u_0}^u  X^{[0]}(s) \,ds, \qquad & \\
\God(X)^{[k]}(u) & = \int_{u_0^-}^u e^{ik\nu I_0(s-u)} X^{[k]}(s) \,ds, \qquad & k <0,
\end{aligned}
\end{equation}
where $u_0^+$, $u_0^- = \overline{u_0^+}$ and $u_0\in \R$ are topmost, bottommost and real leftmost points in $\D_\kappa$, respectively. See Figure~\ref{fig:Domimicomu}.

\begin{remark} \label{rem-sect9:realanalyticity}
The choice of the points $u_0^{\pm}$ and $u_0$ implies that, if $X(u,\theta)$ is real analytic, so is $\God(X)$.
\end{remark}
To solve equation~\eqref{eq:redressantLLouter}, for $r\in \R$, we introduce  the Banach space of analytic, $2\pi$-periodic in~$\theta$ functions
\begin{equation*}
\X_{s} = \{X:\D_{\kappa} \times \T_\sigma \to \C \mid  \, \|X\|_{s} < \infty\},
\end{equation*}
with the norm
\[
\|X\|_{s} = \sum_{k\in \Z} \|X^{[k]}\|_{s} e^{|k|\sigma},
\]
where, for an analytic function $f:\D_\kappa \to \C$,
\[
\| f\|_{s} = \sup_{u \in \D_\kappa} |(1+u^2)^{s} f(u)|.
\]

Next, we state two technical lemmas completely analogous to Lemmas \ref{lem:propietatsespaisXrs} and \ref{lem:operadorGu}.

\begin{lemma}
\label{lem:propietatsespaisXr}
Let $s_1, s_2 \in \R$.
\begin{enumerate}
\item[(1)]
If $X \in \X_{s_1+ s_2}$ and $s_2 \ge 0$, then $X \in \X_{s_1}$ and there exists $K>0$, independent of $s_1, s_2$ such that 
\[
\|X\|_{s_1} \le K \left(\frac{\nu I_0}{\kappa}\right)^{s_2}\|X\|_{s_1+s_2}.
\]
\item[(2)]
If $X_1 \in \X_{s_1}$ and $X_2 \in \X_{s_2}$, then $X_1X_2 \in \X_{s_1+s_2}$ and
\[
\|X_1 X_2 \|_{s_1+s_2} \le \|X_1  \|_{s_1}\|X_2 \|_{s_2}.
\]
\end{enumerate}
\end{lemma}

\begin{lemma}
\label{lem:operadorGenXr}
Let $s\in \R$ and $\God$ be the operator defined by~\eqref{def:Gouterdiferencia}.
\begin{enumerate}
\item[(1)]
If $X \in  \X_{s}$ with  $s\ge 0$, then $\God(X) \in \X_s $ and
\[
\|\God(X)\|_{s} \le K \|X\|_{s}.
\]
If, furthermore, $\langle X\rangle = 0$,
\[
\|\God(X)\|_{s} \le K (\nu I_0)^{-1}\|X\|_{s}.
\]
\item[(2)]
If $X \in \X_{s}$ with $s> 1$, then $\God(X)\in \X_{s-1}$ and
\[
\|\God(X)\|_{s-1} \le K \|X\|_{s}.
\]
\item[(3)]
If $X \in \X_{s}$ with $s>0$, then $\partial_u \God(X), \partial_\theta \God(X) \in \X_{s}$ and
\[
\|\partial_u \God(X)\|_{s} \le K \|X\|_{s}, \qquad
\|\partial_\theta \God(X)\|_{s} \le K (\nu I_0)^{-1} \|X\|_{s}.
\]
\end{enumerate}
\end{lemma}

Now, we summarize the properties of $A$ and $B$ we need.

\begin{lemma}
\label{lem:AiBouter} Let $A$ and $B$ be the functions introduced in~\eqref{def:AiBouter}. There exists $K>0$ such that
\begin{enumerate}
\item[(1)]
$A \in \X_2$ with $\|A\|_2 \le K (\nu I_0)^{-2}$.
\item[(2)]
$B = B_1+B_2$ where $B_1 \in \X_1$, $\langle B_1 \rangle =0$, $\|B_1 \|_1 \le K(\nu I_0)^{-1}$ and $B_2 \in \X_2$ with
$\|B_2\|_2 \le K(\nu I_0)^{-2}$. Consequently, $B \in \X_1$ with $\|B\|_1 \le K(\nu I_0)^{-1}$.
\end{enumerate}
\end{lemma}

\begin{proof}
To prove (1), we use the definitions in~\eqref{eq:Phimouter} and~\eqref{def:AiBouter}. From  Proposition~\ref{prop:primeraiteracio} we have
$\|\partial_\theta \LLo\|_2 \le K(\nu I_0)^{-1}$ and,  from Proposition~\ref{prop:aproximacioperlahomoclinicaversio2} and (1) of Lemma~\ref{lem:propietatsespaisXr},
\[
\|\partial_\theta \Phi_2^+\|_2 \le \left(\frac{\nu I_0}{\kappa}\right)^2 \|\partial_\theta \Phi_2^+\|_4 \le
\left(\frac{\nu I_0}{\kappa}\right)^2  (\nu I_0)^{-1}
\lln  \Phi_2^+\rrn_3 \le K \kappa^{-2} (\nu I_0)^{-1}.
\]

Since the bounds of $\LLom$ and $\Phi_2^-$ are the same as the ones of
 $\LLo$ and $\Phi_2^+$, respectively, and
$\partial_\theta \Phi_0 = 0$,
we have
\[
\|A\|_2 = \frac{\nu}{2\nu I_0}\left\| \partial_\theta \Phi^+ + \partial_\theta \Phi^-\right\|_2
\le \frac{\nu}{\nu I_0} \left( \| \partial_\theta \LLo\|_2 +
\| \partial_\theta \Phi_2^+ \|_2\right) \le K(\nu I_0)^{-2}.
\]

Now we prove (2). Using the definition of $\Phi_0$ in~\eqref{def:Phi0} we can check that $B = B_1+B_2$, where
\[
\begin{aligned}
B_1(u,\theta) & = \frac{1}{2}\frac{(1+u^2)^2}{u^2} (\partial_u \LLo(u,\theta)+\partial_u\LLom(u,\theta)), \\
B_2(u,\theta) & = \frac{1}{2}\frac{(1+u^2)^2}{u^2} (\partial_u\Phi_2^+(u,\theta)+\partial_u\Phi_2^-(u,\theta)).
\end{aligned}
\]
By Proposition~\ref{prop:primeraiteracio},
$\|\partial_u \LLo\|_3 \le K(\nu I_0)^{-1}$ and since $\|(1+u^2)^2/u^2\|_{-2} = \OO(1)$, by
(2) of Lemma~\ref{lem:propietatsespaisXr},
\[
\|B_1\|_1 \le \left\| \frac{1}{2}\frac{(1+u^2)^2}{u^2}\right\|_{-2}\|\partial_u \LLo+\partial_u \LLom\|_ 3 \le K(\nu I_0)^{-1}.
\]
Since $\langle \LLo\rangle =0$ and $(1+u^2)^2/u^2$ does not depend on $\theta$, $\langle B_1 \rangle =0$.

Finally, by
Proposition~\ref{prop:aproximacioperlahomoclinicaversio2},
$ \|\partial_u \Phi_2^+\|_4 \le \lln \Phi_2^+\rrn_3 \le K(\nu I_0)^{-2}$.
Then,
\[
\|B_2\|_2 \le \left\| \frac{1}{2}\frac{(1+u^2)^2}{u^2}\right\|_{-2}\|\partial_u \Phi_2^+ + \partial_u \Phi_2^-\|_4
 \le K(\nu I_0)^{-2}.
\]
\end{proof}

Using the operator~$\God$, we consider the fixed point equation
\begin{equation}
\label{eq:redressantLLouterpuntfix}
X = \God \circ \Fod (X),
\end{equation}
where
\begin{equation}
\label{def:redressantloperadorouter}
\Fod(X ) = \left(\frac{B-A}{1+A}\right)_{\mid (u,\theta) = (v+X(v,\theta),\theta)}.
\end{equation}
It is clear that if $X$ satisfies \eqref{eq:redressantLLouterpuntfix}, then $X$ is a solution of~\eqref{eq:redressantLLouter}.
Let $$X_0 = \God \circ \Fod(0).
$$
\begin{proposition}
\label{lem:fitaX}
$X_0 \in \X_1$ and $\|X_0\|_1 \le K(\nu I_0)^{-2}$. Furthermore, $X_0$ is real analytic.
\end{proposition}

\begin{proof}
We have that $\Fod(0) = \frac{B-A}{1+A} = B_1 + \frac{B_2-A(B_1+1)}{1+A}$.

By Lemma~\ref{lem:AiBouter}, $\|B_1\|_1 \le K(\nu I_0)^{-1}$ and $\langle B_1 \rangle = 0$. Then, by (1) of Lemma~\ref{lem:operadorGenXr}, $\|\God(B_1)\|_1 \le K(\nu I_0)^{-1} \|B_1\|_1 \le K(\nu I_0)^{-2}$.

Also, by Lemma~\ref{lem:AiBouter}, $\|B_2\|_2 \le K(\nu I_0)^{-2}$. Then, the claim follows from the bound
\[
\left\|\frac{B_2-A(B_1+1)}{1+A}\right \|_2 \le \|B_2\|_2\|(1+A)^{-1}\|_0  +\|A\|_2 \|B_1+1\|_0 \|(1+A)^{-1}\|_0
\]
and (2) of Lemma~\ref{lem:operadorGenXr}.
\end{proof}

The following lemma is analogous to Lemma~\ref{lem:composicioambW}.
The only difference is the geometry of the domain. Actually $\D_{\kappa}$ depends on $\kappa,\, \delta,\,\beta_1$ and $\beta_2$.
Having fixed some $\kappa_0 $, $\delta_0 $, $\beta_1$ and $\beta_2$ with  $\beta_2> \beta_1$
we will write $\wh \D_\kappa$ to denote a family of domains with
$\kappa > \kappa_0$ and $\delta < \delta_0$ and with the same angles
$\beta_1,\, \beta_2$ such that if $\tilde \kappa> \kappa > \kappa_0$ the distance from $\wh \D_{\tilde \kappa}$ to the boundary of
$\wh \D_\kappa$ is $(\tilde{\kappa}- \kappa)(\nu I_0)^{-1}  \cos \beta_1$. This implies determining $\tilde \delta $, depending on $\tilde \kappa$ so that
$(\delta -\tilde \delta)\cos \beta_2 =
(\tilde{\kappa}- \kappa)(\nu I_0)^{-1}  \cos \beta_1$.
Note that a change in $\tilde \kappa$ of order one produces  a
change in $\tilde \delta$ of order $(\nu I_0)^{-1}$.
\begin{lemma}
\label{lem:composicioambXouter}
Let $C \in \X_{s}$ in $\wh D_{\kappa}\times \T_\sigma $ and $X\in \X_t$ in $\wh \D_{\tilde \kappa}\times \T_\sigma $ with $s, t \ge 1$ and
$\tilde{\kappa}> \kappa $.
We define
\[
C_X(v,\theta) = C(v+X(v,\theta),\theta), \qquad (v,\theta)\in
\wh \D_{\tilde \kappa}\times \T_\sigma .
\]
Then, there exists $\kappa_0$ such that for any $\tilde \kappa>\kappa > \kappa_0$ there exist $K_0, K_1>0$, depending on $\tilde \kappa, \kappa$ but independent on $ \nu I_0 $, such that the following holds.
\begin{enumerate}
\item[(1)] $\partial_v^j C \in \X_{s+j}$ in
$\wh D_{\tilde \kappa}\times \T_\sigma $ and $\|\partial_v^j C\|_{s+j} \le K_1 j! K_0^j\|C\|_{s}$, $	j\ge  0$.
\item[(2)] If $X \in \X_t $ in
$\wh D_{\tilde \kappa}\times \T_\sigma $, then $C_X \in X_{s}$
in
$\wh D_{\tilde \kappa}\times \T_\sigma $
and
\[
\|C_X \|_s \le \|C\|_s \frac{K}{1- K_0 (\nu I_0 \tilde \kappa^{-1})^t \|X\|_t}.
\]
\item[(3)] If $X, X'\in \X_t$ in
$\wh D_{\tilde \kappa}\times \T_\sigma $, with  $\|X\|_t, \|X'\|_t \le K$,
\[
\|C_X - C_{X'}\|_{s+1} \le K
 (\nu I_0 \tilde \kappa^{-1})^t
 \|C\|_s \|X-X'\|_t.
\]
\end{enumerate}
\end{lemma}

\begin{proof}
(1) is an immediate consequence of Cauchy estimates in the reduced domain.

Using (1) of Lemma~\ref{lem:propietatsespaisXr}, (2) follows from
\[
\begin{aligned}
\|C_X(v,\theta) \|_s & = \left\| \sum_{j\ge 0} \frac{1}{j!} \partial_v^j C(v,\theta) X^j(v,\theta)\right\|_s
\le \sum_{j\ge 0} \frac{1}{j!} \|  \partial_v^j C(v,\theta)\|_s \|X\|_0^j \\
& \le \sum_{j\ge 0} \frac{K}{j!}  \left(\frac{\nu I_0}{\tilde \kappa}\right)^{tj}  \|\partial_v^j C(v,\theta)\|_{s+j} \|X\|_t^j
 \le K K_1 \sum_{j\ge 0} \left(\left(\frac{\nu I_0}{\tilde \kappa}\right)^{t} K_0 \|X\|_t\right)^j \|C\|_s  \\
&  = \|C\|_s \frac{K K_1}{1- \left(\frac{\nu I_0}{\tilde \kappa}\right)^{t}K_0 \|X\|_t}.
 \end{aligned}
\]

(3) follows from
\begin{multline*}
C(v+X(v,\theta),\theta)-C(v+X'(v,\theta),\theta) \\
= \int_0^1 \partial_v C(v+X'(v,\theta)+s(X(v,\theta)-X'(v,\theta)),\theta)\, ds \,(X-X'),
\end{multline*}
(1) and (2) and the fact that $\|X\|_0 \le K (\nu I_0/\tilde \kappa)^t\|X\|_t$.
\end{proof}

To solve equation~\eqref{eq:redressantLLouterpuntfix}, we introduce $\wt X$ by setting $X= X_0+\wt X$. Then, $X$ is a solution of~\eqref{eq:redressantLLouterpuntfix} if and only if $\wt X$ satisfies
\begin{equation}
\label{eq:redressantLLouterpuntfixdef}
\wt X = \God \circ \wtFod (\wt X),
\end{equation}
where
\begin{equation*}
\wtFod(\wt X ) = \Fod(X_0+\wt X)- \Fod(0).
\end{equation*}

\begin{proposition}
\label{prop:wtXqueredressa}
There exists $\kappa_0$ such that, for any $\kappa > \kappa_0$, equation~\eqref{eq:redressantLLouterpuntfixdef} has a unique
solution $\wt X_1 \in  \X_1$ with $\|\wt X_1\|_1 \le K(\nu I_0)^{-2}/\kappa^2$. As a consequence, $X = X_0 + \wt X_1 \in \X_1$ is a solution of~\eqref{eq:redressantLLouterpuntfix},
$\| X\|_1 \le K(\nu I_0)^{-2}$
 and it is real analytic.
\end{proposition}

\begin{proof}
We first remark that, 
by Lemma \ref{lem:propietatsespaisXr}
$\|A\|_0 \le K(\nu I_0)^2 \|A\|_2/\kappa^2 $, $\|A\|_1 \le K\nu I_0 \|A\|_2/\kappa $, 
and taking into account Lemma~\ref{lem:AiBouter}, 
$(B-A)/(1+A) \in \X_1$ with $\|(B-A)/(1+A)\|_1 \le \|B-A\|_1\|(1+A)^{-1}\|_0 \le K(\nu I_0)^{-1}$.

Since $X_0 \in \X_1$ in $\wh \D_{\kappa}$ with $\|X_0\|_1 \le K(\nu I_0)^{-2}$, by (3) of Lemma~\ref{lem:composicioambXouter}, taking $X' = 0$, $s=1$, $t=1$
and $\tilde \kappa>\kappa$
we have that
$\wtFod(0) \in \X_2$  in $\wh \D_{\tilde \kappa}$  with
\[
\| \wtFod (0) \|_2 \le K\frac{\nu I_0}{\tilde \kappa}\left\| \frac{B-A}{1+A}\right\|_1 \|X_0\|_1 \le K \kappa^{-1}(\nu I_0)^{-2}.
\]
Hence, by (2) of Lemma~\ref{lem:operadorGenXr}, $\|\God \circ \wtFod(0)\|_1 \le \|\God\| \| \wtFod (0) \|_2 \le K(\nu I_0)^{-2}/\tilde \kappa$, where here
$ \|\God\| $
is the norm of the operator $\God : \X_2 \to \X_1$.
Let $K^*= 2\|\God \circ \wtFod(0)\|_1$.
 With the same argument, if $X,X' \in \X_1$ with $\|X\|_1,\|X'\|_1 \le K^*$,
\[
\|\wtFod(X) - \wtFod(X') \|_2 \le K \frac{\nu I_0}{\tilde \kappa}\left\| \frac{B-A}{1+A}\right\|_1 \|X-X'\|_1 \le \frac{K}{\tilde \kappa} \|X-X'\|_1 \qquad \text{in} \quad \wh \D_{\tilde \kappa},
\]
for some $K>0$, independent of $\nu I_0$ and $\tilde \kappa$.
Finally, using again (2) of Lemma~\ref{lem:operadorGenXr},
\begin{multline*}
\| \God \circ \wtFod(X)  - \God \circ \wtFod(X') \|_1  \\ \le \|\God\| \|  \wtFod(X) -  \wtFod(X') \|_2  \le  \frac{\|\God\|K}{\tilde \kappa} \|X-X'\|_1,
\end{multline*}
which proves that $\God \circ \wtFod$ is a contraction in $\wh \D_{\tilde \kappa}$ if $\tilde \kappa$ is large enough.
Now we can check that $\God \circ \wtFod$ sends $\B_{K^*}\subset \X_1$ into itself and therefore $\God \circ \wtFod$ has a unique fixed point in that ball. The real analyticity claim follows from the definitions of
$\Fod$ and $\wtFod$ and Remark~\ref{rem-sect9:realanalyticity}.
\end{proof}

\begin{proposition}
\label{prop:inversa_canvi}
Let $X$ be the function, in the domain $\hat \D_{\kappa} \times \T_\sigma$, given by Proposition \ref{prop:wtXqueredressa}. Then,
the map
\[
\Theta(v,\theta) = \begin{pmatrix}
v + X(v,\theta) \\
\theta
\end{pmatrix}
\]
is a well defined change of variables in $\hat D_\kappa \times \T_\sigma$.
Its inverse,
\[
\Theta^{-1}(u,\theta) = \begin{pmatrix}
u + Y(u,\theta) \\
\theta
\end{pmatrix}
= \begin{pmatrix}
u - X(u,\theta)+\wh X(u,\theta) \\
\theta
\end{pmatrix},
\]
is well defined and real analytic in $\hat \D_{\tilde \kappa} \times \T_\sigma$, for some $\tilde \kappa > \kappa$, with
$Y\in \X_1$ and $\|Y\|_1 \le K (\nu I_0)^{-2}$. Moreover
$Y=-X+ \wh X$ with
$\wh X \in \X_1$ and
$\|\wh X\|_1 \le \kappa^{-1}(\nu I_0)^{-3}$ in $\hat \D_{\tilde \kappa}\times \T_\sigma$.
\end{proposition}

\begin{proof}
The function  $Y$ is the solution of the fixed point equation
\[
Y = \PP (Y),
\]
where
\[
\PP (Y) (u,\theta) = - X(u+Y(u,\theta),\theta).
\]
Clearly, $\PP(0) = -X \in \X_1$, with $\|\PP(0)\|_1 = \|X\|_1 \le K(\nu I_0)^{-2}$.
We define $K^* = 2K_1\|X\|_1$.

Then, by (2) of Lemma~\ref{lem:composicioambXouter},
if $Y \in \B_{K^*}\subset \X_1$, $\|\PP(Y)\|_1 \le K(\nu I_0)^{-2}$. Since, by Proposition~\ref{prop:wtXqueredressa}, $X \in   \X_1$,
using (3) of Lemma~\ref{lem:composicioambXouter},
$\PP$ restricted to $\B_{K^*}$ is Lipschitz with Lipschitz constant $K \tilde \kappa^{-1}\nu I_0 \|X \|_1 \le K\tilde \kappa^{-1}(\nu I_0)^{-1}$. Hence, it is a contraction if $\tilde \kappa^{-1}
(\nu I_0)^{-1}$ is small enough.
Now, one easily checks that $\PP (\B_{K^*}) \subset \B_{K^*}$.
Let $ Y_0$ be the unique fixed point of $\PP$ in $\B_{K^*}$. Then, writing $Y_0=-X+\wh X$,
\[
\|\wh X \|_1=
\| Y_0+X\|_1 = \|\PP( Y_0)-\PP(0)\|_1 \le
K \frac{1}{\kappa \nu I_0} \|Y_0\|_1 \le K \frac{1}{\kappa (\nu I_0)^3}.
\]
\end{proof}

\subsection{The exponentially small formula for  $\Phi^+-\Phi^-$}

In this section we finally obtain the formula for the difference $ \wtDeltaout=\Phi^+-\Phi^-$ for real values of $u$ and $\theta$.
First, we introduce
\begin{equation*}
\Upsilon^+(u,\theta)
= \nu I_0 \Deltainn(\nu I_0 (u+Y(u,\theta) -i), \theta),
\end{equation*}
where the function $Y$ is given by Proposition~\ref{prop:inversa_canvi}.
$\Upsilon^+$ is defined in
$\{u\in \C\mid \,  \Im u < 1-\kappa (\nu I_0)^{-1} + \min\{ \tan \beta_1 \Re u , -\tan\beta_1 \Re u \}
\}$
which contains the domain $\D_{\kappa}$.
We also define
\begin{equation}
	\label{def:Upsilonmenys}
	\Upsilon^-(u,\theta) = \overline{\Upsilon^+ (\overline{u},\overline{\theta})}
\end{equation}
and
\begin{equation}
	\label{def:Upsilon}
	\Upsilon = \Upsilon^+ + \Upsilon^-.
\end{equation}
This definition implies that $\Upsilon $
is real analytic.

We recall that, since $\LLin \Deltainn = 0$ (see Theorem~\ref{thm:diferenciadesolucionsdelainner}), we have
\[
\Deltainn (v,\theta) =  \sum_{k \ge 1} f_k e^{ik(\theta-v)},
\]
where the coefficients $f_k$ do not depend on $\nu I_0$.
Then, introducing
\begin{equation}
\label{eq:relaciocoeficientsFourierUpsiloniDeltainn}
\Upsilon_k = \nu I_0 f_k e^{-k \nu I_0}, \qquad k \ge 1,
\end{equation}
we have
\[
\Upsilon^+(u,\theta) = \sum_{k\ge 1} \Upsilon_k e^{ik(\theta-\nu I_0 u-\nu I_0 Y(u,\theta))}
\]
and
\[
	\Upsilon^-(u,\theta) = \sum_{k\le -1} \overline{\Upsilon_{-k}} e^{ik(\theta-\nu I_0 u-\nu I_0 Y(u,\theta))},
\]
since $Y$ is real analytic.

The function $\Upsilon$ will be the desired first order of $\Phi^+-\Phi^-$. In order to prove this fact, we introduce
\begin{equation*}
\E(u,\theta) = \wtDeltaout(u,\theta) - \Upsilon(u,\theta)
\end{equation*}
defined in $\D_\kappa$.

\begin{proposition}
\label{prop:fitadeladiferenciadominicomplex}
Let $N \in \N$ be fixed. There exists $K>0$ and $s>0$ such that, for all $(u,\theta) \in \D_{s \log (\nu I_0)} \times \T_\sigma$ such that $\nu I_0(u-i) \in \D^+_{s \log (\nu I_0),\alpha} \cap \left(- \overline{\D^+_{s \log (\nu I_0),\alpha}}\right)$, $0\le j+ k\le N$,
\[
|\partial_u^j \partial_\theta^k \E(u,\theta)|  \le K (\nu I_0)^{j+1} \frac{\max\{(\nu I_0)^{-2\alpha}, (\nu I_0)^{-1+\alpha} \log (\nu I_0), (\nu I_0)^{-s}\} }{\log (\nu I_0)},
\]
where the domain $\D_\kappa$ was introduced in~\eqref{def:dominiDkappa}
and $\D^+_{\kappa,\alpha}$ in \eqref{def:DominisRaros2}.
\end{proposition}

\begin{proof}
Since $\E$ is real analytic, it is enough to bound it for $\Im u \ge 0$. We write
\begin{equation*}
\E = \E_1 + \E_2 + \E_3,
\end{equation*}
where
\begin{equation*}
\E_1(u,\theta) = \Phi^+(u,\theta)-\Phi^-(u,\theta) - \nu I_0\left( T^+(\nu I_0(u-i),\theta)-T^-(\nu I_0(u-i),\theta)\right),
\end{equation*}
the function $T^+$ was given by Corollary~\ref{cor:solutioninnerequation} and $T^-$ was introduced in~\eqref{def:T0mLLinmRm},
\begin{equation*}
\E_2(u,\theta) =  \nu I_0\left( T^+(\nu I_0(u-i),\theta)-T^-(\nu I_0(u-i),\theta)\right)- \Upsilon^+(u,\theta)
\end{equation*}
and
\begin{equation*}
\E_3(u,\theta) =  - \Upsilon^-(u,\theta).
\end{equation*}
We bound each term separately. We start with $\E_1$. We claim that
\begin{equation}
\label{fita:E1}
|\E_1(u,\theta)| \le  K \nu I_0 \frac{\max\{(\nu I_0)^{-2\alpha}, (\nu I_0)^{-1+\alpha} \log (\nu I_0)\} }{s \log (\nu I_0)}.
\end{equation}
Indeed, for  $(u,\theta) \in \D_{s \log (\nu I_0)} \times \T_\sigma$ such that $\nu I_0(u-i) \in \D^+_{s \log (\nu I_0),\alpha} \cap \left(- \overline{\D^+_{s \log (\nu I_0),\alpha}}\right)$,
by Proposition~\ref{prop:éxtensioaldominiinner},
\begin{equation*}
|\E_1(u,\theta)| \le 2 \nu I_0  \|S_2^+\|_0 \le \frac{K\nu I_0}{s \log (\nu I_0)} \|S_2^+\|_1 \le K \nu I_0 \frac{\max\{(\nu I_0)^{-2\alpha}, (\nu I_0)^{-1+\alpha} \log (\nu I_0)\} }{s \log (\nu I_0)}.
\end{equation*}

To bound $\E_2$ we observe that, by Theorem~\ref{thm:diferenciadesolucionsdelainner} and Propositions~\ref{lem:fitaX} and~\ref{prop:inversa_canvi},
\[
\E_2(u,\theta) = \nu I_0 \Big(  \Deltainn (\nu I_0 (u-i)+Z(\nu I_0 (u-i),\theta),\theta) - \Deltainn (\nu I_0(u-i+Y(u,\theta)),\theta)\Big).
\]
Also, taking into account that $\|Z\|_1 \le K$ and $\|Y\|_1 = K(\nu I_0)^{-2}$, we have that
\[
\begin{aligned}
|Z(\nu I_0 (u-i),\theta)| & \le \frac{K}{\nu I_0 |u-i|}, \\
|\nu I_0 Y(u,\theta)| & \le \frac{K}{\nu I_0 |u-i|}.
\end{aligned}
\]
Hence, by the mean value theorem, for  $(u,\theta) \in \D_{s \log (\nu I_0)} \times \T_\sigma$ we have that
\[
|\E_2(u,\theta)| \le K \frac{\nu I_0}{s (\nu I_0)^s \log (\nu I_0)^4}.
\]

Finally, to bound $\E_3$ we  first observe that, for  $(u,\theta) \in \D_{s \log (\nu I_0)} \times \T_\sigma $ with $\Im u \ge0$,
since $Y$ is real analytic, which implies that $\Im Y (\Re u,\Re \theta) = 0$,  and $\|Y\|_1 = \OO((\nu I_0)^{-2})$,
\[
\begin{aligned}
\Im (u +Y(u,\theta))  &\ge (1-\|\partial_u Y\|_0) \Im u   
  \ge (1-K(\nu I_0)^2/(s\log (\nu I_0))^2\|Y\|_1)\Im u \\ & \ge (1-K/(s\log (\nu I_0))^2)\Im u \ge 0,
\end{aligned}
\]
where we have used Cauchy estimates to relate $\partial _u Y$ and $Y$, and we have slightly reduced the domain by considering a bigger value of $s$.
Hence, from the definition of $\Upsilon^-$ in~\eqref{def:Upsilonmenys} and~\eqref{eq:relaciocoeficientsFourierUpsiloniDeltainn}, for  $(u,\theta) \in \D_{s \log (\nu I_0)} \times \T_\sigma$ with $\Im u >0$,
\begin{align*}
|\E_3(u,\theta)| & \le \sum_{k\le -1} \nu I_0 | f_{-k} |e^{k (\nu I_0-\sigma)} e^{k(\nu I_0 \Im u(1+\OO((\nu I_0)^{-1})))} \\
& \le \sum_{k\le -1} \nu I_0 | f_{-k} |e^{k (\nu I_0-\sigma)} 
\le K \nu I_0 e^{-(\nu I_0-\sigma)}.
\end{align*}
The bounds of the statement follow applying standard Cauchy estimates and slightly reducing the domain.

\end{proof}

Finally, we need the following elementary lemma.

\begin{lemma} \label{lem:cotesexp}
Let $\Psi: \D_{s\log (\nu I_0)}\times \T_\sigma \to \C$ be an analytic function.
We write $\Psi(u,\theta)= \sum_{k\in Z} \Psi^{[k]}(u) e^{ik\theta}$.
Assume
\begin{enumerate}
\item [(i)] $\LL \Psi = 0$.
\item[(ii)] There exists $M> 0$ such that
$
|\Psi^{[k]}(\pm (i-i s \log (\nu I_0)/(\nu I_0)))| \le M.
$
\end{enumerate}
Then, $\Psi(u,\theta) = \sum_{k\in \Z} \Lambda_k e^{ik(\theta-\nu I_0u)}$ and
\[
|\Lambda_{\pm k}|  \le M (\nu I_0)^{k s} e^{-k \nu I_0}, \qquad k \ge 1.
\]
\end{lemma}

\begin{proof}
Since $\LL \Psi = 0$, there exists a $2\pi$-periodic function $\Lambda(\phi) = \sum_{k \in \Z} \Lambda_k e^{ik\phi}$ such that $\Psi(u,\theta) = \Lambda(\theta-\nu I_0u)$. Since $\Psi$ is also periodic with respect to $\theta$, we have that $\Psi(u,\theta) = \sum_{k\in \Z} \Psi^{[k]}(u) e^{ik\theta}$, that is
\[
\Psi^{[k]}(u) = \Lambda_k e^{-ik\nu I_0 u}.
\]
The claim follows evaluating the above equality at $u = i-i s \log (\nu I_0)/(\nu I_0)$, for $k \ge 1$, and at $u = -i+i s \log (\nu I_0)/(\nu I_0)$,
for $k \le -1$.
\end{proof}

\begin{theorem}
\label{thm:formulaexponencialmentpetitadiferenciatotal}
Fix $\alpha$ such that $\min\{2\alpha,1-\alpha\} >0$.
Take  $0 < s < \min\{2\alpha,1-\alpha\}$ and fix $N \in \N$ as in Proposition~\ref{prop:fitadeladiferenciadominicomplex}. There exists $\Lambda_0 \in \R$ such that, for all $(u,\theta) \in (\D_{s \log (\nu I_0)} \cap \R) \times \T_\sigma$, $0\le j+ k\le N$.
Let
\[
\begin{aligned}
\partial_u^j \partial_\theta^k (\wtDeltaout (u,\theta) - \Lambda_0) & =  \partial_u^j \partial_\theta^k \left(\Upsilon (u+Y(u,\theta),\theta)+ \OO\left(\frac{\nu I_0 e^{-\nu I_0}}{\log (\nu I_0)}\right)\right) \\
& =
\nu I_0 e^{-\nu I_0} \left(2 f_1 \partial_u^j\partial_\theta^k( \cos(\theta-\nu I_0 u)) +  \OO\left(\frac{1}{\log (\nu I_0)}\right)\right),
\end{aligned}
\]
where $f_1$ is given in Theorem~\ref{thm:diferenciadesolucionsdelainner}. Furthermore, if one replaces the $r_i$ in the definition of the coefficients $r_i$ of $V$ in~\eqref{def:potencialdeMorsecorrugat} by $\varepsilon r_i$, we have that
\begin{equation}
f_1 = \varepsilon\frac{\pi r_1}{4} + \OO(\varepsilon^2).
\end{equation}
\end{theorem}

\begin{proof}
By Proposition~\ref{prop:wtXqueredressa} and Theorem \ref{thm:diferenciadesolucionsdelainner}, we have that
\[
\wt \E (u,\theta) = \wtDeltaout (u+X(u,\theta),\theta) - \Upsilon (u+X(u,\theta),\theta)
\]
satisfies $\LL \wt \E  = 0$. Since $\|X\|_1 \le K(\nu I_0)^{-2}$, $\wt \E$ satisfies the same bounds as $\E$ given by Proposition~\ref{prop:fitadeladiferenciadominicomplex}. Writing $\wt \E(u,\theta) = \sum_{k\in \Z}  \wt \E_k e^{ik(\theta-\nu I_0 u )}$, 
by Lemma \ref{lem:cotesexp} we have that
\[
\begin{aligned}
|\wt \E_{\pm 1}| & \le K \frac{\nu I_0}{\log (\nu I_0)}  e^{-\nu I_0}, \\
|\wt \E_{\pm k}| & \le K \frac{\nu I_0}{\log (\nu I_0)} e^{-k \nu I_0(1-s\log (\nu I_0))/(\nu I_0)}, \qquad k \ge 2.
\end{aligned}
\]
Hence, for $(u,\theta) \in (\D_{s \log (\nu I_0)} \cap \R) \times \T_\sigma$,
\begin{equation}
\label{fita:wtE}
|\wt \E( u,\theta) - \wt \E_0 | \le K \frac{\nu I_0}{\log (\nu I_0)} e^{-\nu I_0}.
\end{equation}

Since $\|Y\|_1 \le K(\nu I_0)^{-2}$, we have that $|Y(u,\theta)| \le K(\nu I_0)^{-2}$ for $u\in \R$. Then, from the definition of $\Upsilon$ in~\eqref{def:Upsilon} and using that $f_1 \in \R$, we have that, for $(u,\theta) \in (\D_{s \log (\nu I_0)} \cap \R) \times \T_\sigma$,
\[
\begin{aligned}
\Upsilon(u,\theta) & = \nu I_0 \left(\sum_{k\ge 1}  f_k e^{-k \nu I_0} e^{i k (\theta - \nu I_0(u+Y(u,\theta)))} +
\sum_{k\le -1}  \overline{f_{-k}} e^{k \nu I_0} e^{i k (\theta - \nu I_0(u+Y(u,\theta)))} \right) \\
& = \nu I_0 e^{-\nu I_0} \left( 2 f_1 \cos(\theta-\nu I_0 u) + \OO((\nu I_0)^{-1})\right).
\end{aligned}
\]
Hence, using again $\|Y\|_1 \le K(\nu I_0)^{-2}$, $\E$ also satisfies~\eqref{fita:wtE}, for $(u,\theta) \in (\D_{s \log (\nu I_0)} \cap \R) \times \T_\sigma$, and, from $\wtDeltaout = \Upsilon + \E$, 
\[
\wtDeltaout (u,\theta) - \Lambda_0 = \nu I_0 e^{-\nu I_0} \left( 2 f_1 \cos(\theta-\nu I_0 u) + \OO((\log (\nu I_0))^{-1})\right),
\]
where $\Lambda_0 = \wt \E_0$. 
The last claim follows immediately from~\eqref{eq:relaciocoeficientsFourierUpsiloniDeltainn} and Theorem~\ref{thm:diferenciadesolucionsdelainner}.
\end{proof}

\section{Acknowledgements}

F. B has been partially supported by the grant PID2021-122711NB-C'21,
E.F. has been partially supported by the grant PID2021-125535NB-I00, and P.M. has been partially supported by the grant PID2021-123968NB-I00, funded by the
Spanish State Research Agency through the programs MCIN/AEI/10.13039/501100011033
and “ERDF A way of making Europe”. 

Also, E.F. and P.M. authors have been partially supported by the Spanish State Research Agency, through the Severo Ochoa and Mar\'ia de Maeztu Program for Centers and Units of Excellence in R\&D (CEX2020-001084-M).

\appendix

\section{Proof of Theorem~\ref{thm:lambdalemma}}
\label{sec:provadelteoremathm:lambdalemma}

We first introduce a new time in system~\eqref{def:varietatsredressades} so that the origin becomes a true saddle. Since the solutions of~\eqref{def:varietatsredressades} with initial condition $(u_0,v_0,t_0) \in V_{\rho}$ with $u_0,v_0 \ge 0$ and $u_0+v_0 >0$ satisfy $u(t)+v(t) \ge0$ while they belong to $V_{\rho}$, we define the new time~$s$ such that $dt/ds = (u+v)^{-1}$. Equation~\eqref{def:varietatsredressades} becomes
\begin{equation}
\label{eq:infinitymanifoldsstraightenedintimes}
\begin{aligned}
u' & = u (1+\OO_1(u,v)), \\
v' & = -v (1+\OO_1(u,v)), \\
t' & = (u+v)^{-1},
\end{aligned}
\end{equation}
where $\mbox{}'$ denotes $d/ds$.
The $\OO_1(u,v)$ terms depend on $s$ and are uniformly bounded in terms of $(u,v)$ in $V_{\rho} $.

Given  $w_0 = (u_0,v_0,t_0) \in V_{\rho}$, we define
\begin{equation}
\label{def:sw0}
s_{w_0} = \sup \,\{s>0\mid w(\tilde s) \in V_{\rho},\; \forall \tilde s \in [0,s) \},
\end{equation}
where  $w$ is the solution of~\eqref{eq:infinitymanifoldsstraightenedintimes} with initial condition $w_0$.

Next lemma implies (1) of Theorem~\ref{thm:lambdalemma}.
Its proof is postponed to Appendix \ref{sect:topological_lambda_lemma}.
\begin{lemma}
\label{lem:topological_lambda_lemma}
There exist $\rho\in (0,1)$ and $C>0$, satisfying $C\rho <7/8$, such that the solution $w= (u,v,t)$ of~\eqref{def:varietatsredressades} with initial condition $w_0 = (u_0,v_0,t_0) \in V_{\rho}$ with $u_0, v_0>0$ satisfies
\[
\log \left( \left( \frac{\rho}{u_0}\right)^{\frac{1}{1+C\rho}} \right) \le s_{w_0} \le \log \left( \left( \frac{\rho}{u_0}\right)^{\frac{1}{1-C\rho}} \right).
\]
Moreover, for any $0 < a \leq \rho $
and $0 < \delta < a/2$,
the ``Poincar\'e map"
\[
\Psi: \Sigma^1_{a,\delta} \to \Sigma^0_{a,\delta^{1-C a}},
\]
where the sets $\Sigma^0_{a,\delta}$ and $\Sigma^1_{a,\delta}$ are defined in \eqref{def:seccionsSigma0iSigma1}, is well defined and, if $w = (u_0,a,t_0) \in \Sigma^1_{a,\delta}$ and $\Psi(w) = (a,v_1,t_1)$, then
\begin{equation*}
\begin{aligned}
u_0^{1+C a} \le  v_1 & \le u_0^{1-C a}, \\
\wt C_1 u_0^{-(1-C a)/2} \le  t_1 -t_0 & \le \wt C_2 u_0^{-(1+C a)/2}
\end{aligned}
\end{equation*}
for some constants $\wt C_1, \wt C_2 >0$ depending only on $\rho$.
\end{lemma}

Let $w = (u,v,t)$ be a solution of~\eqref{def:varietatsredressades}  with initial condition $w_0
\in V_{\rho}$.
We introduce
\begin{equation}
\label{def:tau}
\tau = v/u
\end{equation}
and we will write $\OO_i=\OO_i(u,v)$.

We have that $0<\tau(s) <\infty$, for all $s$ such that $w\in V_{\rho}$. It is immediate from~\eqref{eq:infinitymanifoldsstraightenedintimes} that
\begin{equation}
\label{def:tauedo}
\frac{d \tau}{d s} = -(2+\OO_1) \tau.
\end{equation}
The variational equations around a solution of system~\eqref{def:varietatsredressades} are
\begin{equation}
\label{eq:variationalequationsatinfinity}
\begin{pmatrix}
\dot U \\ \dot V  \\ \dot T
\end{pmatrix}
=
\begin{pmatrix}
2u+v+\OO_2 &  u (1+\OO_1) &   u \OO_2\\
-v(1+\OO_1) & -u-2v+\OO_2   & v \OO_2\\
0 & 0 & 0
\end{pmatrix}
\begin{pmatrix}
 U \\ V  \\ T
\end{pmatrix}.
\end{equation}

To prove (2) of Theorem~\ref{thm:lambdalemma}, we will  study the behavior of the solutions of~\eqref{eq:variationalequationsatinfinity} with initial
condition $U = U_0 \neq 0$ along solutions of~\eqref{def:varietatsredressades} with initial condition
$w_0 = (u_0,v_0,t_0) \in V_{\rho} \cap \{ u , v >0\}$ for $v_0=a$ small but fixed and $u_0$ arbitrarily small.

Equations~\eqref{eq:variationalequationsatinfinity} become, in the time $s$ in which the equations in ~\eqref{eq:infinitymanifoldsstraightenedintimes} are written, and using $\tau$ in~\eqref{def:tau},
\begin{equation}
\label{eq:variationalequationsatinfinitystau}
\begin{pmatrix}
 U' \\ V'  \\  T'
\end{pmatrix}
=
\begin{pmatrix}
\frac{2+\tau+\OO_1}{1+\tau} & \frac{1+\OO_1}{1+\tau} &  u \OO_1\\
-\frac{(1+\OO_1)\tau}{1+\tau} & -\frac{1+2\tau+\OO_1}{1+\tau}   & v \OO_1\\
0 & 0 & 0
\end{pmatrix}
\begin{pmatrix}
 U \\ V \\ T
\end{pmatrix}.
\end{equation}

\begin{proposition}
\label{lem:straighteningvariationalequations} There exists $\alpha^*$, with $0 < \alpha^* < 5/12$ such that for any  $\rho>0$ small enough,  any $w = (u,v,t)$, solution of~\eqref{eq:infinitymanifoldsstraightenedintimes} with $w_{\mid s = 0} = w_0 =(u_0,v_0,t_0)\in V_{\rho} $ and any $\alpha_0^*\in [0,\alpha^*]$, there exists  $\alpha:[0,s_{w_0}] \to \R$, $C^\infty$,
where $s_{w_0}$ was defined in~\eqref{def:sw0}, with $\alpha(0) = \alpha_0^*$, such that, introducing the new variable
\[
\wt V = V + \alpha U,
\]
equation~\eqref{eq:variationalequationsatinfinitystau} becomes
\begin{equation}
\label{eq:variationalequationsatinfinitystausimplifyied}
\begin{pmatrix}
 U' \\ \wt V' \\  T'
\end{pmatrix}
=
\begin{pmatrix}
\frac{2+\tau+\OO_1}{1+\tau} - \alpha\frac{1+\OO_1}{1+\tau} & \frac{1+\OO_1}{1+\tau}  &  u \OO_1\\
0 & -\frac{1+2\tau+\OO_1}{1+\tau} + \alpha\frac{1+\OO_1}{1+\tau} &    v\OO_1\\
0 & 0 & 0
\end{pmatrix}
\begin{pmatrix}
 U \\ \wt V  \\ T
\end{pmatrix}.
\end{equation}
Furthermore, for $s \in (0,s_{w_0}]$,
\begin{equation}
\label{bound:alphachange}
0 < \alpha(s) < \frac{55}{128}\frac{\tau(s)}{1+\tau(s)}.
\end{equation}
\end{proposition}
\begin{proof} Given $\alpha$ and $\wt V = V + \alpha U$, since $\tau >0$, the equation for $\wt V$ is
\begin{align} \label{eq:tildePprime}
\wt V'  = & \left( -\frac{(1+\OO_1)\tau}{1+\tau} + \alpha' + (3+\OO_1)  \alpha - \alpha^2 \frac{1+\OO_1}{1+\tau}\right) U \\
& + \left( -\frac{1+2\tau+\OO_1}{1+\tau} + \alpha\frac{1+\OO_1}{1+\tau}\right) \wt V
 + (v+\alpha u)T \OO_1 .
\end{align}
The claim will follow finding an appropriate solution of
\begin{equation}
\label{eq:alphavariational}
 \alpha'   = \nu_0+ \nu_1 \alpha +\nu_2 \alpha^2 ,
\end{equation}
where
\begin{equation*}
\nu_0 = \frac{(1+\OO_1)\tau}{1+\tau}, \qquad \nu_1 (s) = -3+\OO_1, \qquad \nu_2 = \frac{1+\OO_1}{1+\tau}.
\end{equation*}
Let $f(w,\alpha) = \nu_0+ \nu_1 \alpha +\nu_2 \alpha^2$ be the right hand side of~\eqref{eq:alphavariational}, where we have omitted the dependence of $\nu_i$, $i=1,2,3$, on $w$ and $s$.
We introduce $\alpha_0$ and $\alpha_1$, the nullclines of~\eqref{eq:alphavariational},  by
\[
f(w,\alpha) = \nu_2 (\alpha-\alpha_0(\tau))(\alpha-\alpha_1(\tau)),
\]
and $R$, where
\begin{equation*}
\begin{aligned}
\alpha_0 (\tau) & = -\frac{\nu_1}{2 \nu_2}\left( 1-\left(1-4 \frac{\nu_0 \nu_2}{\nu_1^2}\right)^{1/2}\right) \\
& = \left(\frac{3}{2} +\OO_1\right)\left(1+\tau-\left((1+\tau)^2-\left(\frac{4}{9}+\OO_1 \right)
\tau \right)^{1/2}\right)
\\
& = \left(\frac{3}{2} +\OO_1\right)\left(1+\tau-\sqrt{R(\tau)}\right).
\end{aligned}
\end{equation*}

To complete the proof of Proposition \ref{lem:straighteningvariationalequations}, we need the following two auxiliary lemmas.
\begin{lemma}
\label{lem:alpha0properties}
The function $\alpha_0$ has the following properties.
For $(u,v) \in V_a$ (that is, $0< \tau < \infty$),
\begin{enumerate}
\item[(1)] $2\sqrt{2}/3+\OO_1 \le \sqrt{R(\tau)}/(1+\tau) <1$,
\item[(2)] $\lim_{\tau\to \infty} \alpha_0(\tau) = 1/3+\OO_1$,
\item[(3)] $\lim_{\tau\to 0} \alpha_0(\tau)/\tau = 1/3+\OO_1$,
\item[(4)]
\[
\frac{d}{ds} \alpha_0 = -(1+\OO_1) \frac{\sqrt{R}-\tau+1+\OO_1}{\sqrt{R}} \alpha_0,
\]
\item[(5)]
\[
- \frac{(2+\OO_1)}{\sqrt{R}}\alpha_0  \le \frac{d}{ds} \alpha_0 \le
- \frac{(16/9+\OO_1)}{\sqrt{R}}\alpha_0
\]
\item[(6)] $\lim_{\tau\to 0} (d\alpha_0/ds)/\alpha_0 = -2+\OO_1$.
\end{enumerate}
Furthermore,
\begin{equation}
\label{bound:alpha0bytau}
0 < \alpha_0(\tau) < \frac{11}{32}\frac{\tau}{1+\tau}.
\end{equation}
\end{lemma}

\begin{proof}
Items (1) to (6) are proven in~\cite{GorodetskiK12}.
The bound~\eqref{bound:alpha0bytau} follows from a direct computation.
\end{proof}

Next lemma provides solutions of~\eqref{eq:alphavariational} close to the nullcline~$\alpha_0$. The proofs of the next two lemmas are given in Appendix \ref{sect:topological_lambda_lemma}.

\begin{lemma}
\label{lem:attractingnullcline}
For any $0 < \rho < 1$ small enough,  the following is true.
For any solution $w = (u,v,t)$ of~\eqref{def:varietatsredressades} with initial condition $w_0
\in V_{\rho}  \times \T$,
if $\alpha$ is a solution of~\eqref{eq:alphavariational} with $0 \le \alpha(s_0) \le 5 \alpha_0(\tau(s_0))/4$ for some $0 < s_0 < s_{w_0}$, then $0 < \alpha(s) < 5\alpha_0(\tau(s))/4$ for all $s \in [s_0, s_{w_0}]$.
\end{lemma}

Let $\alpha$ be any solution of~\eqref{eq:alphavariational} with $\alpha(0) \in [0,5\alpha_0(\tau(0))/4]$, which, by the definition of $\tau$ and (2) of Lemma~\ref{lem:alpha0properties}, is a nonempty interval (recall that $\tau(0)\gg 1$).
 By Lemma~\ref{lem:attractingnullcline}, $\alpha$ is well defined for $s\in [0,s_{w_0}]$ and
$0 < \alpha(s) < 5\alpha_0(\tau(s))/4$.  Then, bound~\eqref{bound:alpha0bytau} implies~\eqref{bound:alphachange}. System~\eqref{eq:variationalequationsatinfinitystausimplifyied} is obtained by a straightforward computation. Observe that, by~\eqref{bound:alpha0bytau}, the terms $(v+\alpha u)\OO_1$ in~\eqref{eq:tildePprime} are indeed $v\OO_1$.
This finishes the proof of Proposition \ref{lem:straighteningvariationalequations}.
\end{proof}

Below, $\|\cdot\|$ will denote the $\sup$-norm.

\begin{lemma}
\label{lem:solvingstraightenedvariationalequations}
Let $w = (u,v,t)$ be a solution of~\eqref{def:varietatsredressades} with initial condition, $w_0
\in \Sigma^0_{a,\delta}$, at $s=0$. Let $\tilde s_{w_0}$ be such that $w(\tilde s_{w_0}) \in \Sigma^0_{a,\delta^{1-Ca}}$. Let $W=(U,\wt V,T)$ be a solution of~\eqref{eq:variationalequationsatinfinitystausimplifyied} with initial condition  $W_0=(U_0,\wt V_0,T_0)$ at $s=0$. Then
\begin{enumerate}
\item[(1)]
For $s = \tilde  s_{w_0}$,
\[
|\wt V(\tilde s_{w_0})| \le  C u_0^{\frac{9}{10}+\OO_1(\rho)}(|\wt V_0|+ \OO_1  |T_0|).
\]
\item[(2)]
Assume that $W_0$ satisfies $U_0 \neq 0$. Then, there exists $\delta$ such that for any $w_0 \in \Sigma^0_{a,\delta}$,
\[
|U(\tilde s_{w_0})| \ge C \left(|U_0|- C u_0^{\frac{4}{5}+\OO_1(\rho)} \|W_0\|\right) u_0^{-(\frac{9}{10}+\OO_1(\rho))}.
\]
\end{enumerate}
\end{lemma}

\begin{proof}[Proof of Theorem~\ref{thm:lambdalemma}]
Lemma~\ref{lem:topological_lambda_lemma} proves that the transition map $\Psi:\Sigma^1_{a,\delta} \to \Sigma^0_{a,\delta^{1-C\rho}}$ is well defined and implies the estimates of (1) of Theorem~\ref{thm:lambdalemma}.

Given the interval $I=[0,\delta]$, let $\gamma(\sigma) = (\sigma,a,t_0(\sigma))$, $\sigma \in I$, be a $C^1$ curve in $\Sigma^1_{a,\delta}$ with
$0 < \sigma < \delta$,
and $\tilde \gamma = \Psi \circ \gamma = (a,v_1,t_1)$,
which is well defined if $\delta$ is small enough.

Let us compute $\tilde \gamma'(\sigma)$. Let $X = (X_u,X_v,X_t)$ denote the vector field in~\eqref{def:varietatsredressades} and $w=(u,v,t)$. Since
$\tilde \gamma(\sigma) = \phi_{t_1(\sigma)-t_0(\sigma)}(\gamma(\sigma))$, where $\phi_t$ denotes the flow of $X$,
we have that
\begin{equation}
\label{eq:gammatildeprime}
\begin{aligned}
\tilde \gamma'(\sigma) =
\begin{pmatrix}
0 \\ v_1'(\sigma)  \\ t_1'(\sigma)
\end{pmatrix}
& =
D_w \phi_{t_1(\sigma)-t_0(\sigma)}(\gamma(\sigma))\gamma'(\sigma)+ X(\tilde \gamma(\sigma)) (t_1'(\sigma)-t_0'(\sigma)) \\
& =
\begin{pmatrix}
U_{\mid t_1(\sigma)-t_0(\sigma)} + X_u(\tilde \gamma(\sigma))(t_1'(\sigma)-t_0'(\sigma)) \\
V_{\mid t_1(\sigma)-t_0(\sigma)} + X_v(\tilde \gamma(\sigma))(t_1'(\sigma)-t_0'(\sigma)) \\
T_{\mid t_1(\sigma)-t_0(\sigma)} + X_t(\tilde \gamma(\sigma))(t_1'(\sigma)-t_0'(\sigma))
\end{pmatrix} ,
\end{aligned}
\end{equation}
where $(U,V,T)$ is the solution of~\eqref{eq:variationalequationsatinfinity}
along $\phi_{t-t_0(\sigma)}(\gamma(\sigma))$ with initial condition $\gamma'(\sigma)$.

From the first component of~\eqref{eq:gammatildeprime},
\begin{equation}
\label{eq:t1pmt0p}
t_1'(\sigma)-t_0'(\sigma) = -\frac{U_{\mid t_1(\sigma)-t_0(\sigma)}}{X_u(\tilde \gamma(\sigma))}.
\end{equation}
We observe that $X_u(\tilde \gamma(\sigma)) = a(a+\OO(\sigma^{1-\OO_1(a)})+\OO_2(a,\sigma^{1-\OO_1(a)}))>a^2/2$.

We choose $\alpha$ in Lemma~\ref{lem:attractingnullcline} such that $\alpha(0) = 0$.
We apply the change of variables
of Proposition~\ref{lem:straighteningvariationalequations} and consider $(U,\tilde V,T)$, the corresponding solution
of~\eqref{eq:variationalequationsatinfinitystausimplifyied}.
By the choice of $\alpha$,
$(U,V,T)_{\mid s=0} = (U,\wt V,T)_{\mid s=0}$.

Now we prove (2) of Theorem~\ref{thm:lambdalemma}.
Let $W_\sigma^0 = (U_\sigma^0,V_\sigma^0,T_\sigma^0)= \gamma'(\sigma)$ and
 $(U_\sigma,V_\sigma,T_\sigma)$ be the solution of~\eqref{eq:variationalequationsatinfinitystau} with initial condition $W_\sigma^0$ and $(U_\sigma,\wt V_\sigma,T_\sigma)$,
the solution of~\eqref{eq:variationalequationsatinfinitystausimplifyied} with the same initial condition.
If $\delta$ is small enough, $\sup_{0<\sigma<\delta} \|\gamma'(\sigma)\| = \sup_{0<\sigma<\delta} \|W_\sigma^0\|  < 2 \|W_0^0\|$.
Hence, if $\delta$ is small enough, by (2) of Lemma~\ref{lem:solvingstraightenedvariationalequations},
\[
|U_\sigma(\tilde s_{w_0})|  \ge C \left(|U_\sigma^0|- C \sigma^{\frac{4}{5}+\OO_1(\rho)} \|W_0^0\|\right) \sigma^{-(\frac{9}{10}+\OO_1(\rho))}
\ge \wt C |U_\sigma^0| \sigma^{-(\frac{9}{10}+\OO_1(\rho))}.
\]
In the case we are considering, $U_\sigma^0 = 1$.
This inequality, combined with~\eqref{eq:t1pmt0p}, implies
\begin{equation*}
|t_1'(\sigma)-t_0'(\sigma)| \ge C  \sigma^{-(\frac{9}{10}+\OO_1(\rho))}.
\end{equation*}
Hence, by (2) of Lemma~\ref{lem:solvingstraightenedvariationalequations}, the bound of $\alpha$ given by Lemma~\ref{lem:attractingnullcline}, bound~\eqref{bound:alpha0bytau}, \eqref{eq:t1pmt0p} and the facts that $T'=0$, $T=t_0'(\sigma)$,
$|X_v(\tilde \gamma(\sigma))| \le C u(\sigma)^{1-Ca}$  and $|X_u(\tilde \gamma(\sigma))| \le C $,
\[
\begin{aligned}
\left| \frac{v_1'(\sigma)}{t_1'(\sigma)}\right| & =
\frac{|V_\sigma (\tilde s_{\gamma(\sigma)})+ X_v(\tilde \gamma(\sigma))(t_1'(\sigma)-t_0'(\sigma)|}{|t_1'(\sigma)|}
\\
& \le
\frac{|V_\sigma (\tilde s_{\gamma(\sigma)})+ X_v(\tilde \gamma(\sigma))(t_1'(\sigma)-t_0'(\sigma)|}{|t_1'(\sigma)-t_0'(\sigma))|}
\left(1+ \frac{|t_0'(\sigma)|}{|t_1'(\sigma)|}\right)\\
& \le C
\frac{|\wt V_\sigma (\tilde s_{\gamma(\sigma)})-\alpha (\tilde s_{\gamma(\sigma)}) U_\sigma (\tilde s_{\gamma(\sigma)})+ X_v(\tilde \gamma(\sigma))(t_1'(\sigma)-t_0'(\sigma))|}{|t_1'(\sigma)-t_0'(\sigma)|}  \\
& =
C \frac{|\wt V_\sigma (\tilde s_{\gamma(\sigma)})+[(\alpha (\tilde s_{\gamma(\sigma)}) X_u(\tilde \gamma(\sigma))+ X_v(\tilde \gamma(\sigma))](t_1'(\sigma)-t_0'(\sigma))|}{|t_1'(\sigma)-t_0'(\sigma)|}  \\
& =
C \left(\frac{|\wt V_\sigma (\tilde s_{\gamma(\sigma)})|}{|t_1'(\sigma)-t_0'(\sigma)|} +
\left|(\alpha (\tilde s_{\gamma(\sigma)}) X_u(\tilde \gamma(\sigma))+ X_v(\tilde \gamma(\sigma))\right|\right) \\
& \le \wt C \sigma^{1-Ca},
\end{aligned}
\]

which proves (2) of Theorem~\ref{thm:lambdalemma}.

\end{proof}

\section{Proofs of claims in Appendix~\ref{sec:provadelteoremathm:lambdalemma}}
\label{sect:topological_lambda_lemma}

\subsection{Proof of Lemma~\ref{lem:topological_lambda_lemma}}	
\label{sec:apendix_proves_lemes_lambda_lema}

\begin{proof}
	Given $w_0 = (u_0,v_0,t_0) \in V_{\rho}$, let $w= (u,v,t)$ be the solution of~\eqref{def:varietatsredressades} with initial condition $w_0$ at $s=0$. Since $u_0,v_0 >0$, while $w \in V_{\rho} $, $u(s),v(s) >0$. Hence, there exists $C_0>0$, independent on $\rho$ such that
	\begin{equation}
		\label{bound:infinitymanifoldsstraightenedintimes}
		\begin{aligned}
			(1-C \rho)u & \le u'  \le   (1+C \rho)u, \\
			-(1+C\rho)v & \le v'  \le  -(1-C\rho)v .
		\end{aligned}
	\end{equation}
	Since $v(s)$ is decreasing and $\{v=0\}$ is invariant, $w(s)$ can only leave $V_{\rho} $ through $\{u=\rho\}$. From~\eqref{bound:infinitymanifoldsstraightenedintimes} we have that for all $s < s_{w_0}$
	\begin{equation}
		\label{bound:solutions_integrated_qp}
		\begin{aligned}
			u_0 e^{(1-C \rho)s} & \le u(s)   \le   u_0 e^{(1+C \rho)s}, \\
			v_0 e^{-(1+C \rho)s} & \le v(s)  \le  v_0 e^{-(1-C \rho)s}.
		\end{aligned}
	\end{equation}
	In particular, the time $s_{u_0,u}$ to reach $u$ from $u_0$ is bounded by
	\begin{equation}
		\label{bound:timesfromq0toq}
		\log \left( \frac{u}{u_0}\right)^{\frac{1}{1+C\rho}} \le s_{u_0,u} \le \log \left( \frac{u}{u_0}\right)^{\frac{1}{1-C\rho}}.
	\end{equation}
	Moreover, since $v(s)$ is decreasing and $u(s)$ increasing, the solution leaves $V_{\rho}$ through $\{u=\rr\}$.
	
	Hence, if $\rho$ is small enough, the solution through $w_0$ remains in $V_{\rho}$
	for all $s$ such that
	\[
	0 < s<s_{w_0} < \log  \left( \frac{\rho}{u_0}\right)^{\frac{1}{1-C\rho}}.
	\]
	
	In the same way
	$$
	s_{w_0} \ge  \log  \left( \frac{\rho}{u_0}\right)^{\frac{1}{1+C\rho}}.
	$$

	Restricting the domain from $V_\rho$ to $V_a$ we have the same estimates, changing $\rho $ by $a$.
	Consequently, for any $0 < \delta < a/2 < \rho$, if $w_0 \in \Sigma^1_{a,\delta}$, the solution through $w_0$
	satisfies $u = a$ at a time $s_*$ bounded by
	\begin{equation*}
		\log  \left( \frac{a}{u_0}\right)^{\frac{1}{1+C a}} \le s_* \le \log  \left( \frac{a}{u_0}\right)^{\frac{1}{1-C a}}.
	\end{equation*}
	From~\eqref{bound:solutions_integrated_qp} with an analogous argument, if $a$ is small
	\[
	u_0^{1+3C a} \le v_0 \left( \frac{a}{u_0}\right)^{\frac{-1-C a}{1-C a}} \le v_0 e^{-(1+ C a)s_*} \le v(s_*) \le v_0 e^{-(1-C a)s_*} \le v_0 \left( \frac{a}{u_0}\right)^{\frac{-1+C a}{1+C a}}
	\le  u_0^{1-2C a}.
	\]
	It remains to estimate $t(s_*)-t_0$. Since
	\[
	t(s_*)-t_0 = \int_0^{s_*} \frac{1}{u(s)+v(s)}ds,
	\]
	$0 < u_0 < \delta$, $v_0 = a$ and $2 \delta < a$, we have that
	\[
	\begin{aligned}
		t(s_*)-t_0 & \ge \int_0^{s_*} \frac{1}{u_0 e^{(1+C a)s}+ v_0 e^{-(1-C a)s}}ds \\
		& = \int_0^{s_*} \frac{e^{- C a s}}{u_0 e^{s}+ v_0 e^{-s}}ds \\
		& \ge \left( \frac{u_0}{a}\right)^{\frac{C a }{1-C a}} \int_0^{s_*} \frac{1}{u_0 e^{s}+ v_0 e^{-s}}ds \\
		& = \left( \frac{u_0}{a}\right)^{\frac{C a }{1-C a}} \frac{1}{(u_0 v_0)^{1/2}} \int_{-\log(v_0/u_0)^{1/2}}^{s_*-\log(v_0/u_0)^{1/2}}
		\frac{1}{e^{\xi}+  e^{-\xi}}d\xi \\
		& \ge \frac{1}{a^{1/2}} \frac{1}{u_0^{(1-3C a)/2}} \int_{-\log 2/2}^{\log 2 /4}
		\frac{1}{ e^{\xi}+  e^{-\xi}}d\xi .
	\end{aligned}
	\]
	With an analogous argument one obtains the upper bound of $t(s_*)-t_0$.
	We write again $C$ for $3C$.
\end{proof}

\subsection{Proof of Lemma~\ref{lem:attractingnullcline}}

\begin{proof}[Proof of Lemma~\ref{lem:attractingnullcline}]
	We only need to prove that $\alpha$ satisfies
	\[
	\mathrm{(i)} \; \;  \frac{d \alpha}{ds}_{\mid \tiny
		\alpha  = 0 }
	>  0,   \qquad
	\mathrm{(ii)} \; \;
	\frac{d \alpha}{ds}_{\mid \tiny
		\alpha  = 5 \alpha_0/4}  <  \frac{5}{4}\alpha_0.
	\]
	
 (i) follows from
	\[
	\frac{d \alpha}{ds}_{\mid \tiny
		\alpha  = 0 } = \nu_0 >0.
	\]

	Now we prove (ii). Using (5) of Lemma~\ref{lem:alpha0properties}, (ii) is implied by
	\[
		\frac{d \alpha}{ds}_{\mid \tiny
			\alpha  = 5 \alpha_0/4 }
		 =  \frac{1+\OO_1}{1+\tau} \frac{\alpha_0}{4} \left(\frac{5}{4}\alpha_0-\alpha_1\right) 
		 < - \frac{5}{4}\frac{2+\OO_1}{\sqrt{R}} \alpha_0,
	\]
	which, since $\alpha_0>0$, is equivalent to
	\begin{equation}
		\label{ineq:impliesiin}
		\frac{1+\OO_1}{1+\tau}\left(\frac{5}{4}\alpha_0-\alpha_1\right)
		<  -\frac{5}{2}\frac{1+\OO_1}{\sqrt{R}},
	\end{equation}
	for $0<\tau$. Taking into account the definition of $\alpha_0$,
	\eqref{ineq:impliesiin} is equivalent to
	\[
	\left(\frac{15}{4}+\OO_1\right) R-(3+\OO_1) (1+\tau) > \left(\frac{3}{4}+\OO_1\right) (1+\tau) \sqrt{R},
	\]
	for $0<\tau$, which in turn is equivalent to
	\[
	\left(\left(\frac{27}{8}+\OO_1\right) R-\left(\frac{5}{2}+\OO_1\right) (1+\tau)\right)^2-
	\left(\frac{9}{64}+\OO_1\right)^2 (1+\tau)^2 R >0.
	\]
	If we disregard the $\OO_1$ terms, which are small if $\rho$ is small, the above inequality simply reads
	\[
	\frac{45}{4} \tau^4 +\frac{289}{16} \tau^3 + \frac{51}{4} \tau^2 + \frac{69}{16} \tau + \frac{5}{8}
	>0.
	\]
	But the above inequality holds, since $\tau>0$.
\end{proof}

\subsection{Proof of Lemma~\ref{lem:solvingstraightenedvariationalequations}}

\begin{proof}

	We recall that, by~\eqref{bound:timesfromq0toq}, the time $\tilde s_{w_0}$  is
	bounded from above by
	\begin{equation}
		\label{bound:tildesw0}
		\tilde s_{w_0} \le \log \left( \frac{a}{u_0}\right)^{\frac{1}{1-C\rho}}.
	\end{equation}
	
	Since $T' = 0$, for $T=T_0$ the equation for $\wt V$ is
	\[
	\wt V ' = A \wt V + v T_0 \OO_1 ,
	\]
	where
	\[
	A = -\frac{1+2\tau+\OO_1}{1+\tau} + \alpha\frac{1+\OO_1}{1+\tau}.
	\]
	Using the bounds on $\alpha$ given by Lemma~\ref{lem:attractingnullcline} and~\eqref{bound:alphachange},  a straightforward computation shows that, if $\rho$ is small enough, we have  $A < -9/10$ for all $\tau>0$.
	Hence, using again~\eqref{bound:solutions_integrated_qp},
	$|\wt V(s) | \le  (|\wt V_0|+  |T_0| \OO_1 ) e^{-\frac{9}{10}s}$, which, taking into account the bound of $\tilde s_{w_0}$, implies (1).
	
	To prove (2), observe that the equation for $U$ is
	\[
	U' = \wt A U + \wt B \wt V + u T \OO_1 ,
	\]
	where
	\[
	\wt A = \frac{2+\tau+\OO_1}{1+\tau} - \alpha\frac{1+\OO_1}{1+\tau}, \qquad \wt B= \frac{1+\OO_1}{1+\tau}.
	\]
	Again, a straightforward computation shows that, if $\rho$ is small enough, $\wt A > 9/10$.
	Defining $f(s) = \exp \int_0^s \wt A (\xi)\,d\xi$, we have that
	\begin{equation}
		\label{eq:Qs}
		U(s) = f (s) \left[U_0 + \int_0^s f^{-1}(\xi)(\wt B(\xi) \wt V (\xi)+ u (\xi)T_0 \OO_1 )\,d\xi  \right].
	\end{equation}
	We bound the terms in the integral in the following way. First we observe that, using~\eqref{def:tauedo},
	\[
	0 < \wt B(\xi) = \frac{1+\OO_1}{1+\tau} \le \frac{u_0}{v_0} \frac{1+\OO_1}{\frac{u_0}{v_0}+e^{-(2+\OO_1(\rho))s}}
	< \frac{u_0}{v_0} (1+\OO_1) e^{(2+\OO_1(\rho))s}.
	\]
	Hence, by the previous bound on $\wt V$, for some constant $K>0$,
	\[
	\begin{aligned}
		\left| \int_0^s f^{-1}(\xi)\wt B(\xi) \wt V (\xi)\,d\xi \right|
		& \le (|\wt V_0|+ \OO_1  |T_0|) (1+\OO_1) \frac{u_0}{v_0}\int_0^s e^{(\frac{1}{5}+\OO_1(\rho))\xi} \,d\xi \\
		& \le K (|\wt V_0|+ \OO_1  |T_0|)  u_0 e^{(\frac{1}{5}+\OO_1(\rho))s}.
	\end{aligned}
	\]
	Using~\eqref{bound:solutions_integrated_qp} to bound $u(s)$, the other term in the integral can be bounded as follows:
	\[
	\begin{aligned}
		\left| \int_0^s f^{-1}(\xi)u (\xi)\OO_1 T_0\,d\xi \right|
		& \le \OO_1(\rho)|T_0| u_0 \int_0^s e^{(\frac{1}{10}+\OO_1({\rho}))\xi} \,d\xi \\
		& \le \OO_1(\rho)|T_0| u_0  e^{(\frac{1}{10}+\OO_1({\rho}))s}.
	\end{aligned}
	\]
	That is, since $0 < s < \tilde s_{w_0} $, using~\eqref{bound:tildesw0},
	\[
	\left|U_0 + \int_0^s f^{-1}(\xi)(\wt B(\xi) \wt V (\xi)+ u (\xi)T_0 \OO_1 )\,d\xi \right|
	\ge |U_0|- K q_0^{(\frac{1}{5}+\OO_1(\rho))s} \|W_0\|.
	\]
	Finally, since $0 < u_0 < \delta$, substituting this bound into~\eqref{eq:Qs} and evaluating it at $s = \tilde s_{w_0} $, we obtain (2).

\end{proof}

\bibliography{references}
\bibliographystyle{alpha}

\end{document}